\documentclass[11pt]{article}

\usepackage{macros}

\title{Strong Algebras and Radical Sylvester-Gallai Configurations}

\author{
Rafael Oliveira \thanks{Cheriton School of Computer Science, University of Waterloo,
\email{rafael@uwaterloo.ca}}
\and
Akash Kumar Sengupta \thanks{Department of Pure Mathematics, University of Waterloo, \email{aksengup@uwaterloo.ca} }
}
\date{}

\renewcommand\footnotemark{}

\begin{document}

\maketitle

\begin{abstract}
    In this paper, we prove the following non-linear generalization of the classical Sylvester-Gallai theorem. 
    Let $\K$ be an algebraically closed field of characteristic $0$ and $\cF=\{F_1,\cdots,F_m\} \subset \K[x_1,\cdots,x_N]$ be a set of irreducible homogeneous polynomials of degree at most $d$ such that $F_i$ is not a scalar multiple of $F_j$ for $i\neq j$. 
    Suppose that for any two distinct $F_i,F_j\in \cF$, there is  $k\neq i,j$ such that $F_k\in \rad(F_i,F_j)$. 
    We prove that such radical SG configurations must be low dimensional. 
    More precisely, we show that there exists a function $\lambda : \N \to \N$, independent of $\K,N$ and $m$, such that any such configuration $\cF$ must satisfy
    $$ \dim (\Kspan{\cF}) \leq \lambda(d). $$
    
    Our result confirms a conjecture of Gupta \cite[Conjecture 2]{G14} and generalizes the quadratic and cubic Sylvester-Gallai theorems of \cite{S20,OS22}. 
    Our result takes us one step closer towards the first deterministic polynomial time algorithm for the Polynomial Identity Testing (PIT) problem for depth-4 circuits of bounded top and bottom fanins. Our result, when combined with the Stillman uniformity type results of \cite{AH20a,DLL19,ESS21}, yields uniform bounds for several algebraic invariants such as projective dimension, Betti numbers and Castelnuovo-Mumford regularity of ideals generated by radical SG configurations.
\end{abstract}

\newpage
\tableofcontents

\newpage

\section{Introduction}\label{sec:intro}

In 1893, Sylvester posed a fundamental question in combinatorial geometry (\cite{S93}): given a finite set of distinct points $\{ v_1, \ldots, v_m \} \subset \R^2$ such that the line defined by any pair of distinct points $v_i, v_j$ contains a third point $v_k$ in the set, must all points in the set be collinear?
This question was answered positively and independently by Melchior and Gallai \cite{M40, G44}, and is now known as the Sylvester-Gallai (SG) theorem.

Since then, mathematicians and computer scientists have given considerable attention to generalizations of Sylvester's question \cite{EK66, H66, S66, H83, K86, EPS06, PS09, BDWY11, GT13, DSW14}, and results such as the above are known as Sylvester-Gallai (SG) type theorems.
In its most general form, a Sylvester-Gallai (SG) configuration is a finite set of geometric objects which satisfy certain local conditions or dependencies.
For instance, in the original question, the geometric objects are  points and the local dependencies are collinear dependencies amongst triples from the set.
The main theme underlying Sylvester-Gallai type problems is the following local-to-global phenomenon: must these local constraints on the geometric objects imply a global constraint on such configurations?
More specifically, the main question of concern is:

\begin{center}
    Are Sylvester-Gallai type configurations always low-dimensional?
\end{center}

\noindent The answer to this question has turned out to be affirmative for various Sylvester-Gallai type problems with linear algebraic constraints. 
The above works prove beautiful generalizations of the classical SG theorem such as the colored, robust and higher dimensional linear SG  theorems. 
For a comprehensive survey on earlier works, we refer the reader to \cite{BM90}.

Linear Sylvester-Gallai type theorems have been successfully applied to obtain important consequences in diverse subfields of theoretical computer science.
In algebraic complexity theory, linear Sylvester-Gallai theorems were used in algorithms for Polynomial Identity Testing (PIT) and reconstruction of depth-3 circuits of bounded top fanin \cite{DS07, KS09, SS13, S16}. 
In coding theory, linear Sylvester-Gallai theorems were used to prove non-existence of 2-query Locally Correctable Codes over fields of characteristic zero \cite{BDWY11}.

In this paper we prove a non-linear generalization of the SG theorem, which is at the intersection of algebraic geometry, commutative algebra, algebraic complexity theory and combinatorics. 
Motivated by the PIT problem for depth-4 circuits, we consider Sylvester-Gallai configurations where the geometric objects are hypersurfaces in a projective space, or equivalently homogeneous polynomials in a polynomial ring $\K[x_1,\cdots,x_N]$. 
The local constraint is  given by radical dependence of triples of polynomials, which implies the following geometric constraint: 
for any two distinct hypersurfaces $Y_i,Y_j \subset \P^{N-1}$ in the configuration, there is a third hypersurface $Y_k$ in the configuration such that $Y_i\cap Y_j\subset Y_k$. 

We now formally describe such SG configurations. 
Henceforth, we adopt the standard notation and use the term \emph{form} to refer to a homogeneous polynomial. 
We also let $\K$ be an algebraically closed field of characteristic $0$ and denote our polynomial ring by $S := \K[x_1,\ldots, x_N]$.

\begin{definition}[Radical Sylvester-Gallai configuration]\label{definition: radical SG intro}
Let $d$ be a positive integer. 
We say that a finite set $\cF = \{F_1, \ldots, F_m\} \subset S$ of irreducible forms of degree $\leq d$ is a $\rsg{d}$ configuration if the following conditions hold:
\begin{enumerate}
    \item $F_i \not\in (F_j)$ for any $i \neq j \in [m]$, \hfill (pairwise non-associate forms)
    \item for every pair $F_i, F_j$, there is $k \neq i,j$ such that $F_k \in \rad(F_i, F_j)$. \hfill (radical SG condition)
\end{enumerate}
\end{definition}

Note that if $d = 1$, the definition above specializes to the dual of the linear Sylvester-Gallai condition for points in $\P^{N-1}$. 
In this case, condition (1) is equivalent to the linear forms being pairwise linearly independent. 
For examples of linear and higher degree radical Sylvester-Gallai configurations, we refer to \cref{subsection: examples}.

Our main result is to prove that such SG configurations must be ``low dimensional.'' 
More formally, we prove the following theorem.

\begin{restatable}[Radical Sylvester-Gallai Theorem]{theorem}{radSGmain}
\label{theorem: sylvester-gallai}
There exists an increasing function $\lambda : \N \to \N$ such that
if $\cF \subset \K[x_1, \ldots, x_N]$ is a $\rsg{d}$ configuration, then we have 
$$\dim (\Kspan{\cF}) \leq \lambda(d).$$
That is, the dimension of the $\K$-linear span of $\cF$ is bounded above by a function only of $d$, independent of $\K$, $N$ and the number of polynomials $m$. 
\end{restatable}

\cref{theorem: sylvester-gallai} settles a conjecture of Gupta \cite[Conjecture 2]{G14}, which is a fundamental first step towards completing Gupta's program to solve the PIT problem for depth-4 circuits with bounded top and bottom fanins. 
Our result generalizes the results of \cite{S20, OS22}, which proved Gupta's conjecture for $d=2,3$ respectively. 
For $d=1$, the statement is equivalent to the linear SG theorem over the field $\K$, which is one of the ingredients of our proof. 
\cref{theorem: sylvester-gallai} is a consequence of our main technical result (see \cref{theorem: radical SG theorem over UFD}), which works in a more general setting of radical SG-configurations of reducible forms in graded unique factorization domains which are quotients of polynomial rings by \emph{strong} sequences.

Our perspective highlights the \emph{geometric} aspects of Sylvester-Gallai problems. 
When viewed through the \emph{algebraic and computational} lens (via the algebra-geometry dictionary), one realizes that these problems are intrinsically related to the study of \emph{cancellations or syzygies} in algebraic problems. 
In algebraic and computational terms, the Sylvester-Gallai question can be rephrased as:
\begin{center}
    Must Sylvester-Gallai configurations depend only on ``few variables''?
\end{center}

\noindent Thus, from an algebraic point of view, the radical Sylvester-Gallai problem is a question about bounding the syzygies of the configuration of polynomials. 
A key observation is that Sylvester-Gallai problems can be viewed in light of the commutative algebraic principle of \emph{Stillman uniformity}. 
One way to quantify the complexity of a set of polynomials $\{F_1,\cdots,F_m\}\subset \K[x_1,\cdots,x_N]$, is to measure various commutative algebraic invariants of the ideal $I=(F_1,\cdots, F_m)$, such as the degree, projective dimension and Betti numbers, among others. 
In general, polynomials in $N$ variables will have greater complexity as $N$ increases. 
However, the following principle of so-called Stillman uniformity \cite{ESS19a} emerges from several recent developments in commutative algebra. 
If the degree and the number of polynomials is held fixed, then the complexity of polynomials remains bounded even if $N$ increases, because they depend on few ``variables'' \cite{AH20a, ESS19b, DLL19}. 
A model case of this phenomenon is the breakthrough work of \cite{AH20a}, which proves Stillman's conjecture by constructing \emph{strong subalgebras} of polynomial rings. 

Even though the Sylvester-Gallai question above has a striking similarity with Stillman uniformity, a major difference is that in the commutative algebraic works mentioned above, the bounds necessarily depend on the number of polynomials defining the ideal. 
On the other hand, in Sylvester-Gallai type problems the bounds must be independent of the number of polynomials. 
Also, note that \cref{theorem: sylvester-gallai} provides an uniform bound (independent of $\K$, $N$ and $m$) on the dimension of the $\K$-linear span of $F_1,\cdots,F_m$, which is much stronger than any Stillman uniformity type bound on other algebraic invariants. 
Indeed, we obtain the following consequence of \cref{theorem: sylvester-gallai}, which provides uniform bounds for several algebraic invariants of a radical Sylvester-Gallai configuration, via Stillman uniformity. 

\begin{restatable}[Stillman Bounds for Radical Sylvester-Gallai Configurations]{corollary}{stillmanSG}\label{corollary: stillman type bounds}
There exist functions $\lambda_{\deg},\lambda_{\mathrm{reg}},\lambda_{pdim}, \lambda_{\beta_{i,j}}:\N\rightarrow \N$,  such that the following holds.

If $\cF = \{F_1, \ldots, F_m\} \subset \K[x_1, \ldots, x_N]$ is a $\rsg{d}$ configuration, and $I=(F_1,\cdots,F_m)$, then the following invariants of $I$ are uniformly bounded by functions independent of $\K$, $N$ and $m$.

\begin{enumerate}
    \item degree of $I$, i.e.  $\deg(I)\leq \lambda_{\deg}(d)$.
    \item projective dimension of $I$, i.e. $\mathrm{pdim}(I)\leq \lambda_{pdim}(d)$,
    \item  Castelnuovo-Mumford regularity, i.e. $\mathrm{reg}(I)\leq \lambda_{\mathrm{reg}}(d)$,
    \item the $(i,j)$-th Betti numbers of $I$, i.e. $\beta_{ij}(I)\leq \lambda_{\beta_{i,j}}(d)$.
\end{enumerate}

\end{restatable}

In fact, several other invariants of $I$ such as the number of minimal or associated primes, degrees and number of generators of minimal primes will also be uniformly bounded. We refer to \cite{AH20a,ESS21} for a list of such invariants.

On one hand, Sylvester-Gallai problems demand a much stronger uniform bounds than Stillman's uniformity. 
On the other hand, SG configurations possess many  combinatorial and geometric properties. 
Hence the challenges posed by the Sylvester-Gallai problem are novel in nature and require new techniques to overcome them. 
To that end, we build upon the work of \cite{AH20a} and generalize their methods to work over quotients of polynomial rings. 
We also develop several algebraic geometric tools to prove effective quantitative bounds for commutative algebraic questions appearing naturally in SG type problems. 
Our results enable us to carry out an inductive argument by leveraging the combinatorics of the configuration and eventually we reduce the radical SG problem to the robust linear SG theorem. 
We expect these techniques to be fundamental for the full resolution of Gupta's program to obtain a polynomial-time deterministic black-box PIT algorithm for dept-4 circuits with bounded top and bottom fanin. 
In the following subsections we discuss how SG problems are motivated from PIT, our results,  technical contributions, key ideas and an overview of the proof, some of which are interesting in their own right.

\subsection{Polynomial Identity Testing and Gupta's conjecture}

Given an algebraic circuit, the Polynomial Identity Testing (PIT) problem asks to determine whether the circuit computes the identically zero polynomial. 
Even though there are efficient randomized algorithms for PIT \cite{Z79, S80}, devising an efficient deterministic algorithm for this problem is one of the main challenges in theoretical computer science. 
The PIT problem is also intimately connected to other fundamental questions such as lower bounds for arithmetic circuits \cite{HS80, Agr05,KI04,DS07,FSV18, CKS18} and derandomization problems \cite{FS13, KSS15, Mul17, GT17, FGT19}. 
We refer the reader to the excellent surveys \cite{Sax09, SY10, Sax14} for more on the PIT problem and its applications. 
The results of \cite{AV08, GKKS16} show that in order to solve the PIT problem in general, it is enough to solve the problem for low depth circuits such as unrestricted depth-$3$ circuits or homogeneous depth-$4$ circuits. 
Therefore, these classes have received considerable attention in recent years. 
In particular, we have seen significant progress towards PIT for depth-$3$ circuits via rank bounds for linear SG-configurations \cite{DS07, KS09, SS13}.

Sylvester-Gallai configurations are intrinsic to the PIT problem as demonstrated by the following example. 
Suppose a polynomial $P\in\mathbb{C}[x_1,\cdots,x_N]$ has the expression $P=a_1\cdots a_d+b_1\cdots b_d+c_1\cdots c_d$ where $a_i,b_j,c_k$ are pairwise non-associate linear forms (in particular, $P$ is represented by a depth-$3$ circuit of the form $\Sigma^{[3]} \Pi^{[d]} \Sigma$). 
If $P\equiv 0$, then whenever $a_i,b_j$ vanish there exists a $c_k$ which also vanishes, in particular $c_k\in \mathrm{rad}(a_i,b_j)=(a_i,b_j)$. 
Hence, the sets $R=\{a_1,\cdots, a_d\}$, $B=\{b_1,\cdots, b_d\}$, $G=\{c_1,\cdots, c_d\}$ form a colored Sylvester-Gallai configuration and by the Edelstein-Kelly theorem \cite{EK66}, we have $\mathrm{dim}(\mathrm{span}_{\mathbb{K}}(R\cup B\cup G))\leq 3$. 
Therefore, if $P\equiv 0$, then we may assume that all the linear forms $a_i,b_j,c_k$ depend on three variables after a change of coordinates and reduce the problem to $\C[x,y,z]$. 
Thus the rank bounds in linear SG theorems imply that the constituent linear polynomials of the circuit depend on few variables, and more sophisticated applications of this idea yielded the PIT algorithms for special depth-$3$ circuits. 
However, we are still far from obtaining polynomial time deterministic PIT algorithms for unrestricted depth-$3$ circuits.

For depth-$4$ circuits, the Sylvester-Gallai configurations appearing in a polynomial identity are no longer linear, due to the presence of higher degree polynomials. 
Hence the case of depth-$4$ circuits is even more challenging. 
In \cite{BMS13,G14}, it was conjectured that such non-linear SG-configurations should also depend on few variables or have low algebraic rank. 
In particular, Gupta~\cite{G14} proposed far-reaching non-linear generalizations of the known Sylvester-Gallai type problems and  showed that if the posed conjectures are true, then we will have a deterministic polynomial-time black-box PIT algorithm for the model of depth 4 algebraic circuits with constant top and bottom fan-in. 
Such a circuit is  denoted by $\Sigma^k\Pi\Sigma\Pi^d$ (where $d, k$ should be thought of as constants) and it computes a polynomial $P$ of the form 
\begin{equation}\label{eq:depth 4 intro}
    P=\sum_{i=1}^k \prod_{j=1}^{n_i} Q_{ij}
\end{equation}   
where $\deg(Q_{ij}) \leq d$ for each $i, j$.

Gupta conjectured that the non-linear Sylvester-Gallai configuration of polynomials appearing in a depth-$4$ polynomial identity must also be depend on few variables. 
More precisely, if a polynomial $P$, evaluated by a depth-$4$ circuit\footnote{One needs to assume that the circuit is simple and minimal. For the precise definitions, we refer the reader to \cite{SS13, G14}.} as above, is identically $0$, then the transcendence degree of the set $\{Q_{i,j}\}$ is conjectured to be bounded by a function depending only on $k$ and $d$. 

The first challenge in Gupta's series of conjectures is the following non-linear version of the linear Sylvester-Gallai problem, \cite[Conjecture 2]{G14}:

\begin{conjecture}[Radical Sylvester-Gallai Conjecture]\label{conjecture: gupta radical SG}
There is a function $\lambda : \N \to \N$ such that the transcendence degree of any $\rsg{d}$ configuration $\cF$, is upper bounded by $\lambda(d)$.
\end{conjecture}

 Since the transcendence degree of $\cF$ is upper bounded by the dimension of the $\K$-linear span of $\cF$, the following stronger version of the conjecture has been considered.

\begin{conjecture}[Strong Radical Sylvester-Gallai Conjecture]\label{conjecture: strong radical SG}
There is a function $\lambda : \N \to \N$ such that for any $\rsg{d}$ configuration $\cF$, we have $\dim (\Kspan{\cF} )\leq \lambda(d)$.
\end{conjecture}

In other words, for any $\rsg{d}$ configuration $\cF$, we must have that $\dim(\Kspan{\cF})$ is upper bounded by a function independent of the polynomial ring $\K[x_1,\cdots,x_N]$ and the cardinality $|\cF|$. In \cite{S20}, Shpilka broke ground on \cref{conjecture: strong radical SG}, solving it for $d = 2$.
After that, \cite{OS22} solved \cref{conjecture: strong radical SG} for $d = 3$. In this work, we positively settle \cref{conjecture: strong radical SG} for any value of $d$, which in turn positively resolves \cref{conjecture: gupta radical SG}.
In this work, our \cref{theorem: sylvester-gallai} completely settles this question in the affirmative.

In this paper, we restrict to the case of characteristic $0$, since the conjecture is false in characteristic $p>0$. 
In fact, even the linear Sylvester-Gallai theorem is not true over $\F_p$, as one may take the finite set of linear forms to be all the linear forms in $\F_p[x_1,\cdots,x_N]$. 
Our \cref{theorem: sylvester-gallai} is a consequence of our main technical theorem which holds in the more general setting of reducible forms in \emph{strong} quotients of polynomial rings. 
In the following subsection we will describe our main technical theorem and the key ideas behind the proof.

\subsection{Main technical results \& overview of proofs }

We develop an inductive framework for solving \cref{conjecture: strong radical SG}. 
We induct on the degree of the forms in the configuration, eventually reducing the \rsg{d}  problem to the robust linear Sylvester-Gallai theorem (see \cref{theorem: fractional linear SG}).
An important issue that arises in the process of reducing degree is that radical SG configurations in polynomial rings do not seem to behave well when we try to reduce the degree.
For this reason, as is often the case in inductive arguments, one needs to start with a stronger hypothesis, which may be harder to prove but amenable to induction.

We will now describe the ideas for reducing the degree of the configuration, along with our generalized definition of radical SG configurations and a high-level overview of our inductive approach. 
We will also see how our generalized definition arises naturally in the process.

\subsection*{I. Degree reduction and general radical SG configurations}
In the previous works on low-degree SG-configurations \cite{S20,OS22}, the idea to reduce the degree was employed in a special case as follows.

\emph{Special case in low degree.} Let $\cF=\{F_1,\cdots,F_m\}$ be a radical SG configuration of degree at most $3$ in a polynomial ring $S$. Suppose that there is a small set of linear forms $W=\{\ell_1,\cdots,\ell_r\} \subset S_1$, such that the ideal $(W)$ contains $\cF$. Then one proceeds as follows.
\begin{itemize}
    \item \emph{General projection.}
Apply a ``general projection'' of $W$, sending all forms in $W$ generically to a new variable $z$. In other words, consider the composition $S\hookrightarrow S[z]\xrightarrow[]{\varphi} S[z]$ defined by $\varphi(\ell_i)=\alpha_iz$, where $\alpha_i\in \K$ are general scalars . Since $\cF\subset (W)$, after the projection all forms in $\varphi(\cF)$ are divisible by $z$. Hence for all $F_i\in \cF$, one obtains a residual form $G_i\in S[z]$ where $\varphi(F_i)=zG_i$. 
\item \emph{Degree reduction.} Show that the set $\{G_i|i\in [m]\}\cup \{z\}$ forms a SG-type configuration and thereby reduce the degree of the configuration.
\item  \emph{Lifting.} The general choice of the scalar multiples of $z$ allows one to lift the SG-type bound back to the configuration $\cF$.

\end{itemize}

Note that the process above yields a set of forms divisible by $z$, which we will allow in our generalized definition. Furthermore, the residual forms $G_i$ as above, are not necessarily irreducible. For example, let $F = x u^2 - y v^2\in S:= \K[x,y,u,v]$, where $W = \Kspan{x, y}$.
Consider the projection $x \mapsto \alpha^2 z$, $y \mapsto \beta^2 z$, where $\alpha, \beta \in \K$. Then we have $F \mapsto z (\alpha u - \beta v)(\alpha u + \beta v)$ in the quotient ring $S/(x-\alpha^2z,y-\beta^2z)$. As this behaviour is unavoidable, one must consider \emph{reducible forms} in the definition of Sylvester-Gallai configurations if one is to use such projections or quotients inductively.

The reason why we would like to work over a larger class of rings (rather than just polynomial rings) is more subtle, and has not appeared in previous works. 
In general, it is not necessary that a there is a small set of linear forms $W$ such that $\cF\subset (W)$, even in the low-degree case. 
For example, if $\cF$ contains a high-rank quadric then $W$ must be large. 
The previous works navigate this issue by heavily using structure theorems for ideals generated by two quadrics or cubics. 
In general, one cannot hope for such a classification or structure theorem for codimension $2$ complete intersection ideals of arbitrary degree.

We shift perspective here and employ a new strategy to overcome this challenge. 
Given a $\rsg{d}$ configuration $\cF$ in $S$, we will \emph{directly} construct a small set of possibly \emph{non-linear} forms $W$ such that the ideal $(W)$ contains all the highest degree forms in $\cF$. 
Then we will take a general quotient using $W$ and reduce $\cF$ to a lower degree SG-configuration in a graded quotient of $S$. 
This quotient ring will be a UFD, but not necessarily a polynomial ring. 
Thus we are led to generalizing the definition of radical SG-configurations in UFDs. 

We generalize the definition of radical SG configurations in two key ways.  Firstly, we allow the ambient ring to be a finitely generated, $\N$-graded $\mathbb{K}$-algebra which is a unique factorization domain (UFD). Secondly, we allow the forms to be reducible. However, we require that there is a special linear form $z$ such that for any two forms in the configuration, the only common factor is $z$.

\begin{restatable}[Radical Sylvester-Gallai Configuration in a UFD]{definition}{radicalSGgeneral}\label{definition: radical SG general}

Let $d$ be a positive integer. Let $R=\bigoplus_{i\in \N} R_i$ be a finitely generated, $\N$-graded $\mathbb{K}$-algebra with $R_0=\K$,  $R$ is generated by $R_1$. Suppose that $R$ is a UFD. 
Let $z \in R_1$ be a linear form and $\cF := \{z, zF_1, \ldots, zF_m\} \subset R$ be a finite set of square-free forms of degree at most $d+1$.
We say that $\mathcal{F}$ is a $\rsgufd{d}{z}{R}$ configuration if for every $i \neq j \in [m]$: 
$$\gcd(F_i, F_j) =1 \text{ and } |\rad(zF_i, zF_j) \cap \mathcal{F}| \geq 3.$$
\end{restatable}

With this definition at hand, we are ready to outline the key steps of our inductive approach.

\subsubsection*{II. Our inductive approach.} 
Let $\cF$ be a $\rsgufd{d}{z}{R}$ configuration as above, and $\cF_d := \{ F_i \ \mid \ zF_i \in \cF \text{ and } \deg(F_i) = d \}$. 
We would like to construct a small set of forms $W$ containing $\cF_d$ and $z$. 
Once we construct $W$ we will apply general quotients to reduce the degree of the configuration, thereby obtaining a $\rsgufd{d-1}{y}{R'}$ configuration. 
We outline the steps below.

\begin{itemize}
    \item \emph{Small set of forms.} We construct a small set of forms $W := \{z, H_1,\cdots,H_r\}$ of degree at most $d$, such that $\cF_d \subset (W)$.  
    Note that the forms in $W$ are no longer required to be linear.
    \item \emph{General quotient.} Let $y$ be a new variable. We apply a general quotient given by $z \mapsto \alpha y$ and $H_i\mapsto \alpha_i y^{\deg(H_i)}$ where $\alpha,\alpha_i\in \K$ are sufficiently general. 
    More precisely, consider the image configuration $\varphi(\cF)$ under the quotient morphism $\varphi:R[y]\rightarrow R':=R[y]/J$, where $J=(z-\alpha y, H_i-\alpha_i y^{\deg(H_i)}|i\in [r])$. 
    Note that $R'$ is a $\N$-graded finitely generated $\K$-algebra, as $J$ is a homogeneous ideal in $R[y]$. 
    \item \emph{Degree reduction.} After taking the reduced parts of the forms in $\varphi(\cF)$, we obtain a radical SG-configuration of square-free forms $\cF'=\varphi(\cF)_\red$ in the quotient ring $R'$. 
    Now all the forms in $\cF'$ are divisible by $y$ and the maximum degree of the residual forms in $\cF'$ is $(d-1)$.
    \item  \emph{Lifting.} We show that sufficiently general choice of the scalars $\alpha_i$ allows us to lift the inductive bounds back to $\cF$ in $R$. 
    Furthermore, the cardinality of $W$ being small implies that the lifted bound depends only on $d$.

\end{itemize} 

All of the steps above present us with distinct technical challenges, which were not encountered in previous works on low-degree configurations. The first step is crucial and the most challenging one due to following reasons.

First, the cardinality of $W$ should be small. More precisely, $|W|$ should be bounded above by a function of $d$ only, independent of the field $\K$, the number of variables $N$ and $|\cF|$ or $|\cF_d|$.

Second, the radical SG theorem does not hold in arbitrary quotients of polynomial rings even if all the forms in the SG-configuration are of degree $1$ (see \cref{example: counterexamples in quotients}). 
Therefore, even if we apply successive general quotients and reduce to a  radical SG configuration of degree $1$ in a quotient ring $R$, we cannot ensure that such a \rsg{1} configuration is low-dimensional in general. 
The failure of the radical SG-theorem in an arbitrary quotient ring, is due to the fact that an ideal $(\ell_1,\ell_2)$ generated by two linear forms $\ell_1,\ell_2\in R$ may not be a prime ideal. 
Hence, the condition $\ell_3\in \rad(\ell_1,\ell_2)$ is not equivalent to the linear SG-condition $\ell_3\in \Kspan{\ell_1,\ell_2}$, unlike the case of linear forms in a polynomial ring. 
In fact, one cannot ensure primality of $(\ell_1,\ell_2)$, even if $R$ is a UFD, as shown in the following example.

\begin{example} \label{example: SG in ufd}
    Let $F=x_1x_2+x_3x_4+x_5^2\in S=\K[x_1,\cdots,x_7]$ and $R=S/(F)$. Then $R$ is a UFD. Consider the images $x_1,x_3,x_5\in R$. Then we have $x_5\in \rad(x_1,x_3)$ in $R$, but $x_5\not \in \Kspan{x_1,x_3}$. Note that this issue does not arise if the rank of the quadric is high enough. Indeed, let $G=x_1x_2+x_3x_4+x_5^2+x_6^2+x_7^2$. Then $S/(G)$ is a UFD, and  the ideal generated by any two linear forms $(\ell_1,\ell_2)$ is a prime ideal in $S/(G)$.
\end{example} 

Therefore, we need to  identify a subclass of UFDs, such that the radical SG theorem holds over rings in this subclass. 
Furthermore, in order to strike a balance between the two requirements above, we need a method for constructing a small set of forms $W$ with $\cF_d \subset (W)$, such that even after a general quotient we obtain a UFD in this subclass. 

One major challenge is that $W$ needs to be \emph{small}, hence the requirements above are in tension with each other. 
For example, if we started with a polynomial ring $R$, we can ensure that the quotient ring $R'$ is a polynomial ring, by choosing $W$ to be a set of linear forms. 
But, in that case $W$ will not always be small as observed earlier. 
Therefore we cannot restrict to work only with polynomial rings. 
As hinted in \cref{example: SG in ufd} above,  if the forms in $W$ are ``high rank'' or \emph{strong} enough, then we can expect the quotient to be a UFD where the \rsg{1} theorem holds. 
In this work, we identify the correct subclass to be quotients of polynomial rings by \emph{strong} sequences and our main technical result \cref{theorem: radical SG theorem over UFD} proves that the $d$-radical SG theorem holds over such strong quotients.

Now we will discuss the precise notion of strength, our main technical theorem and the ideas for constructing the desired small set of forms.

\subsection*{III. Strength of forms, strong vector spaces \& algebras} 
Let us note two simple yet extremely useful properties of linear forms in a polynomial ring $S=\K[x_1,\cdots,x_N]$. 
First, a set of linearly independent linear forms $\ell_1,\cdots,\ell_r\in S$, is also algebraically independent. 
Therefore the $\K$-subalgebra $\K[\ell_1,\cdots,\ell_r]$ is isomorphic to a polynomial ring and $\ell_1,\cdots,\ell_r$ can be thought of as variables in this polynomial ring. 
Second, the quotient $S/(\ell_1,\cdots,\ell_r)$ is again a polynomial ring, and hence a UFD.

In general, both the properties above fail for higher degree forms, even for quadrics. 
However, we know that if a quadric $Q$ is of rank at least $3$, then $S/(Q)$ is a UFD \cite[Chapter II.6 Exercise 6.5]{Har77}. 
The notion of strength\footnote{Classically known as Schmidt rank.} of a form generalizes the notion of rank for quadratic forms. 
In the course of proving Stillman's conjecture \cite{AH20a}, Ananyan and Hochster  defined strength of forms and proved thoroughgoing generalizations of the properties of linear and quadratic forms above to higher degree forms. We will now discuss the notion of strength and their results below.

\emph{Collapse and strength.} Let $S=\K[x_1,\cdots,x_N]$ be a polynomial ring and $S_d$ denote the space of degree $d$ forms.
We say that a form $F\in S_d$ has an $r$-collapse if $F$ can be written as $F = G_1 H_1 + \cdots + G_r H_r$ where $1 < \deg G_i < d$ for all $i \in [r]$.
The strength of a form $F$, denoted $s(F)$, is the maximum integer $s$ such that $F$ has no $s$-collapse (hence $F$ has a $s+1$-collapse). Linear forms are defined to have infinite strength.  We say that a form $F$ is $k$-strong if $s(F)\geq k$.

With the above definition, \cite[Theorem A]{AH20a} generalizes the above classical result as follows. 
There is an increasing function $A : \N \to \N$, independent of $\K$ and the number of variables $N$, such that for any form $F \in S_d$ with $s(F) \geq A(d)$, it holds that $S/(F)$ is a UFD.
 
In order to extend this result to homogeneous ideals generated by more than one form, one needs to quantify the collective strength of a set of forms or a vector space. 
Define the \emph{minimum strength} of a vector space $V \subset S_d$, denoted by $\smin{V}$, to be the minimum strength of  \emph{any nonzero} element of $V$. 
Since the generators of a homogeneous ideal can be of different degrees, we also need to define collective/minimum strength of graded vector spaces. 
Let $V=\bigoplus_{i=1}^dV_i \subset S$ be a graded vector space with $\dim(V_i)=\delta_i$. 
Given any function $C:\N^d \to \N^d$, we say that $V$ is $C$-strong if $\smin{V_i}\geq C_i(\delta)$, where $\delta=(\delta_1,\cdots,\delta_d)$ is the dimension sequence of $V$ and $C=(C_1,\cdots,C_d)$.  
We say that a $\K$-subalgebra $\cA\subset S$ is a $C$-strong algebra if $\cA=\K[V]$, where $V$ is a $C$-strong vector space. 
Now \cite[Theorem A]{AH20a} shows that there exists an ascending function $A : \N^d \to \N^d$, independent of $\K$ and $N$, such that if $V\subset S$ is a $A$-strong vector space, then any basis of $V$ is algebraically independent. 
Thus the $A$-strong algebra $\K[V]$ is isomorphic to a polynomial ring. 
Furthermore, we have that $S/(V)$ is a UFD. 

Now that we have defined strong vector spaces and discussed some of their important properties, we can discuss how we build on them to solve our SG problem.

\subsection*{IV. Main technical theorem}

In subsection II, we saw that one can only hope to prove SG type theorems over certain subclasses of UFDs.
In subsection III we saw how strong spaces give rise to special UFDs. 
In this section, we consider the subclass of such UFDs given by quotient rings $R=S/(U)$, where $S$ is any polynomial ring and $U\subset S$ is a sufficiently strong vector space. 
Our main technical theorem below shows that general radical Sylvester-Gallai configurations over such strong quotients are low dimensional. 

\begin{restatable}[General Radical Sylvester-Gallai Theorem]{theorem}{generalRadicalSG}
\label{theorem: radical SG theorem over UFD}
For any two positive integers $d,e$ such that $e\geq 2$ and $d\leq e$, there exist ascending functions $\Lambda_{d,e} : \N^e \to \N^e$ and  $\lambda_{d, e}:\N \rightarrow \N$, both independent of $\K$ and $N$, such that the following holds. 

Let $\K$ be an algebraically closed field of characteristic $0$ and $S := \K[x_1,\cdots,x_N]$. 
Let $U\subset S_{\leq e}$ be a $\Lambda_{d,e}$-strong graded vector space and $R=S/(U)$. 
If $\cF$ is a  $\rsgufd{d}{z}{
R}$ configuration for some $z\in R_1$, then we have
$$\dim (\Kspan{\cF} )\leq \lambda_{d,e}(\dim(U)).$$
i.e. the dimension of the $\K$-linear span of $\cF$ is upper bounded by a function of $d,e$ and $\dim(U)$, which is independent of the field $\K$, the number of variables $N$ and the cardinality of $\cF$.
\end{restatable}

The functions $\Lambda_{d,e}$ will be larger than the Ananyan-Hochster function $A$ described above. 
Hence the quotient ring $R=S/(U)$ will be a UFD. \cref{theorem: radical SG theorem over UFD} shows that if $U$ is generated in degree at most $e$, and it is $\Lambda_{d,e}$-strong, then the dimension of the span of any $\rsgufd{d}{z}{R}$ configuration will be uniformly bounded by $\lambda_{d,e}(\dim(U))$ for all $d\leq e$. 
We note that the condition $d\leq e$ is to make the inductive definition of the functions $\Lambda_{d,e}, \lambda_{d,e}$ simpler and it is not restrictive in practice, as one can replace $e$ by $\max{\{d,e\}}$.
Now given any polynomial ring $S$, we may write $S=S[u]/(u)$, where $u$ is a new variable. 
Thus we recover \cref{theorem: sylvester-gallai} with the bound $\lambda(d)=\lambda_{d,d}(1)$.

The proof of \cref{theorem: radical SG theorem over UFD} is much more involved than previous proofs of the low-degree cases $d= 2,3$. Thus it requires significantly more techniques from commutative algebra and algebraic geometry. We generalize the methods of \cite{AH20a} to work over strong quotients of polynomial rings. In particular, we generalize the notion of strength to graded algebras and using a notion of \emph{lifted strength} we show that an analogue of \cite[Theorem A]{AH20a} works for strong quotients of polynomial rings (see \cref{sec:strong-algebras}). However, the construction of the small set $W$ in the inductive step is the most involved part of the argument. It is also arguably the most interesting part, because of the interplay of algebraic geometry and combinatorics. We describe some of the ideas for this construction below and refer to \cref{sec:main} for details.

\subsection*{V. Construction of a small set of strong forms.} 
Let $S=\K[x_1,\cdots,x_N]$. Given a \rsgufd{d}{z}{R} configuration in a strong quotient ring $R=S/(U)$, recall that our goal is to construct a \emph{small} set of \emph{strong} forms $W$ such that $\cF_d \subset (W)$, where $|W|$ depends only on $d$, independent of $\K$, $N$ and $|\cF|$. 
In fact, the $\K$-linear span of $W$ needs to be sufficiently strong to ensure that successive general quotients are still in the subclass of strong quotient UFDs. 

In \cite[Theorem B]{AH20a}, the authors proved the following statement and confirmed Stillman's conjecture as a consequence. 
Given forms $f_1,\cdots,f_r\in S_{\leq d}$ there exist $s$ number of sufficiently strong forms $g_1,\cdots,g_s\in S_{\leq d}$, such that $\{f_1,\cdots,f_r\}\subset \K[g_1,\cdots,g_s]$ , where $s$ depends only on $d$ and $r$, independent of $N$ and $\K$. 
However, this does not solve our problem even if $R$ is a polynomial ring, as the number of forms $s$ necessarily depends on the initial number of forms $r$. 
Therefore in addition to generalizing the construction above, we also need to  exploit the geometric and combinatorial properties of SG configurations by developing algebraic-geometric tools. 
We give an overview of some of these results below.

\emph{Construction of strong algebras.} The construction of \cite{AH20a} works with specific strength lower bounds given by the function $A: \N \rightarrow \N$ described earlier. We need to have the flexibility to make the forms as strong as we want, due to our successive quotients. We formalize their construction as an iterative process and generalize it to a robust version (see \cref{proposition: constructing AH algebras}). In particular, we show that given a vector space $U\subset S$ and any function $B:\N^d\rightarrow \N^d$, we can construct a $B$-strong vector space $V$ with $U\subset \K[V]$, where $\dim(V)$ is bounded by a function depending on $B$, $\dim(U)$ and $d$. Furthermore, we have the robustness property that any sufficiently strong form in $U$ is preserved during the construction and is contained in $V$.

\emph{A small set of strong forms.}  We will construct the small set of strong forms $W$ by iteratively applying a sequence of general quotients as follows. 
We start by choosing a small number of forms $F_1,\cdots,F_r\in \cF_d$. 
Here $r$ is some function of $d$ only, and in fact it turns out to be enough to  take $r=O(d^3)$. 
Now we apply the construction above to obtain a small, sufficiently strong algebra $\K[V]$ containing $F_1,\cdots,F_r$. 
Let us say that a form $F\in \cF$ of degree $d$ is good if $F\in (V)$, otherwise we call $F$ a bad form. 
If $\cF_d \subset (V)$ then we are done. 
Otherwise, we will apply a general quotient using $V$ to obtain a SG-configuration $\cF'$ in a strong quotient ring $R'$. 
Then we apply the process iteratively on $\cF'$ in $R'$ and obtain a strong algebra $\K[V']\subset R'$. 
We repeat the process above as long as there exist bad forms in our strong quotient ring. 
By our robustness property, these successive general quotients can be realized as a strong general quotient of the polynomial ring $S$. 

Note that at each step of the process, the construction reduces the number of bad forms by at least $r$, which was chosen to be a small number. However, using our algebraic geometric results and quantitative bounds described below, we are able to leverage the combinatorics of the SG configuration to ensure that the number of bad forms actually reduces significantly (see \cref{subsection: combinatorial preparations}). More precisely, we show that if a sufficiently strong algebra contains a large number of forms, i.e. $\K[V]\cap \cF$ is large, then we must have that $\cF_d\subset (V)$, as desired (see \cref{lemma: algebra contains many then ideal contains Fd}). Otherwise, in \cref{lemma: fraction of Fd}, we show that we can indeed increase the algebra $\K[V]$ to another small strong algebra $\K[V']$, such that the ideal $(V')$ contains a large fraction of $\cF_d$. 
Therefore, the number of bad forms reduces significantly. 
Finally, in \cref{subsection: final proof}, we show that after $O(d)$ number of iterations of the step above, we will end up with all highest degree forms in a ideal generated by a small dimensional strong vector space $V'$ in a general quotient $S'$. Then we can obtain $W$ by lifting the generators to $S$. 

\subsection*{VI. Algebraic-geometric tools and effective bounds.} 
Lastly, our strategy naturally leads us to study several interesting algebraic-geometric questions. 
We will now describe some of the algebraic-geometric results that we prove, along with our motivation in the context of SG configurations. Our results work over strong quotients of polynomial rings (see \cref{sec:algebraic-geometry}), however here we state the version for polynomial rings for simplicity.

Given a sufficiently strong vector space $V$ in a polynomial ring $S$ and a \rsg{d} configuration $\cF\subset S$, we would like to bound the number of bad forms, i.e. the forms $P\in (\cF\cap S_d)\setminus (V)$. In general, there might be many such bad forms. However, a key observation is that if the $\K$-span of degree $d$ forms in $\cF$ is high dimensional, then many of the ideals $(P,Q)$ tend to be non-radical, where $P,Q\in\cF$ are both degree $d$ forms. Indeed, since $P,Q\in S_d$, if $(P,Q)$ is radical, then by the SG condition there exists a third form $F\in \cF\cap \rad(P,Q)=(P,Q)$. As $\deg(F)\leq d$, we have that $F$ must be of degree $d$ and $F\in \Kspan{P,Q}$. Therefore, if most of the ideals $(P,Q)$ are radical, then we can expect that a large fraction of $\cF\cap S_d$ is a robust linear SG-configuration. Hence it must have low dimensional $\K$-span. We refer the reader to \cref{sec:main} for the precise statements and arguments. 

Thus we are naturally led to the following question. Given a form $P\not\in (V)$, how many degree $d$ forms $Q\in\K[V]\cap \cF$  can be there such that $(P,Q)$ not radical? We first prove the following criterion for radical ideals using discriminants.

\begin{lemma}\label{lemma: CM discriminant radical intro} (see \cref{lemma: CM discriminant radical})
Let $S:=\K[z_1,\cdots,z_r,x_1,\cdots,x_s]$ be a polynomial ring. Let $A=\K[x_1,\cdots,x_s]$. Let  $P,Q_1,\cdots,Q_k\in S$ be forms of positive degree such that $(Q_1,\cdots,Q_k)$ is a reduced complete intersection, $P$ is reduced and $P\not \in (x_1,\cdots,x_s)$. Suppose that, for all $i\in[r]$ such that $P$ depends on the variable $z_i$, we have $\mathrm{Disc}_{z_i}(P)\not\in \fq\cdot S$ for any minimal prime $\fq$ of $(Q_1,\cdots,Q_k)$ in $A$. Then the ideal $(P,Q_1,\cdots,Q_k)$ is radical in $S$.
\end{lemma}

The result above is a generalization of the fact that a polynomial in one variable is non-reduced iff its discriminant is zero. It also generalizes \cite[Lemma 3.21]{OS22}, where a similar result was proved for ideals $(P,Q)$ with $P,Q$ irreducible. As a corollary we obtain the following effective quantitative bound which answers our question above and generalizes \cite[Corollary 3.24]{OS22}.

\begin{corollary}[see \cref{corollary: best corollary discriminant}]
 Let $S:=\K[z_1,\cdots,z_r,x_1,\cdots,x_s]$ be a polynomial ring. Let $A=\K[x_1,\cdots,x_s]$. Let $P\in S$ be reduced form of degree $d\geq 1$ such that $P\not\in (x_1,\cdots,x_s)$. Then there exist at most $d^2(2d-1)$ number of reduced forms $Q_i\in A$ such that $Q_i,Q_j$ do not have common factors for $i\neq j$ and $(P,Q_i)$ is not radical for all $i$.

\end{corollary}

We remark that the bound $d^2(2d-1)$ helps us ensure that a choice of $r=O(d^3)$ works in the construction of small set of forms described earlier. We refer to \cref{sec:main} for precise statements and proofs. Our methods also yield the following primality criterion.

\begin{lemma}[see \cref{lemma: irreducible prime}]
 Let $P,Q$ be irreducible forms in $S=\K[x_1,\cdots, x_s,y_1,\cdots,y_n]$. Suppose that $Q\in \K[x_1,\cdots,x_s]$  and $P$ is irreducible in $S/(x_1,\cdots,x_s)$. Then the ideal $(P,Q)$ is prime.
 \end{lemma}

In fact, \cref{lemma: irreducible prime} works with the subalgebra $\K[x_1,\cdots,x_s]$ replaced by any sufficiently strong subalgebra $\K[V]$. In practice, we can apply the criterion as follows. Given irreducible forms $P,Q\in S$,  apply \cref{proposition: constructing AH algebras} to construct a sufficiently strong algebra $\K[V]\subset S$ such that $Q\in \K[V]$. If $P$ is irreducible in $S/(V)$, then the ideal $(P,Q)$ is prime.

The results above are stated for a polynomial ring. However, using our construction and properties of strong algebras, we are able to prove the results for strong quotient rings (see \cref{sec:algebraic-geometry} and \cref{sec:strong-algebras}). We achieve this by using a transfer principle which allows us to transfer the results from polynomial rings to strong quotients.

\emph{A transfer principle.} Let $R=S/(U)$ be a quotient of a polynomial ring $S$, where $U$ is sufficiently strong. Let $F_1,\cdots,F_r\in R$ be a small number of forms. Since $U$ is sufficiently strong, then we can construct a strong algebra $A:=\K[V]\subset R$, such that $F_1,\cdots,F_r\in A$ and the algebra $A$ is isomorphic to a polynomial ring. Furthermore, the extension $A\subset R$ is intersection flat and has several useful properties. For example, for any prime ideal $\fp\subset A$, the extension $\fp \cdot R$ is a prime ideal in $R$ (see \cref{proposition: extension of primes general}). Now the transfer principle can be stated as follows. In order to prove commutative algebraic statements for the ideal $I=(F_1,\cdots,F_r)$ in $R$ such as primality, reducedness, number of minimal primes or to bound the number of non-radical pairs in $\{F_1,\cdots,F_r\}$, it is enough to prove the corresponding statement in the polynomial ring $A$ (see \cref{proposition: radicals and prime sequences}).

We employ this transfer principle throughout our work. The requirement of $U$ being sufficiently strong can be quantified depending on the number $r$. The larger the number of forms $r$, the stronger $U$ needs to be. The precise quantitative requirements for strength of $U$ depends on the context of the application. We refer the reader to \cref{sec:algebraic-geometry} and \cref{sec:strong-algebras} for precise statements and details of applications of this transfer principle.

\subsection{Related Work}

\paragraph{Radical Sylvester-Gallai problems:} 
The works \cite{S20, OS22} proved radical SG theorems in low degree $d=2,3$. These works heavily relied on structure theorems for complete intersection ideals generated by pairs of quadrics or cubics, and the arguments crucially used that the degree of the forms is $2$ or $3$. In general one can not hope for such characterizations for ideals of arbitrary degree. We vastly generalize the methods and ideas of \cite{S20, OS22} to work over strong quotient rings and our proof is completely new as we do not rely on such structure theorems. Some of the results in \cref{sec:algebraic-geometry} of this paper are similar to the results in \cite{OS22}, where the authors proved weaker versions for irreducible forms in polynomial rings. We generalize these results to reducible forms in strong quotient UFDs. Furthermore, our definition of strong algebras \cref{sec:strong-algebras} is a robust generalization of wide vector spaces in \cite{OS22}. The notion of wide vector spaces worked in degree at most $2$ and were specifically tailored to work with SG configurations. Our general definition of strong algebras is more conceptual and the results generalizing \cite{AH20a} in this paper are of independent interest.

\paragraph{Progress on PIT:} In recent years, we have seen remarkable progress on the PIT problem for constant depth circuits.
In \cite{DDS21}, the authors give a deterministic \emph{quasi-polynomial} time algorithm for blackbox PIT for depth 4 circuits with bounded top and bottom fanins.
Their approach is analytic in nature, allowing them to bypass the need of Sylvester-Gallai configurations.
In \cite{LST21}, the authors give a deterministic \emph{weakly-exponential} algorithm for PIT for these circuits via the hardness vs randomness paradigm for constant depth circuits \cite{CKS19}.
The hardness result for the determinant, proved by \cite{LST21}, coupled with techniques from \cite{AF22}, also yields another weakly-exponential time black-box PIT for constant-depth circuits, as is shown in \cite[Section 6]{AF22}.

Despite all the formidable progress above mentioned, the Sylvester-Gallai type approach of \cite{G14} is the only one so far which could yield deterministic \emph{polynomial-time} black-box PIT algorithms for depth-4 circuits with constant top and bottom fanins.
In \cite{PS22}, this Sylvester-Gallai type approach was carried out for depth-4 circuits with top fanin 3 and bottom fanin 2.

\subsection{Organization}

In \cref{sec:prelim}, we establish the notation and conventions that will be used throughout the paper.
In \cref{sec:sg-configurations}, we formally define all variants of Sylvester-Gallai configurations and previous results that we will need in this paper. We also provide examples of Sylvester-Gallai configurations.
In \cref{sec:algebraic-geometry}, we establish the results from commutative algebra and algebraic geometry that we will need.
In \cref{sec:strong-algebras}, we formally define strength of forms and one of our main tools to solve the Sylvester-Gallai problem: strong Ananyan-Hochster vector spaces and algebras. We prove useful properties of these algebras generalizing the results of \cite{AH20a}.
In \cref{sec:general-quotients}, we formally define general quotients, and prove their technical properties. We also study how Sylvester-Gallai configurations behave under general quotients and how we can lift dimension bounds from general quotients to the original configurations.
In \cref{sec:main} we combine all the results from the previous sections to prove our main technical result and its consequences and thereby resolve \cref{conjecture: gupta radical SG}.

\section*{Acknowledgments}

The authors would like to thank Abhibhav Garg, Aise Johan de Jong,  J\'anos Koll\'ar, Shir Peleg and Amir Shpilka for helpful conversations throughout the course of this work. This work started while Rafael was a member of the Simons Institute during the Geometry of Polynomials program, and Akash was a member of the program on Birational Geometry and Moduli Spaces at MSRI, Berkeley. Part of this work was done while the first author was a postdoctoral fellow at the University of Toronto, and the second author was a graduate student at Princeton University. Majority of this work was completed in the years when the second author was at Columbia University. The authors would like to thank the aforementioned institutions for their hospitality and generous support.

\section{Notations and Conventions}\label{sec:prelim}


In this section, we establish the notation which will be used throughout the
paper and some important background which we shall need to prove our claims
in the next sections.

We will work over an algebraically closed field $\K$ with $\char(\K) = 0$. We will denote by $\N=\{0,1,2,\cdots\}$, the set of natural numbers including $0$. For any positive integer $n$, we let $[n]=\{1,2,\cdots, n\}$. Throughout the paper, we will have $S := \K[x_0, \cdots, x_N]$ a polynomial ring for some positive integer $N$. Note that, $S$ is a $\N$-graded $\K$-algebra with the standard grading by degree $\deg(x_i)=1$ for all $i\in [N]$. 
That is, $\dst S = \oplus_{d \in \N} S_d$, where $S_d$ is the $\K$-vector space of homogeneous polynomials of degree $d$.
As it is standard, we will use \emph{form} when we refer to a homogeneous polynomial. Let $R=\oplus_{d\in \N}R_d$ be a graded ring and $I\subset R$ an ideal. We say that $I$ is a homogeneous ideal iff $I$ is a graded submodule of $R$. If $I\subset R$ is a homogeneous ideal, then the quotient ring $R/I$ is also a graded ring where $(R/I)_d = R_d/I_d$.
A graded vector space of homogeneous elements is a subspace $V \subset R$ generated by homogeneous elements from $R$, thereby inheriting the grading from $S$, i.e. $V = \oplus_{d\geq 0} V_d$, where $V_d=V\cap S_d$.

Recall that $S$ is a unique factorization domain (UFD).
Given a unique factorization domain $R$ and $F,G\in R$, we denote by $\gcd(F,G)$ a generator of the unique minimal principal ideal containing $F,G$. Note that $\gcd(F,G)$ is unique up to a scalar multiple.

For any two functions $f,g:\N\rightarrow \N$, we write that $f=O (g)$ if there exists some constant $C$ and $n_0\in \N$, such that $f(n)\leq C g(n)$ for $n\geq n_0$. Similarly, we write  $f=\Omega(g)$ if if there exists some constant $C$ and $n_0\in \N$, such that $f(n)\geq C g(n)$ for $n\geq n_0$. Furthermore, we say that $f=\Theta (g)$ if there exists constants $C_1,C_2$ and $n_0\in \N$ such that we have $C_1 g(n)\leq f(n) \leq C_2 g(n)$ for all $n\geq n_0$.

\section{Sylvester-Gallai Configurations}\label{sec:sg-configurations}

In this section we formally describe all the generalizations of Sylvester-Gallai configurations that we will need for the rest of the paper, and state the main results that we need from previous works.

To simplify the definitions of our configurations, we say that two forms $F, G \in S$ are \emph{non-associate forms} iff $F$ is not a scalar multiple of $G$ (and vice-versa).
In this section (and in general in all the sets of forms we study), our SG configurations will be made of sets of pairwise non-associate forms.

\subsection{Linear Sylvester-Gallai Configurations}

\begin{definition}[Robust linear SG configuration]
Let $\cF := \{ \ell_1,\cdots, \ell_m \} \subset S_1$ be a finite set of pairwise non-associate linear forms in $S$ and let $\delta \in (0, 1]$. 
We say that $\cF$ is a $\vdlsg{\delta}$ configuration if the following condition holds:
\begin{enumerate}
    \item for every $i \in[m]$, there exist at least $\delta (m-1)$ values of $j \neq i$  such that $|\cF \cap \Kspan{\ell_i, \ell_j}|\geq 3$.
\end{enumerate}  
If $\delta=1$, then we simply call it a linear SG configuration.
\end{definition}

It was proved in \cite{BDWY11,DSW14, DGOS18} that the dimension of the span of a $\vdlsg{\delta}$ configuration is bounded by a function depending only on the robustness parameter $\delta$.
We now state the sharpest known result, from \cite[Theorem 1.6]{DGOS18}.

\begin{theorem}\label{theorem: fractional linear SG}
If $\cF$ is a $\vdlsg{\delta}$ configuration, then $\displaystyle \mathrm{dim}(\Kspan{\cF})\leq \left\lceil \frac{4}{\delta} \right\rceil - 1.$
\end{theorem}

We will need the following strengthening of the result above, which comes from considering a slightly more general type of linear SG configurations, where we also allow certain SG pairs $(\ell_i, \ell_j)$ to intersect non-trivially a small dimensional vector space, instead of spanning a third element of the configuration.

\begin{restatable}[Robust linear Sylvester-Gallai configurations over a vector space]{definition}{robustLinearSG}\label{def:robust-linear-SG}
Let $c \in \N$, $0 < \delta \leq 1$ and $\cF := \{\ell_1, \ldots, \ell_m\} \subset S_1$ be a set of pairwise non-associate linear forms. 
We say that $\cF$ is a $\dlsg{c}{\delta}$ configuration if there exists a vector space $W \subset S_1$ of dimension at most $c$ such that the following condition holds: 
\begin{itemize}
    \item for any $\ell_i \in \cF \setminus W$, there exist at least $\delta (m-1)$ indices $j \in [m] \setminus \{i\}$ such that $\ell_j \not\in W$ and
    $$| \Kspan{\ell_i, \ell_j} \cap \cF | \geq 3 \ \ \ \text{ or } \ \ \  \Kspan{\ell_i, \ell_j} \cap W \neq 0.$$
\end{itemize} 
\end{restatable}

It was noted in \cite[Corollary 16]{S20}, that a $(1,\delta)$-linear SG-configuration has dimension bounded by $50/\delta$. 
The proof of \cite[Corollary 16]{S20} uses \cite[Theorem 1.9]{DSW14} which states that a $\delta$-linear SG-configuration has dimension bounded by $12/\delta$. 
However, using the sharper bound of \cref{theorem: fractional linear SG}, the same argument yields the following version of \cite[Corollary 16]{S20}.

\begin{proposition}\label{proposition: one delta amir}
A $\dlsg{1}{\delta}$ configuration has dimension at most $1 + 8/\delta$.    
\end{proposition}

The following proposition generalizes \cite[Corollary 16]{S20} and \cref{proposition: one delta amir} to show that the dimension of a $(c,\delta)$-linear SG configuration is also bounded.

\begin{restatable}[Robust Linear SG Configurations]{proposition}{lincdeltasg}\label{prop:lincdeltasg}
Let $\cF$ be a $\dlsg{c}{\delta}$ configuration.
Then $\dim \Kspan{\cF} \leq c + 1 + 8/\delta$.
\end{restatable}

\begin{proof}
Let $\cF = \{\ell_1, \dots, \ell_m\}$, and $W \subset S_1$ be the vector space given by \cref{def:robust-linear-SG}. 
Let $\varphi := \varphi_{W, \alpha} : S \to S[z]$ be a general quotient, and let $y_i = \varphi(\ell_i)$.
Hence, we have $\varphi(W) = \Kspan{z}$, and by item 3 of \cref{proposition: general quotient general}, $\ell_i \not\in W \then y_i \not\in (z)$.
Moreover, by \cref{proposition: gcd after projection}, for any pair $\ell_i, \ell_j \in \cF \setminus W$, we have $y_i \not\in (y_j)$.
In particular, this implies that $\varphi(\cF)$ is a $\dlsg{1}{\delta}$ configuration, which by \cref{proposition: one delta amir} implies $\dim \Kspan{\varphi(\cF)} \leq 1 + 8/\delta$.
Since $\cF \subset \Kspan{\varphi(\cF), z, W}$, we have $\dim \Kspan{\cF} \leq c + 1 + 8/\delta$.
\end{proof}

The next proposition proves that if one has a set of pairwise non-associate forms which is almost a $\vdlsg{\delta}$ configuration, then it must be a $\dlsg{c}{\delta}$ configuration for some small $c$.

\begin{proposition}\label{linear sg simple}
Let $r\in \N$, $\delta \in (0,1)$ and $\cF \subset S_1$ be a finite set of pairwise non-associate linear forms with $|\cF|:=m$. 
Suppose that there exist $r$ forms $F_1,\cdots,F_r\in \cF$ such that for any $F\in \cF\setminus \{F_1,\cdots,F_r\}$, there exists at least $2\delta m$ forms $G_i\in \cF$ such that $| \Kspan{F,G_i} \cap \cF | \geq 3 $. 
Then $\cF$ is an $\dlsg{r}{\delta}$ configuration.
\end{proposition}

\begin{proof}
Let $W=\Kspan{F_1,\cdots,F_r}$. 
As the elements of $\cF$ are pairwise non-associate, no two elements of $\cF$ are scalar multiples of each other. 
For each element $F \in \cF$, let 
$$\Fspan(F) := \{ G \in \cF \ \mid \ |\Kspan{F, G} \cap \cF | \geq 3 \}.$$

If $F\in \cF\setminus W$ then $F\in \cF\setminus \{F_1,\cdots,F_r\}$. 
Hence $|\Fspan(F)| \geq 2\delta m$. 
Let $\Fspan(F) = \{G_1, \ldots, G_s\} \sqcup \{H_1, \ldots, H_t\}$, where $G_i \in \cF \setminus W$ and $H_j \in W$ for all $i \in [s], j \in [t]$.

To prove that $\cF$ is an $\dlsg{r}{\delta}$ configuration, it is enough to show that $s \geq t$, as this implies $2s \geq s + t \geq 2 \delta m \then s \geq \delta m$.
Note that for any $H_j$, there must exist $G_{i_j} \in \Fspan(F) \setminus W$ such that $G_{i_j} \in \Kspan{F, H_j}$.
Additionally, we have $\Kspan{F, H_j} \cap \Kspan{F, H_k} = \Kspan{F}$, otherwise we would have $F \in \Kspan{H_j, H_k} \subset W$, which is a contradiction.
Hence, the above proves that $G_{i_j} \neq G_{i_k}$ for every $j \neq k \in [t]$, which implies that $s \geq t$, as we wanted.
\end{proof}

\subsection{Radical Sylvester-Gallai configurations over quotient rings}

We now define generalizations of the previous definitions of SG configurations to the setting of finitely generated $\K$-algebras which are UFD.
To motivate our general definition, we restate Gupta's original generalization, from \cite[Conjecture 2]{G14}.\footnote{\cref{definition: radical SG} is a restatement of \cref{definition: radical SG intro} with the new notation.}

\begin{restatable}[(Basic) Radical Sylvester-Gallai Configuration]{definition}{radicalSGBasic}\label{definition: radical SG}
    Let $\cF := \{F_1, \ldots, F_m\} \subset S$ be a subset of irreducible, pairwise non-associate forms of degree at most $d$.
    We say that $\cF$ is a $\rsg{d}$ configuration if for every $i, j \in [m]$, we have that $|\rad(F_i, F_j) \cap \cF| \geq 3.$
\end{restatable}

Note that a $\rsg{1}$ configuration is the classical linear SG configuration.

\radicalSGgeneral*

Note that in the definition above we require that the linear form $z$ is in the set $\cF$. 
This is simply a minor convention which will allow these configurations to be preserved under special quotients that will be defined in \cref{sec:general-quotients}. 
We remark that a $\rsgufd{1}{z}{S[z]}$ configuration recovers the $\dlsg{1}{1}$ configuration, which was first defined in \cite[Corollary 16]{S20}.

\begin{remark}\label{remark: grading of SG configs}
Since each form in a $\rsgufd{d}{z}{R}$ configuration $\cF$ is divisible by $z$, we denote by $\cF_e$ the set of forms $F \in R_e$ such that  $zF\in \cF$. 
With this convention and the previous remark, we define $\cF_0 = \{1\}$.
\end{remark}

The next proposition shows that \cref{definition: radical SG general} recovers 
\cref{definition: radical SG}.

\begin{proposition}\label{proposition: vanilla radical SG reduces to general radical SG}
Note that if we consider the ring $S[z]$, where $z$ is free over $S$, $U = (0)$, and $G_1, \ldots, G_m \in S$ are irreducible forms of degree $\leq d$, then the set $\cF := \{z, zG_1, \ldots, zG_m \}$ is a $\rsgufd{d}{z}{S[z]}$ configuration if, and only if, $\cG := \{G_1, \ldots, G_m\}$ is a $\rsg{d}$ configuration.
\end{proposition}

\begin{proof}
    If $\cG$ is a $\rsg{d}$ configuration, then condition 1 of \cref{definition: radical SG general} is satisfied since $\gcd(G_i, G_j) = 1$ for any $i \neq j$.
    The SG condition on $\cG$ implies that there is $G_k \in \cG$ such that $G_k \in \rad(G_i, G_j)$, which implies $zG_k \in \rad(zG_i, zG_j)$.
    Thus, $\cF$ is a $\rsgufd{d}{z}{S}$ configuration.

    Conversely, if $\cF$ is a $\rsgufd{d}{z}{S}$ configuration, then the gcd condition of \cref{definition: radical SG general} implies $\gcd(G_i, G_j) = 1$ for $i \neq j$, which in particular implies $G_i, G_j$ are non-associate. 
    Since $z$ is free over $S$ and $G_i \in S$, we have that $z, G_i, G_j$ is a regular sequence for any $i \neq j$, as $G_i, G_j$ are non-associate.
    Then, the radical condition of \cref{definition: radical SG general} implies there exists $zG_k \in \rad(zG_i, zG_j) \subset \rad(G_i, G_j)$, and since $z, G_i, G_j$ is a regular sequence, we must have $G_k \in \rad(G_i, G_j)$, which implies $\cG$ is a $\rsg{d}$ configuration.
\end{proof}

As we will see in \cref{sec:main}, for special choices of vector spaces $U$, i.e. the strong vector spaces that we define in \cref{sec:strong-algebras}, we are able to prove that these more general SG configurations must span a small dimensional vector space.

\subsection{Examples of Sylvester-Gallai configurations}\label{subsection: examples}

In this section we discuss some examples of Sylvester-Gallai configurations. Note that due to the classical Sylvester-Gallai theorem over $\C$ or by \cref{theorem: fractional linear SG}, we know that any linear SG-configuration or a $1$-linear SG-configuration has $\dim(\Kspan{\cF})\leq 3$. Therefore, if we projectivize $S_1$ and consider linear forms up to scalar multiples as points in the projective space $\P^{N-1}$, then all linear SG-configurations must be planar. We recall a few examples of such configurations below.

\begin{example}[Hesse configuration] Let $C\subset \P^2_{\C}$ be a non-singular cubic curve. 
Let $\cH$ be the set of inflection points of $C$. 
Then $|\cH|=9$ and the line joining any two inflection points passes through another inflection point. Therefore $\cH$ is a linear SG-configuration.
    
\end{example}

More generally, we have the Fermat configuration of cardinality $3n$ for $n\geq 3$.

\begin{example}[Fermat configuration]
Let $n\geq 3$ and $C=V(x^n+y^n+z^n)\subset \P^2_{\C}$. 
Let $\cF$ be the set of inflection points of $C$. Let $\omega$ be a primitive $n$-th root of $-1$. 
Then we have $\cF=\{[0:1:\omega^j],[\omega^j:0:1],[1:\omega^j:0]|j\in[n]\}$. Hence the line joining any two points in $\cF$ contains a third point from $\cF$ (see \cite[Example 1]{BDSW19}).
    
\end{example}

The following example of cardinality $12$ was constructed in \cite{KN73}.

\begin{example}[Kelly-Nwankpa configuration] Let $i=\sqrt{-1}$.
Let $a=\frac{1+i}{2}$ and $b=\frac{1-i}{2}$. Furthermore, let $p_1=[0:0:1]$, $p_2=[0:1:1]$, $p_3=[1:0:1]$, $p_4=[1:1:1]$, $p_5=[a:a:1]$, $p_6=[a:b:1]$, $p_7=[b:a:1]$ and $p_8=[b:b:1]$. 
Let $p_9,p_{10},p_{11}, p_{12}$ be intersection points of the line at infinity and the lines joining $p_1$ with $p_2,p_3,[1:i:1],[1:-i:1]$ respectively. 
Then the line joining any $p_i,p_j$ contains a third point $p_k$ and $\cF=\{p_1,\cdots,p_{12}\}$ forms a SG configuration (see \cite[Theorem 3.10]{KN73})
    
\end{example}

Now we consider examples of radical SG configurations of polynomials. Recall that for linear forms a $1$-radical SG configuration is a classical linear SG-configuration. Therefore, the previous examples are also examples of $1$-radical SG configurations in a polynomial ring $S$. Note that, if $\cF\subset S$ is a $1$-radical SG-configuration, then $\dim(\Kspan{\cF})\leq 3$.

The following is an example of a non-linear $d$-radical-SG configuration for higher degree forms in $\K[x,y,z]$.

\begin{example}
    Fix an integer $m\geq 6$. Let $F_1=x$, $F_2=y$, $F_3=xz+y^2$ and $F_4=x(xz+y^2)+y^3$. For $5\leq k\leq  m$, we define $F_k$ recursively as $F_k=F_3\cdots F_{k-1}+y^{\sum_{i=3}^{k-1} \deg(F_i)}$. 
    Note that $F_i$ is homogeneous for $i\in [m]$ and $\deg(F_m)=\Theta( 2^m)$.
    
    We have that $\cF=\{F_1,\cdots,F_m\}$ is a $d$-radical SG-configuration in $\K[x,y,z]$ where $d\Theta( 2^m)$. 
    Indeed, note that $F_i\in (x,y)\cap (y,z)$ for all $i\in [3,m]$. 
    In particular, $F_3\in\rad(F_1,F_2)=(x,y)$. Also, by induction we see that for all $i\in [3,m]$, we have $F_i $ is congruent to a non-zero scalar multiple of $ y^{\deg(F_i)}$ modulo $(x)$. Therefore $\rad(F_1,F_i)=(x,y)$ and $F_2=y\in \rad(F_1,F_i)$ for all $i\in[3,m]$. Similarly, for all $i\in [3,m]$, we have $F_i$ congruent to a non-zero scalar multiple of $x^{a_i}z^{b_i}$ modulo $(y)$ for some positive integers $a_i,b_i$. 
    Hence $\rad(y,F_i)=(x,y)\cap (y,z)$ for all $i\in [3,m]$. 
    Hence, for all $j\in [3,m]$ there exists $k\neq j$ such that $F_k\in \rad(F_2,F_j)$. Now, we consider pairs $F_i,F_j$ where $i,j\in [3,m]$. 
    Suppose $i<j$, then $F_j$ is congruent to $y^{\deg(F_j)}$ modulo $(F_i)$. Hence we have that $y\in \rad(F_i,F_j)$.

    Furthermore, note that $\deg(F_i)\neq \deg(F_j)$ for any $i\neq j\in [2,m]$. Hence, we have that $F_1,\cdots,F_m$ are linearly independent and $\dim(\Kspan{\cF})=\Omega(\mathrm{log}_2(d))$. 
\end{example}

We note that the radical SG-theorem does not hold in arbitrary quotients of polynomial rings.

\begin{example}\label{example: counterexamples in quotients}
   Let $H_i=x_ix_{i+1}-x_{i+2}^2\in S=\K[x_1,\cdots,x_{2r}]$ for $i\in [2r]$, where we denote $x_{2r+1}:=x_1$ and $x_{2r+2}:=x_2$. Consider the ideal $J=(H_1,\cdots,H_{2r})$ and the quotient $R=S/J$. Let $\cF=\{x_1,x_3,\cdots,x_{2r-1}\}\subset R$. Then $x_i,x_j$ are pairwise linearly independent for $i\neq j$. Furthermore, $x_{2i+1}\in \rad(x_{2j+1})$ for all $i,j$. Hence, for all $x_i,x_j\in \cF$ there exists $k\neq i,j$ such that $x_k\in \rad(x_i,x_j)\cap \cF$. However, $\dim(\Kspan{\cF})=r-1$.
\end{example}

\section{Algebraic  geometry}\label{sec:algebraic-geometry}

In this section we establish the necessary definitions and develop our algebraic-geometric toolkit. We refer the reader to the standard sources of \cite{AM69, Eis95, Har77, stacks-project} for in-depth discussions of the notions in this section.

\subsection{Regular, prime and \texorpdfstring{$\mathcal{R}$\textsubscript{$\eta$}}{R eta}-sequences}

\begin{definition}[Regular sequence]\label{definition: regular sequence}
Let $R$ be ring and $M$ an $R$-module. A sequence of elements $F_1, F_2, \cdots F_n \in R$ is called an $M$-regular sequence if
\begin{itemize}
\item[(1)] $(F_1,F_2,\cdots,F_n)M\neq M$, and
\item[(2)] for $i=1,\cdots, n$, $F_i$ is a non-zerodivisor on $M/(F_1,\cdots,F_{i-1})M$.
\end{itemize}   
If $M=R$, then we simply call it a regular sequence.
\end{definition}

If $P$ and $Q$ are two irreducible polynomials in the polynomial ring $S$ such that $P$ does not divide $Q$, then $P,Q$ is an $S$-regular sequence. 
Indeed, since $P$ is irreducible, we know that $S/(P)$ is a integral domain. 
Therefore $Q$ is a non-zero element in $S/(P)$, and hence a non-zero divisor.

If $F_1,\cdots,F_n$ is a regular sequence of forms in $S$, then $F_1,\cdots,F_n$ are algebraically independent. 
Thus the subalgebra generated by $F_1,\cdots,F_n$ is isomorphic to a polynomial ring. 
In particular, the homomorphism $k[y_1,\cdots,y_n] \rightarrow S$ defined by $y_i\mapsto F_i$ is an isomorphism onto its image.

Even though the $\K$-algebra $\K[F_1, \ldots, F_n] \subset S$ is isomorphic to a polynomial ring, its elements may not ``behave well'' when seen as elements of $S$.
We next present a sufficient condition which will ensure that the subalgebra is well behaved with respect to $S$, which we describe in \cref{subsection: intersection flatness}.

\begin{definition}[Serre's $\mathcal{R}_\eta$-property] 
Let $\eta \in \N$. 
We say that a Noetherian ring $R$ satisfies the $\mathcal{R}_\eta$ property if the local ring $R_{\mathfrak{p}}$ is a regular local ring for all prime ideals $\mathfrak{p}\subset R$ such that $\mathrm{height}(\mathfrak{p})\leq \eta$.
\end{definition}

\begin{definition}[Prime and $\cR_\eta$-sequences] 
Let $\eta \in \N$ and $R$ a Noetherian ring.
A sequence of elements $F_1, \dots,F_n \in R$ is called a prime sequence (respectively an $\mathcal{R}_\eta$-sequence) if
\begin{enumerate}
    \item $F_1,\cdots,F_n$ is a regular sequence, and 
    \item  $R/(F_1,\cdots,F_i)$ is an integral domain (respectively, satisfies the $\mathcal{R}_\eta$ property) for all $i \in [n]$. 
\end{enumerate}
\end{definition}

\begin{remark}\label{remark: relation reta prime}
Since prime sequences and $\mathcal{R}_\eta$-sequences are regular sequences, we know that if $F_1,\cdots,F_n$ is a prime sequence or $\mathcal{R}_\eta$-sequence in the polynomial ring $S$, then $F_1,\cdots,F_n$ are algebraically independent. 
\end{remark}

We note the following simple statement about radicals and regular sequences.

\begin{lemma}\label{lemma: radical regular sequence} Let 
$R$ be an integral domain. Then we have the following.

\begin{enumerate}
    \item If $F,P_1,P_2$ is a regular sequence in $R$ and $G\in R$, then $FG\in\rad{(FP_1,FP_2)}$ iff $G\in \rad{(P_1,P_2)}$.
    \item Suppose $R$ is a UFD. Let $F,G\in R$ be non-units such that $(F,G)\neq R$. Then, $F,G$ is a regular sequence in $R$ if and only if $F,G$ do not have any common factors in $R$.
\end{enumerate}

\end{lemma}
\begin{proof}
1. If $G\in \rad(P_1,P_2)$, then $FG\in \rad(FP_1,FP_2)$. 
Conversely, if $FG\in \rad(FP_1, FP_2)$, then we have $FG \in \rad(P_1, P_2)$, since $\rad(FP_1, FP_2) \subset \rad(P_1, P_2)$. 
Hence there is $D \in \N^*$ such that $F^DG^D \in (P_1, P_2)$.
Since $F$ is a non-zero divisor in $R/(P_1,P_2)$, we conclude that $G^D \in (P_1,P_2)$.

2. Since $R$ is a domain, we have that $F$ is a non-zero divisor in $R$. If $FP=GQ$ for some $P,Q\in R$. Then by unique factorization we must have $F\mid Q$, if $F,G$ do not have any common factors. Therefore, $F,G$ is a regular sequence in $R$. Conversely, if $F=PF_1$ and $G=PG_1$ for some irreducible factor $P$, then we have $GF_1=G_1F\in (F)$. However, $F\not\mid F_1$, as $P$ is a non-unit.
\end{proof}

\begin{lemma}\label{lemma: d^2 minimal primes basic} Let $S=\K[x_1,\cdots,x_N]$ be a polynomial ring and $z$ be a new variable. For any polynomial $F\in S$, let $F^*\in S[z]$ denote the homogenization of $F$ with respect to $z$. Then we have the following.

\begin{enumerate}
    \item If $F\in S$ is irreducible in $S$, then $F^*$ is irreducible in $S[z]$.
    \item Let $P,Q\in S$ be non-constant polynomials. If $P,Q$ do not have any common factors in $S$, then $P^*,Q^*$ do not have any common factors in $S[z]$.
    \item Let $P,Q\in S$ such that $P,Q$ is a regular sequence in $S$. Then the number of minimal primes over $(P,Q)$ in $S$ is at most $\deg(P)\deg(Q)$. 
\end{enumerate}

\end{lemma}

\begin{proof}
1. For any homogeneous polynomial $G\in S[z]$, let $G_*\in S$ be the dehomogenization of $G$, obtained by setting $z=1$. Note that $z$ does not divide $F^*$. If $F^*=GH$ in $S[z]$, then $G,H$ are both homogeneous. By \cite[Section 2.6]{Ful89} we obtain $F=(F^*)_*=G_*H_*$. Since $F$ is irreducible, we may assume that $G_*=F$ and $H_*\in \K$. This is a contradiction, since $H=(H_*)^*$, as $z$ does not divide $H$.

2. Let $P=\Pi_if_i$ and $Q=\Pi_j g_j$ be irreducible factorizations of $P,Q$ in $S$. Then by \cite[Section 2.6]{Ful89}, we have that $P^*=\Pi_i f_i^*$ and $Q^*=\Pi_j g_j^*$. By part (1), we note that $f_i^*,g_j^*$ are irreducible in $S[z]$. If $P^*,Q^*$ have a common factor, say $f_i^*=g_j^*$ for some $i,j$. Then we have $f_i=(f_i^*)_*=(g_j^*)_*=g_j$, which is a contradiction.

3. By \cref{lemma: radical regular sequence}, we know that $P,Q$ do not have any common factors. Hence, by part (2), we conclude that $P^*,Q^*$ also do not have any common factors in $S[z]$. Since $P^*,Q^*$ are homogeneous of positive degree, we see that $P^*,Q^*$ is a regular sequence in $S[z]$ by \cref{lemma: radical regular sequence}. Note that the minimal primes over $(P,Q)$ in $S$ are in one-to-one correspondence with the irreducible components of $V(P,Q)\subset \A^N$. Furthermore, the irreducible components of $V(P,Q)$ are in one-to-one correspondence with the irreducible components of $V(P^*,Q^*)
\subset \P^{\dim(W)}$ which are not contained in the hyperplane $V(z)$. Now the total number of  irreducible components of $V(P^*,Q^*)$ is at most $\deg(V(P^*,Q^*))$. Since $P^*,Q^*$ is a homogeneous regular sequence in $S[z]$, we know that $\deg(V(P^*,Q^*))=\deg(P^*)\deg(Q^*)=d^2$, by Bezout's theorem \cite[Proposition 8.4]{Ful84}.
\end{proof}

\subsection{Homogeneous system of parameters and Cohen-Macaulay rings}

In this subsection, we recall the standard definitions and properties of homogeneous system of parameters and Cohen-Macaulay rings.

\begin{definition}\label{definition: hsop}
Let $R=\oplus_{i\in \N} R_i$ be a finitely generated $\N$-graded $\K$-algebra such that $R_0=\K$. 
A sequence of homogeneous elements $F_1,\cdots,F_n\in R$ is a homogeneous system of parameters, denoted by \hsop\!, if $n=\dim(R)$ and $\dim(R/(F_1,\cdots,F_n))=0$. 
\end{definition}

If $R$ such a graded $\K$-algebra, an \hsop always exists. 
In particular, more is true:

\begin{theorem}\label{theorem: hsop}
Let $R=\oplus_{i\in \N} R_i$ be a finitely generated $\N$-graded $\K$-algebra such that $R_0=\K$. 
Suppose the Krull dimension of $R$ is $\dim(R)=n$ and let $\fm=\oplus_{i\geq1}R_i$ be the homogeneous maximal ideal. 
An \hsop always exists. Moreover, for any sequence of forms $F_1,\cdots,F_n$ of positive degree in $R$, the following are equivalent:

\begin{enumerate}
    \item $F_1,\cdots,F_n$ is a homogeneous system of parameters.
    \item $\fm$ is nilpotent modulo $(F_1,\cdots,F_n)$.
    \item $R/(F_1,\cdots,F_n)$ is finite dimensional as a $\K$-vector space.
    \item $R$ is module-finite over the subring $A=\K[F_1,\cdots,F_n]$, i.e. $R$ is a finitely generated $A$-module.
\end{enumerate}
Moreover, whenever these conditions hold, $F_1,\cdots,F_n$ are algebraically independent over $\K$, so that $\K[F_1,\cdots,F_n]$ is a polynomial ring.
\end{theorem}

\begin{definition}[Cohen-Macaulay]\label{definition: CM}
A local ring $(R,\fm)$ of $\dim(R)=n$, is Cohen-Macaulay if there exists a regular sequence $F_1,\cdots,F_n\in \fm$ of length $n$. A Noetherian ring $R$ is Cohen-Macaulay if $R_\fp$ is a Cohen-Macaulay for all maximal (equivalently prime) ideals $\fp$ in $R$.
\end{definition}

We have the following characterization of Cohen-Macaulay graded $\K$-algebras.

\begin{theorem}\label{theorem: CM hsop}
Let $R=\oplus_{i\in \N} R_i$ be a finitely generated $\N$-graded $\K$-algebra such that $R_0=\K$. Suppose  $\dim(R)=n$ and let $\fm=\oplus_{i\geq1}R_i$ be the homogeneous maximal ideal. The following conditions are equivalent:
\begin{enumerate}
    \item Some homogeneous system of parameters is a regular sequence.
    \item Every homogeneous system of parameters is a regular sequence.
    \item For some \hsop $F_1,\cdots,F_n$, the ring $R$ is a module-finite free-module over $\K[F_1,\cdots,F_n]$.
    \item  For every \hsop $F_1,\cdots,F_n$, the ring $R$ is a module-finite free-module over $\K[F_1,\cdots,F_n]$.
    \item $R$ is Cohen-Macaulay.
\end{enumerate}
\end{theorem}

\begin{proposition}\label{proposition: hsop injection}
Let $R$ be a Cohen-Macaulay $\N$-graded $\K$-algebra, $x_1,\cdots,x_n$ be an \hsop for $R$  and $\cA := \K[x_1, \ldots, x_n] \subset R$.
Then there exist homogeneous elements $u_1,\cdots,u_r\in R$ that generate $R$ as a free $\cA$-module and we have an isomorphism of $\cA$-modules $R\simeq \cA^{\oplus r}$ such that the composition $\cA\subset R\simeq \cA^{\oplus r}$ is an isomorphism onto one of the direct summands of $\cA^{\oplus r}$.
\end{proposition}

\begin{proof}
Let $\fm=(x_1,\cdots,x_n)$. 
By \cref{theorem: hsop}, $V=R/\fm R$ is a finite dimensional graded $\K$-vector space. 
Let $u_1,\cdots,u_r\in R$ be homogeneous lifts of a homogeneous basis of $V$. 
Then, $u_1,\cdots,u_r$ form a free basis of $R$ as an $\cA$-module. 
Hence we have an isomorphism $R\simeq \cA^{\oplus r}$ given by $u_i\mapsto e_i$. 
Furthermore, we may assume that $u_1=1\in R$, therefore the inclusion $\cA\subset R$ is an isomorphism onto a direct summand of $\cA^{\oplus r}$.
\end{proof}

\begin{proposition}\label{proposition: strong sequence CM UFD}
Let $F_1,\cdots,F_n$ be a sequence of forms in $S=\K[x_1,\cdots,x_N]$. 
\begin{enumerate}
    \item If $F_1,\cdots, F_r$ is a regular sequence in $S$, then $S/(F_1,\cdots,F_r)$ is a Cohen-Macaulay ring.
    \item If $F_1,\cdots,F_n$ is a prime sequence, then $S/(F_1,\cdots,F_r)$ is a Cohen-Macaulay domain for all $r\in [n]$. 
    The images of $F_{r+1},\cdots,F_n$ in $S/(F_1,\cdots,F_r)$ form a prime sequence.
    \item Let $F_1,\cdots,F_n\in S$ be an $\mathcal{R}_\eta$-sequence of forms where $\eta\geq 3$ . 
    Then $S/(F_1,\cdots,F_r)$ is a Cohen-Macaulay UFD for all $r\in [n]$. 
    The images of $F_{r+1},\cdots,F_n$ in $S/(F_1,\cdots,F_r)$ is an $\mathcal{R}_\eta$-sequence.
\end{enumerate}

\end{proposition}

\begin{proof}
1. Since $F_1,\cdots,F_r$ is a regular sequence in $S$, we have that $\mathrm{ht}(F_1,\cdots,F_r)=N-r$. Since, $S$ is Cohen-Macaulay, we conclude that $S/(F_1,\cdots,F_r)$ is Cohen-Macaulay by \cite[Proposition 18.13]{Eis95}. 

2. Since primes sequences in $S$ are regular sequences, we have that $S/(F_1,\cdots,F_i)$ is Cohen-Macaulay. By definition of prime sequences, we have that $S/(F_1,\cdots, F_i)$ is a domain for all $i\in [m]$. Therefore, the images of $F_{r+1},\cdots,F_m$ in $S/(F_1,\cdots,F_r)$ form a prime sequence.

3. Similarly, if $F_1,\cdots,F_m$ is an $\cR_\eta$-sequence, then the images of $F_{r+1},\cdots,F_m$ in $S/(F_1,\cdots,F_r)$ form a $\cR_\eta$-sequence. If $F_1,\cdots,F_m$ is an $\mathcal{R}_\eta$-sequence, we note that the quotient $S/(F_1,\cdots,F_m)$ satisfies the $R_3$-property. Therefore, $S/(F_1,\cdots,F_m)$ is an UFD by \cite[Corollary 1.5]{AH20b}.
\end{proof}

\begin{proposition}\label{proposition: CM hsop and degree}
Let $R=\oplus_{i\in \N} R_i$ be a finitely generated $\N$-graded $\K$-algebra such that $R_0=\K$. 
Suppose $R$ is generated by $R_1$ as a $\K$-algebra and $R$ is Cohen-Macaulay. 
Let $\dim(R)=n$.  
Then we have:
\begin{enumerate}
    \item Every regular sequence $F_1,\cdots,F_r$ of homogeneous elements can be extended to a \hsop  $F_1,\cdots,F_n$ where $\deg(F_i)=1$ for $r+1\leq i\leq n$.
    \item If $F_1,\cdots,F_n$ is a \hsop in $R$, then for any $r\leq n-1$, we have $R/(F_1,\cdots,F_r)$ is Cohen-Macaulay and the images of $F_{r+1},\cdots,F_n$ in $R/(F_1,\cdots,F_r)$ form a \hsop
    \item Suppose $G_1,\cdots,G_t$ is a regular sequence of forms in $S$ and $R=S/(G_1,\cdots,G_t)$. 
    Let $F_1,\cdots,F_r\in R$ be a regular sequence of homogeneous elements in $R$, and $F_1,\cdots,F_n\in R$ an extension to a \hsop such that $\deg(F_i)=1$ for $r+1\leq i\leq n$. 
    Then the rank of $R$ as a free module over $\K[F_1,\cdots,F_n]$ is given by $\Pi_{j=1}^t\deg(G_j)\Pi_{i=1}^r\deg(F_i)$.
\end{enumerate}
\end{proposition}

\begin{proof}
1. Since $F_1,\cdots,F_r$ is a regular sequence, we have that $\mathrm{ht}(F_1,\cdots,F_r)=n-r$. Therefore, $R'=R/(F_1,\cdots,F_r)$ is Cohen-Macaulay by \cite[Proposition 18.13]{Eis95}.
If $r=n$, then we are done. 
Otherwise we proceed by induction. 
The homogeneous maximal ideal of $R$ is given by $\fm=(R_1)$, and hence the homogeneous maximal ideal of $R'$ is $(R'_1)$. 
If $\fm'=(R'_1)$ is a minimal prime of $R'$, then $\dim(R')=\mathrm{ht}(\fm')=0$, and hence $F_1,\cdots, F_r$ is a homogeneous system of parameters in $R$. Otherwise, by homogeneous prime avoidance, there exists a homogeneous element $f_{r+1}\in \fm'$ which is not contained in any minimal prime of $R'$, and hence $f_{r+1}\in R'_1$ is a nonzero divisor. Let $F_{r+1}\in R_1$ be a homogeneous lift of $f_{r+1}$. Then $F_1,\cdots,F_{r+1}$ is a regular sequence and by induction we can extend to a homogeneous system of parameters  $F_1,\cdots,F_{r+1},\cdots, F_n$ such that $\deg(F_i)=1$ for $r+1\leq i\leq n$. 

2. Let $R'=R/(F_1,\cdots,F_r)$. Since, $R$ is Cohen-Macaulay, we know that $F_1,\cdots,F_r$ is a regular sequence. Hence $R'$ is Cohen-Macaulay and $\dim(R')=n-r$. Also, we have $R'/(\overline{F}_{r+1},\cdots,\overline{F}_n)\simeq R/(F_1,\cdots,F_n)$, thus the images $\overline{F}_{r+1},\cdots,\overline{F}_n$ are a homogeneous system of parameters in $R'$, by \cref{theorem: hsop}.

3. Note that we may lift $F_1,\cdots,F_n\in R$ to $\widetilde{F}_1,\cdots,\widetilde{F}_n\in S$ such that $\deg(\widetilde{F}_i)=1$ for $r+1\leq i\leq n$ and we have $G_1,\cdots,G_t,\widetilde{F}_1,\cdots,\widetilde{F}_n$ is a homogeneous system of parameters for $S$. Now rank of $R$ as a free $\K[F_1,\cdots,F_n]$-module is given by the $\K$-vector space dimension of $R/(F_1,\cdots,F_n)$. Now $R'=R/(F_1,\cdots,F_n)\simeq S/(G_1,\cdots,G_t,\widetilde{F}_1,\cdots,\widetilde{F}_n)$, which is of Krull dimension $0$. Now we have $\dim_{\K}(R')=\Pi_{j=1}^t\deg(G_j)\Pi_{i=1}^r\deg(F_i)$.
\end{proof}

Suppose that $u,v,u_1,\cdots,u_m\in S_1$ are pairwise linearly independent linear forms such that $v\in \Kspan{u,u_i}$ for all $i\in [m]$. Then we must have that $\dim(\Kspan{u,v,u_1,\cdots,u_m})=2$. In particular, if $m\geq 2$, then we have $\Kspan{u,u_i}=\Kspan{u,u_j}$ for all $i,j$. If $P,Q,P_1,\cdots,P_m\in S$ are pairwise non-associate forms of higher degree such that $Q\in \rad(P,P_i)$ for all $i\in [m]$, then it is no longer necessary  that $\rad(P,P_i)=\rad(P,P_j)$ for all $i,j$. However, the following lemma shows that if $m$ is sufficiently large, then there exist two distinct forms $P_i,P_j$ such that $\rad(P,P_i)=\rad(P,P_j)$. Note that, in \cref{lemma: finiteness of radicals}, the bound on $m$ depends on both the degree $d$ of the forms and the parameter $t$, which determines the dimension of the ring $R$. In \cref{lemma: d^2 minimal primes}, we will prove a strengthening for strong quotients of polynomial rings, where the bound on $m$ is independent of $t$, and depends only on the degree $d$ of the forms.

\begin{lemma}\label{lemma: finiteness of radicals}

Let $G_1,\cdots,G_t$ be a regular sequence of forms of degree at most $d$ in $S=\K[x_1,\cdots,x_N]$ and $R=S/(G_1,\cdots,G_t)$. Suppose $R$ is a UFD. Let $P,Q,P_1,\cdots,P_m\in R_{\leq d}$ be  homogeneous elements of positive degree where $m\geq 2^{d^{t+2}}$ and $\gcd(P,Q)=\gcd(P,P_i)=\gcd(Q,P_i)=1$ for all $i\in[m]$. If $Q\in \rad(P,P_i)$ for all $i\in [m]$, then we must have $\rad{(P,P_i)}=\rad(P,P_k)$ for two distinct $i,k\in [m]$.
\end{lemma}
\begin{proof}
Note that $P,Q$ and $P,P_i$ are regular sequences in $R$.
We have  $\rad(P,Q)\subset \rad(P,P_i)$ for all $i\in [m]$, since $Q\in \rad(P,P_i)$. Let $\mathcal{S}=\{\mathfrak{p}_1,\cdots\mathfrak{p}_\ell\}$  and $\mathcal{S}_i=\{\mathfrak{p}_{i1},\cdots, \mathfrak{p}_{i\ell_i}\}$ be the set of minimal primes over $(P,Q)$ and $(P,P_i)$ in $R$, respectively. We have $\bigcap \mathfrak{p}_j=\rad(P,Q)\subset \rad(P,P_i)=\bigcap_j \mathfrak{p}_{ij}$. Therefore, by \cite[Proposition 1.11]{AM69}, we must have that for all $i,j$, there exists $k$ such that $\mathfrak{p}_k\subset \mathfrak{p}_{ij}$. Since we know that $\mathrm{ht}(\mathfrak{p}_j)=\mathrm{ht}(\mathfrak{p}_{ij})=2$ for all $i,j$, we must have that for all $i,j$, there exists $k$ such that $\mathfrak{p}_k= \mathfrak{p}_{ij}$, i.e. $\mathcal{S}_i\subset \mathcal{S}$ for all $i \in [m]$. 

Note that $G_1,\cdots,
G_t,P,Q$ is a regular sequence in $S$. Furthermore, the minimal primes $\fp_j$ over $(P,Q)$ in $R$ are in one-to-one correspondence with the minimal primes $\widetilde{\fp}_k$ over $(G_1,\cdots,G_t,P,Q)$ in $S$. We have \[\sum_jm(\wt{\mathfrak{p}}_j)e(S/\wt{\mathfrak{p}}_j)=e(S/(G_1,\cdots,G_t,P,Q))=\deg(P)\deg(Q)\Pi_{i=1}^t \deg(G_i)\leq d^{t+2}.\] 
Since $m(\wt{\mathfrak{p}}_j)\geq 1$ and $ e(S/\wt{\mathfrak{p}}_j)\geq 1$, we conclude that $|\mathcal{S}|\leq d^{t+2}$, i.e. there exist at most $d^{t+2}$ minimal primes of $(P,Q)$. 
Therefore there are at most $2^{d^{t+2}}-1$ number of distinct choices for the set of minimal primes $\mathcal{S}_i=\{\mathfrak{p}_{ij}\}$ of the ideals $(P,P_i)$. 
By the pigeonhole principle, there exist distinct $i,k\in [m]$ such that $\mathcal{S}_i=\mathcal{S}_k$. 
Thus, $\rad(P,P_i)=\bigcap_{\mathfrak{q} \in \mathcal{S}_i}\mathfrak{q} = \bigcap_{\mathfrak{p}\in \mathcal{S}_k}\mathfrak{p} =\rad(P,P_k)$.
\end{proof}

\subsection{Intersection flatness and prime ideals}\label{subsection: intersection flatness}

Recall that if $R\rightarrow R'$ is a flat ring homomorphism, then $IR'\bigcap JR'=(I\bigcap J)R'$ for any two ideals $I,J$ in $R$. 
The notion of intersection flatness generalizes this fact to arbitrary intersections.
A flat ring homomorphism $R\rightarrow R'$ is called intersection flat if for every family $\mathcal{I}$ of ideals in $R$, we have $\bigcap_{I\in \mathcal{I}}(IR')=(\bigcap_{I\in \mathcal{I}}I)R'$. 
The following result is from the discussion on extensions of prime ideals in in \cite[Section 2]{AH20a}. 
We rephrase it here in the form of the following proposition for convenience.

\begin{proposition}[Intersection flatness]\label{proposition: intersection flatness}
Let $R$ be a finitely generated $\N$-graded $\K$-algebra with $R_0=\K$. Suppose $R$ is Cohen-Macaulay. Let $F_1,\cdots,F_m$ be a regular sequence of homogeneous elements in $R$. 
Then $\K[F_1,\cdots,F_m]\subset R$ is a free extension of rings. Furthermore, $\K[F_1,\cdots,F_m] \subset R$ is an intersection flat extension of rings.
\end{proposition}

\begin{proof}
Since $F_1,\cdots, F_m$ is a regular sequence, we may extend to a homogeneous system of parameters $F_1,\cdots,F_n$, by \cref{proposition: CM hsop and degree}.
Now, as $R$ is Cohen-Macaulay, we have that $\K[F_1,\cdots,F_n]\subset R$ is a free extension, by \cref{theorem: CM hsop}. 
Also, $\K[F_1,\cdots,F_m]\subset \K[F_1,\cdots,F_n]$ is a free extension. 
Therefore the composition $\K[F_1,\cdots,F_m]\subset R$ is a free extension. 
By \cite[Page 41]{HH94}, free extensions are intersection flat.
\end{proof}

\emph{Hilbert Rings}. We say that $R$ is a Hilbert ring if every prime ideal is an intersection of maximal ideals. 
The polynomial ring $S$ and any finitely generated $\K$-algebra $R=S/I$ are Hilbert rings.

We recall the following result, from \cite[Theorem 2.7]{AH20a}.

\begin{theorem}\label{theorem: AH hilbert ring}
Let $A$ be a Noetherian Hilbert ring and let $R\supseteq A$ be a Noetherian $A$-algebra that is intersection flat over $A$. Suppose that for every maximal ideal $\fm$ of $A$, we have $R/\fm R$ is an integral domain. Then for every prime ideal $\fp$ of $A$, we have that $R/\fp R$ is an integral domain, i.e. $\fp R$ is a prime ideal of $R$.
\end{theorem}

In particular, we have the following corollary (\cite[Corollary 2.9]{AH20a}): if $g_1, \cdots, g_s\in S$ is a prime sequence, then prime ideals in $\K[g_1,\cdots,g_s]$ extend to prime ideals in $S$.

\begin{corollary}\label{lemma: extension of prime}
Let $S$ be our polynomial ring and $g_1,\cdots,g_s$ be a prime sequence of forms in $S$. 
Then for any prime ideal $\mathfrak{p}\subset \K[g_1,\cdots,g_s]$, the extension ideal $\mathfrak{p}S$ is also prime.
\end{corollary}

We prove that this result can be generalized to certain quotients of $S$.

\begin{proposition}\label{proposition: extension of primes general}
Let $G_1,\cdots, G_t$ be a prime sequence of forms in $S$. Let $R=S/(G_1,\cdots,G_t)$. Let $F_1,\cdots,F_m$ be a prime sequence in $R$. Then for any prime ideal $\fp\subset \K[F_1,\cdots,F_m]$, the extension ideal $\fp R$ is a prime ideal in $R$.
\end{proposition}

\begin{proof}
Note that, $R$ is a finitely generated graded $\K$-algebra, with $R_0=\K$. 
Also, $R$ is generated by $R_1$ as a $\K$-algebra. 
By \cref{proposition: strong sequence CM UFD} we have $R$ is a Cohen-Macaulay integral domain. 
By \cref{proposition: intersection flatness}, we have that  $A:=\K[F_1,\cdots,F_m]\subset R$ is intersection flat. 
Therefore, by \cref{theorem: AH hilbert ring}, it is enough to show that $\fm R$ is a prime ideal for every maximal ideal $\fm$ of $A$. 
Since, $A$ is isomorphic to a polynomial ring, we know that $\fm=(F_1-c_1,\cdots,F_m-c_m)$ for some $(c_1,\cdots,c_m)\in \K^m$. 
Now, $F_1,\cdots,F_m$ is a regular sequence of homogeneous elements in $R$ and the ideal $(F_1,\cdots,F_m)$ is a prime ideal in $R$. 
Then, by \cite[Proposition 2.8]{AH20a}, we conclude that $(F_1-c_1,\cdots,F_m-c_m)R$ is a prime ideal in $R$ for all $(c_1,\cdots,c_m)\in \K^m$. 
Therefore $\fm R$ is a prime ideal in $R$ for all maximal ideals $\fm$ of $A$, and we are done.
\end{proof}

\begin{lemma}\label{lem:factor-ufd-algebra}Let $G_1,\cdots, G_t$ be an $\cR_\eta$-sequence of forms in $S$, such that $\eta\geq 3$. Let $R=S/(G_1,\cdots,G_t)$. Let $F_1,\cdots,F_m$ be a prime sequence in $R$.
Let $A=\K[F_1,\cdots,F_m]$. 

\begin{enumerate}
    
    \item Let $Q\in A$. Then $Q$ is irreducible in $A$ iff $Q$ is irreducible in $R$.
    \item Let $P_1, \ldots, P_k \in R$ be irreducible elements and $d_i \in \N_+$ such that $ \prod_{i=1}^k P_i^{d_i} \in A$. Then $P_i \in A$ for $i \in [k]$.
\end{enumerate}

\end{lemma}

\begin{proof}
1. By \cref{proposition: strong sequence CM UFD}, we know that $R$ is a Cohen-Macaulay UFD. 
Since a prime sequence is algebraically independent, we know that $A$ is isomorphic to a polynomial ring, and hence an UFD. Note that the units in $A$ and $R$ are only elements of $\K$. Let $Q\in A$. If $Q=H_1H_2$ is a factorization of $Q$ into non-units in $A$, then it is also a factorization into non-units in $R$. Therefore, if $Q$ is irreducible in $R$, then it is irreducible in $A$. Conversely, suppose $Q$ is irreducible in $A$. Then $(Q)$ is a prime ideal in $A$. By \cref{proposition: extension of primes general}, we have that $(Q)\cdot R$ is a prime ideal in $R$. 
Therefore, $Q$ is irreducible in $R$, as $R$ is a UFD.

2. Let $F = \prod_{i=1}^k P_i^{d_i}$.
Since  $A$ is a UFD and $F \in A$, we have that $F = \prod_{i=1}^t Q_i^{e_i}$, where each $Q_i \in A$ is irreducible.
Therefore, by part (1), we have that $Q_i$ is irreducible in $R$ for all $i\in [k]$. 
Hence, $F = \prod_{i=1}^t Q_i^{e_i}$ is an irreducible factorization of $F$ in $R$. 
Since $R$ is a UFD, by uniqueness of irreducible factorization, we must have that $P_i = Q_i$ and $d_i = e_i$ (possibly after a permutation of the indices). 
Hence $P_i\in A$ for all $i\in[k]$.
\end{proof}

Note that, in \cref{lem:factor-ufd-algebra}, it is necessary to assume that $A$ is generated by a prime sequence. 
For example, consider the subalgebra $A=k[xy]\subset k[x,y]$. 
Now $A$ is isomorphic to a polynomial ring in one variable and hence an UFD. However $xy\in A$ but $x,y\not\in A$.

\begin{proposition}\label{proposition: radicals and prime sequences}
Let $G_1,\cdots, G_t$ be a prime sequence of forms in $S$ and $R=S/(G_1,\cdots,G_t)$. 
Let $F_1,\cdots,F_m$ be a prime sequence in $R$ and $A:=\K[F_1,\cdots,F_m]\subset R$. 
Let $I\subset A$ be an ideal.

\begin{enumerate}
    \item For any ideal $J\subset A$ we have $JR\cap A=J$.
    \item  If $\fp\subset A$ is a minimal prime over $I$ in $A$, then $\fp R$ is a minimal prime over $IR$ in $R$. 
    Conversely, for any minimal prime $\fq\subset R$ over $IR$ in $R$, there is a minimal prime $\fp\subset A$ over $I$ in $A$ such that $\fq=\fp R$. \item We have
$$ \rad_R(IR) \cap A = \rad_{A}(I) \ \text{ and } \ \rad_{A}(I)\cdot R = \rad_R(IR). $$
where $\rad_R(IR)$ denotes the radical of the ideal $IR$ in the ring $R$, and $\rad_A(I)$ is the radical of the ideal $I$ in the ring $A$.
\item $I$ is radical in $A$ iff $IR$ is radical in $R$.
\item We have that $\fq_0\subsetneq \cdots \subsetneq \fq_h$ is a strict chain of prime ideals in $A$ iff the extension $\fq_0 S\subsetneq \cdots \subsetneq \fq_h S$ is a strict chain of prime ideals in $S$. 
\end{enumerate} 
\end{proposition}

\begin{proof}
1. Note that $R$ is a Cohen-Macaulay domain. Since $F_1,\cdots,F_m$ is a regular sequence, we know that $A\subset R$ is a free extension. Hence, $A\subset R$ is a faithfully flat extension. Therefore, for any ideal $J$ of the ring $A$, we have that the contraction of the ideal $JR$ to the ring $A$ is $I$ itself. Thus $IR\cap A=I$.

2. Let $\fp\subset A$ be a minimal prime over the ideal $I$ in $A$. Then, $\fp R$ is a prime ideal in $R$, by \cref{proposition: extension of primes general}. Let $IR\subseteq\fq\subseteq \fp R$ be a minimal prime over $IR$ in $R$. Then $I=IR\cap A\subseteq \fq\cap A\subseteq \fp R\cap A=\fp$. Since $\fp$ is a minimal prime, we see that $\fq\cap A=\fp$ and hence $\fp R\subseteq \fq$. Therefore $\fq=\fp R$ and $\fp R$ is a minimal prime over $IR$ in $R$.

Conversely, let $IR \subseteq\fq \subset R $ be a minimal prime over $IR$ in $R$. Then $I=IR\cap A\subseteq\fq \cap A$. Hence there exists a minimal prime $I\subseteq\fp$ in $A$ such that $\fp \subseteq \fq \cap A$. Now, $\fp R$ is a prime ideal in $R$. We have $IR \subseteq\fp R\subseteq (\fq\cap A)R\subseteq \fq$. As $\fq$ is a minimal prime over $IR$ in $R$, we conclude that $\fq =\fp R$.

3. We have $\rad_R(IR)=\cap_{i=1}^a\fq_i$, where $\fq_i$ are the minimal primes over $IR$ in $R$. 
Similarly, we have $\rad_A(I)=\cap_{j=1}^b\fp_j$ where $\fp_j$ are the minimal primes over $I$ in $A$. 
Now, by part (2), we know that for each $i\in [a]$, there exists $j\in [b]$ such that $\fq_i=\fp_j R$. 
We also have that for each $j\in [b]$, there exists $i\in [a]$ such that $\fp_j R=\fq_i$. 
Hence we have $\{\fq_i|i\in[a]\}=\{\fp_j R|j\in [b]\}$. 
Therefore, $\rad_R(IR)\cap A=(\cap_{i=1}^a\fq_i)\cap A=(\cap_{j=1}^b\fp_j R)\cap A=\cap_{j=1}^b(\fp_j R\cap A)=\cap_{j=1}^b\fp_j=\rad_A(I)$. 
Conversely, we have $\rad_A(I)\cdot R=(\cap_{j=1}^b\fp_j)\cdot R=\cap_{j=1}^b(\fp_j R)=\cap_{i=1}^a\fq_i=\rad_R(IR)$, where the second equality holds as $A\subset R$ is flat.

4. If $I$ is radical in $A$, then $\rad_R(IR)=\rad_A(I)\cdot R=IR$. Conversely, if $IR$ is radical in $R$, then $\rad_A(I)=\rad_R(IR)\cap A=IR\cap A=I$.

5. By part (1) we know that $\fq_i S\cap A=\fq_i$. Thus, if $\fq_i S$ is a prime then $\fq_i$ is a prime. Conversely, by \cref{lemma: extension of prime}, we know that if $\fq_i$ is a prime in $A$ then $\fq_i S$ is a prime in $S$ for all $i$. If $\fq_i S=\fq_{i+1} S$ for some $i$, then we have $\fq_i=(\fq_i S)\cap A=(\fq_{i+1} S)\cap A=\fq_{i+1}$ which is a contradiction. Conversely, if $\fq_i=\fq_{i+1}$ then $\fq_i S=\fq_{i+1} S$. Therefore, we have that $\fq_0\subsetneq \cdots \subsetneq \fq_h$ is a strict chain of prime ideals in $A$ iff the extension $\fq_0 S\subseteq \cdots \subsetneq \fq_h S$ is a strict chain of prime ideals in $S$.  
\end{proof}

\subsection{A primality criterion}

Recall that an ideal $I\subset S$ is prime iff $S/I$ is an integral domain iff the affine scheme $\mathrm{Spec}(S/I)$ is irreducible and reduced. Also, $I$ is radical iff $\mathrm{Spec}(S/I)$ is reduced. The following lemma provides a method of showing that a Cohen-Macaulay affine scheme $Y$ is irreducible or reduced, if there is a suitable morphism $\pi:Y\rightarrow X$ to an affine scheme $X$, such that $X$ and general fibers of $\pi$ are irreducible or reduced. We will use this to prove a primality criterion in \cref{lemma: irreducible prime} and a critierion for reducedness in \cref{lemma: CM discriminant radical}.

\begin{lemma}\label{lemma: property of fibers}
Let $\phi: \cA \rightarrow \cB$ be a homomorphism of finitely generated $\K$-algebras. Let $ Y=\mathrm{Spec}(\cB)$, $X=\mathrm{Spec}(\cA)$ and $\pi:Y\rightarrow X$ be the corresponding morphism of affine schemes. Suppose that every irreducible component of $Y$ dominates some irreducible component of $X$. 
\begin{enumerate}
    \item Suppose that $X$ is irreducible and reduced, i.e. $\cA$ is an integral domain. If $\pi^{-1}(x)$ is irreducible for a general closed point $x\in X$, then $Y$ is irreducible.
    \item Suppose $X$ is reduced and $Y$ satisfies Serre's property $\cS_1$, i.e. $Y$ does not have embedded primes. If for every irreducible component $W$ of $X$, we have that $\pi^{-1}(x)$ is reduced for a general closed point $x\in W$, then $Y$ is reduced.
\end{enumerate}
\end{lemma}
\begin{proof}
(1) Since $\cA\rightarrow \cB$ is a homomorphism of finitely generated $\K$-algebras, the morphsim $\pi$ is of finite type. Therefore, by \cite[\href{https://stacks.math.columbia.edu/tag/0554}{Tag 0554}]{stacks-project}, the generic fiber of $\pi$ is irreducible. Since there is bijection between the irreducible components of $Y$ that dominate $X$ and the irreducible components of the generic fiber \cite[Chapter 0, 2.1.8]{Gro60}, we conclude that $Y$ is irreducible. \\

(2) Since $\cB$ satisfies the $\cS_1$-property, in order to show that $\cB$ is reduced, it is enough to show that $\cB$ is generically reduced, i.e. $\cB_{\mathfrak{p}}$ is reduced for any minimal prime $\mathfrak{p}$ of $\cB$.

Note that the minimal primes $\mathfrak{p}$  of $\cB$ correspond to the generic points of the irreducible components of $Y$. Similarly the generic points of irreducible components $W$ of $\mathrm{Spec}(\cA)$ correspond to minimal primes $\fq$ in $\cA$ and we have $W=\mathrm{Spec}(A/\fq)$. Since every irreducible component of $Y$ dominates some irreducible component of $X$, any such generic point $\mathfrak{p}$ maps to a minimal prime $\fq$ of $\cA$ under $\mathrm{Spec}(\cB)\rightarrow \mathrm{Spec}(\cA)$. Therefore, for any minimal prime $\fp$ of $\cB$, there exists a minimal prime $\fq$ of $\cA$ such that $\phi^{-1}(\mathfrak{p})=\fq$. Note that we have $Y\times_X W =\mathrm{Spec}(\cB\otimes_{\cA}\cA/\fq)$ and we have the fiber product diagram

\[\begin{tikzcd}
  {Y\times_X W} \arrow[d, "\pi_W"] \arrow[r] & Y \arrow[d,"\pi"]\\
  W \arrow[r] & X\\
\end{tikzcd}
\]

Since $\cA$ is reduced and $\fq$ is a minimal prime of $\cA$, we know that the localization $\cA_\fq$ is a reduced local ring of dimension zero. Hence $\cA_\fq$ is a field and $\fq\cA_\fq=(0)$. Therefore, we have $\cA_\fq\simeq \cA_\fq/\fq\cA_\fq\simeq (\cA/\fq)_{(0)}\simeq K(\cA/\fq)$, where $K(\cA/\fq)$ is the field of fractions of the integral domain $\cA/\fq$. We have that the fiber of $\pi$ over $\fq$ is given by $Y_\fq=\mathrm{Spec}(\cB\otimes_{\cA}\cA_\fq)$ and the the generic fiber of $\pi_W$ is given by $ (Y\times_X W)_\fq=\mathrm{Spec}(\cB\otimes_{\cA/\fq}K(\cA/\fq)) $.  Note that, by the universal property of fiber products, we have $Y_\fq\simeq (Y\times_X W)_\fq$.
Now $\pi_W:Y\times_X W \rightarrow W$ is a morphism of finite type and $W$ is irreducible with generic point $\fq$. By assumption, a the fiber $\pi_W^{-1}(x)$ is reduced for a general closed point $x\in W$. By \cite[\href{https://stacks.math.columbia.edu/tag/0575}{Tag 0575}]{stacks-project} we conclude that the generic fiber $(Y\times_X W)_\fq$ is reduced. Therefore, the scheme $Y_\fq$ is reduced and the ring 
$\mathrm{Spec}(\cB\otimes_{\cA}\cA_\fq)$ is reduced.

Let $\frac{f}{t}\in \cB_{\mathfrak{p}}$ be a nilpotent element. Therefore $sf^k=0$ in $\cB$ for some $s\not\in \mathfrak{p}$. So $(sf)^k=0$ in $\cB$ and hence $(sf)^k\otimes1=0$ in $\cB\otimes_\cA \cA_\fq$. Therefore  $sf\otimes 1=0$  in $\cB\otimes_\cA \cA_\fq$ by reducedness.  Let $T\subset \cA$ be the multiplicatively closed set $\cA\setminus \fq$.  Consider the $\cA$-bilinear map $\psi: \cB\times \cA_\fq\rightarrow T^{-1}\cB$ given by $(s,\frac{a}{b})=\frac{a\cdot s}{b}$. If $sf\otimes 1=0$, then by the universal property of tensor products we must have $sf=\psi(sf, 1)=0$ in $T^{-1}\cB$. Therefore there exists an $a\in T=\cA\setminus \fq$, such that $a\cdot sf=0$ in $\cB$. Note that $s\not\in \mathfrak{p}$. We also have $\phi(a)\not\in \mathfrak{p}$, as $\phi^{-1}(\mathfrak{p})=\fq$ and $a\not \in \fq$. So we have $\phi(a)sf=a\cdot sf=0\in \mathfrak{p}$, where $\phi(a)s\not \in \mathfrak{p}$. Therefore we conclude that $f=0$ in $\cB_{\mathfrak{p}}$. Thus $\cB_{\mathfrak{p}}$ is reduced, since $\frac{f}{t}$ was an arbitrary nilpotent element.
\end{proof}

\begin{proposition}\label{proposition: dominant}
Let $S:=\K[z_1,\cdots,z_m,x_1,\cdots,x_s]$ be the graded polynomial ring with $\deg(z_i)=d_i\geq 1$ and $\deg(x_i)=e_i\geq 1$. Let  $P,Q_1,\cdots,Q_k\in $ be homogeneous elements of positive degree. Let $A=\K[x_1,\cdots,x_s]$ and $Q_1,\cdots,Q_k\in A$. Suppose $P\not \in (x_1,\cdots,x_s)$ and the ideal $(Q_1,\cdots,Q_k)$ does not have any embedded primes in $S$. Then the following holds:

\begin{enumerate}
    \item $P$ is a non-zero divisor in $S/(Q_1,\cdots,Q_k)$. 
    \item For any minimal prime $\fp$ over $(P,Q_1,\cdots,Q_k)$ in $S$, we have that $\fp\cap A$ is a minimal prime over $(Q_1,\cdots,Q_k)$ in $A$.
 \end{enumerate}

\end{proposition} 

\begin{proof}
1. Since $(Q_1,\cdots,Q_k)$ does not have any embedded primes in $S$, we know that all associated primes of $(Q_1,\cdots,Q_k)$ are minimal primes. If $P$ is a zero-divisor in $S/(Q_1,\cdots,Q_k)$, then there exists a minimal prime $\fp$ over $(Q_1,\cdots,Q_k)$ in $S$ such that $P\in \fp$. By \cref{proposition: radicals and prime sequences}, we know that there exists a minimal prime $\fq$ over $(Q_1,\cdots,Q_k)$ in $A$ such that $\fp=\fq\cdot S$. Since $(Q_1,\cdots,Q_k)$ is a homogeneous ideal in $A$, the minimal prime $\fq$ is also homogeneous. Hence $\fq\subset (x_1,\cdots,x_s)$, which is a contradiction as $P\not\in (x_1,\cdots,x_s)$.

2. Since primary decomposition commutes with polynomial extensions, we know that $(Q_1,\cdots,Q_k)$ does not have any embedded primes in $A$ as well. Let $\fp$ be a minimal prime over $(P,
Q_1,\cdots,Q_k)$ in $S$. Since $(P,Q_1,\cdots,Q_k)$ is a homogeneous ideal, we know that $\fp$ is homogeneous and $\fp\subset (z_1,\cdots,z_m,x_1,\cdots,x_s)$ in $S$. Now $(Q_1,\cdots,Q_k)\subset \fp\cap A$. Hence there exists a minimal prime $\fq$ over $(Q_1,\cdots,Q_k)$ in $A$ such that $\fq \subseteq \fp\cap A$. We will show that $\fq=\fp\cap A$.

By \cref{proposition: radicals and prime sequences}, we have that $\fq S$ is a minimal prime over $(Q_1,\cdots,Q_s)$ in $S$. Suppose $\mathrm{ht}(\fq)=h$. Then $\mathrm{ht}(\fq S)=h$, as $S/\fq S= (A/\fq)[z_1,\cdots,z_m]$. Note that $\fp$ is a minimal prime over $(\fq S,P)$ in $S$. By part (1), we know that $P$ is a non-zero divisor in $S/\fq S$. Therefore, $\mathrm{ht}(\fp)=h+1$. Suppose $\fp\cap A\neq \fq$. Then $\mathrm{ht}(\fp\cap A)\geq h+1$. Let $\fq_0\subsetneq\fq_1\cdots \subsetneq \fq_{h+1}=\fp\cap A$ be a strict chain of primes of length $h+1$ in $A$. Then by \cref{proposition: radicals and prime sequences}, we have a chain $\fq_0S\subsetneq\fq_1S \subsetneq\cdots \subsetneq (\fq_{h+1} S)= (\fp\cap A)S\subseteq \fp$ of length at least $h+1$ in $S$. Hence we must have $\fp=(\fp\cap A)S$. Now $\fq\cap A\subset (z_1,\cdots,z_m,x_1,\cdots,x_s)\cap A =(x_1,\cdots,x_s)$ in $A$. Therefore we have a contradiction, since $P\in \fp=(\fp\cap A)S$ and $P\not \in (x_1,\cdots,x_s)$.
\end{proof}

We note that the proof of part (1) of \cref{proposition: dominant} works for quotients of $S$ by prime sequences. We provide a proof below for completeness.

\begin{corollary}\label{corollary: dominant in quotient}
Let $H_1,\cdots,H_t\in S$ be a prime sequence of forms in $S$. Let $R=S/(H_1,\cdots,H_t)$. Suppose $g_1,\cdots,g_r,f_1,\cdots,f_s$ is a prime sequence of homogeneous elements in $R$. Let $P,Q_1,\cdots,Q_k\in R$  be homogeneous elements of positive degree. Let $A=\K[f_1,\cdots,f_s]$ and $A'=\K[g_1,\cdots,g_r,f_1,\cdots,f_s]$. Suppose $P\not \in (f_1,\cdots,f_s)$ and the ideal $(Q_1,\cdots,Q_k)$ does not have any embedded primes in $R$. Then  $P$ is a non-zero divisor in $R/(Q_1,\cdots,Q_k)$. 
\end{corollary}

\begin{proof}
1. Since $(Q_1,\cdots,Q_k)$ does not have any embedded primes in $R$, we know that all associated primes of $(Q_1,\cdots,Q_k)$ are minimal primes. If $P$ is a zero-divisor in $R/(Q_1,\cdots,Q_k)$, then there exists a minimal prime $\fp$ over $(Q_1,\cdots,Q_k)$ in $R$ such that $P\in \fp$. By \cref{proposition: radicals and prime sequences}, we know that there exists a minimal prime $\fq$ over $(Q_1,\cdots,Q_k)$ in $A$ such that $\fp=\fq\cdot S$. Since $(Q_1,\cdots,Q_k)$ is a homogeneous ideal in $A$, the minimal prime $\fq$ is also homogeneous. Hence $\fq\subset (f_1,\cdots,f_s)$, which is a contradiction as $P\not\in (f_1,\cdots,f_s)$.
\end{proof}

The following lemma proves a criterion for an ideal $(P,Q)$ to be prime. 

\begin{lemma}[Primality criterion]\label{lemma: irreducible prime}
Let $P,Q$ be irreducible forms in $S=\K[y_1,\cdots,y_n]$ . Let $g_1,\cdots,g_s$ be an $\mathcal{R}_\eta$-sequence in $S$ with $\eta\geq 3$. Suppose that $Q\in \K[g_1,\cdots,g_s]$  and $P$ is irreducible in $S/\mathfrak{I}_Q$, where $\mathfrak{I}_Q=(g_1,\cdots,g_s)$ . Then the ideal $(P,Q)$ is prime.
\end{lemma}

\begin{proof}
Since $g_1,\cdots, g_s$ is an $\mathcal{R}_\eta$-sequence, we know that the subalgebra $\K[g_1,\cdots,g_s]$  is isomorphic to a polynomial ring $\K[x_1,\cdots,x_s]$. We identify $\K[g_1,\cdots,g_s]$ with $\K[x_1,\cdots,x_s]$ using this isomorphism and, by abuse of notation, the polynomial in $\K[x_1,\cdots,x_s]$ corresponding to $Q$ is still denoted by $Q$. Let $Y\subset \mathbb{A}^n$ and $X\subset \mathbb{A}^s$ be the closed subschemes defined by the ideals $(P,Q)$ and $(Q)$ respectively. Let $B=S/(P,Q)$ and $A=\K[x_1,\cdots,x_s]/(Q)$. We have $\mathrm{Spec}(B)=Y$ and $\mathrm{Spec}(A)=X$. We have an inclusion of rings $A\hookrightarrow B$, since $(P,Q)\cap \K[g_1,\cdots,g_s]=(Q)$ by Proposition \ref{proposition: dominant}. Let $\pi: \mathbb{A}^n\rightarrow \mathbb{A}^s$ be the projection morphism defined by $\pi(y)=(g_1(y),\cdots,g_s(y))$. Note that $\pi(Y)\subset X$. We will show that $Y$ is irreducible and reduced.\\

First we will show that $(\pi|_Y)^{-1}(x)$ is irreducible and reduced for any closed point $x\in X$. Let $x=(c_1,\cdots,c_s)$ be a closed point in $X$ and let $\mathfrak{m}_x=(x_1-c_1,\cdots,x_s-c_s)$ the maximal ideal corresponding to $x$. Let $\mathfrak{m}_x\cdot S$ denote the ideal generated by $(g_1-c_1,\cdots,g_s-c_s)$ in $S$. By Lemma \ref{lemma: extension of prime}, we know that $\mathfrak{m}_x\cdot S$ is a prime ideal in $S$. Note that we have  $$(\pi|_Y)^{-1}(x)\simeq \mathrm{Spec}(S/(\mathfrak{m}_x\cdot S,P))$$ for $x\in X$. Since $\mathfrak{m}_x\cdot S=(g_1-c_1,\cdots,g_s-c
_s)$, it is enough to show that the ideal $(g_1-c_1,\cdots,g_s-c_s,P)$ is prime for any $x=(c_1,\cdots,c_s)\in X$. Since $g_1,\cdots,g_s$ is an $\mathcal{R}_\eta$ sequence with $\eta\geq 3$, we know that  $S/\mathfrak{J}_Q$ is an UFD. As the image $\overline{P}$ is irreducible in $S/\mathfrak{J}_Q$, we have that $(\overline{P})$ is a prime ideal in $S/\mathfrak{J}_Q$ and $(\mathfrak{J}_Q,P)$ is a prime ideal in $S$. Therefore $(g_1,\cdots,g_s,P)$ is a prime ideal in $S$. Hence $(g_1,\cdots,g_s,P)$ is a prime sequence in $S$, because $(g_1,\cdots,g_s)$ was an $\mathcal{R}_\eta$ sequence.  By \cite[Proposition 2.8]{AH20a}, we conclude that $(g_1-c_1,\cdots,g_s-c_s,P)$ is a prime ideal for any $(c_1,\cdots,c_s)\in X$. Therefore $(\pi|_Y)^{-1}(x)$ is an integral scheme, in particular it is irreducible and reduced.\\

Note that by \cref{proposition: dominant}, every irreducible component of $Y$ dominates $X$. We know that $X$ is irreducible, since $Q$ is irreducible. Since $P,Q$ is a regular sequence, we know that $Y$ is Cohen-Macaulay.  We have shown that $(\pi|_Y)^{-1}(x)$ is irreducible for all $x\in X$. 
Therefore $Y$ is irreducible and reduced by Lemma \ref{lemma: property of fibers}. 
\end{proof}

\subsection{Elimination, Resultants and Discriminant Lemma}

We now prove some structural results on the elimination ideals of radical ideals.

\begin{lemma}[Elimination of radical]\label{lemma: principal eliminations linear}
Let $P, Q_1,\cdots, Q_k \in S := \K[x_1, \ldots, x_r, y_1, \ldots, y_s]$ be polynomials of positive degree. Let $S'=\K[x_2, \ldots, x_r, y_1, \ldots, y_s]$. Suppose the following conditions hold:
\begin{enumerate}
    \item The ideal $(Q_1,\cdots,Q_k)$ is a radical ideal in $S$, and $Q_1,\cdots,Q_k \in \K[y_1, \ldots, y_s]$.
    \item The polynomial $P$ depends on the variable $x_1$, i.e. $P=\sum_{j=0}^e p_j x_1^j$, where $e>0$, $p_e\neq 0$ and $p_j\in S'$ for all $j$.
    \item We have that $p_e$ is a non-zero divisor in $S/(Q_1,\cdots,Q_k)$.
\end{enumerate}
Then the elimination ideal $\rad(P, Q_1,\cdots,Q_k)_{x_1}=\rad(P,Q_1,\cdots,Q_k)\cap S'$ is generated by $Q_1,\cdots, Q_k$, i.e. $\rad(P, Q_1,\cdots,Q_k)_{x_1} = (Q_1,\cdots,Q_k)$. In particular, $\rad{(P,Q_1,\cdots,Q_k)}\cap \K[y_1,\cdots,y_s]=(Q_1,\cdots,Q_k)$.
\end{lemma}

\begin{proof}
By item (2), we have $P = \sum_{k=0}^e p_j x_i^j$, where $0 < e$ and $p_j \in S'$, with $p_e \neq 0$.
Let $J = \rad(P, Q_1,\cdots,Q_k) \cap S'$. Since $Q_1,\cdots,Q_k \in J$, we have that $(Q_1,\cdots,Q_k) \subseteq J$.

To show the other inclusion, if $F \in J$, we have that 
$$ F^D = P A + Q_1B_1+\cdots+Q_kB_k $$
for some $A, B_1,\cdots,B_k \in S$. We will now show that $A \in (Q_1,\cdots,Q_k)$.
To prove this, it is enough to show that for any $A, B_1,\cdots,B_k \in S$ that satisfy the equation above, we must have that the leading term of $A$ is in $(Q_1,\cdots,Q_k)$, when $A$ considered as a polynomial in $x_1$ with coefficients in $S'$.

We have $\deg_{x_1}(F^D) = \deg_{x_1}(Q_t) = 0$ for all $t\in [k]$. 
Let $\ell:=\deg_{x_1}(PA)$. Note that $\ell\geq e$. Consider $A$ as a polynomial in $x_1$ with coefficients in $S'$ and then we may write $A = \sum_{j=0}^{\ell-e} a_j x_1^j$  where $a_{\ell-e}\neq 0$. Similarly, for all $t\in [k]$, we may write  $B_t = \sum_{j=0}^{\ell_t}b_{tj} x_1^j$. Comparing the coefficients of $x_1^\ell$ on both sides of the equation above, we obtain

$$p_e \cdot a_{\ell-e} + Q_1 b_{1\ell}+\cdots+Q_k b_{k\ell}= 0.$$ Therefore $p_ea_{\ell-e}\in (Q_1,\cdots,Q_k)$. Since $p_e$ is a non-zero divisor in $S/(Q_1,\cdots,Q_k)$
, we conclude that $a_{\ell-e}\in (Q_1,\cdots,Q_k)$. Thus, by induction, we must have that $A \in (Q_1,\cdots,Q_k)$, which proves that $F^D \in (Q_1,\cdots,Q_k)$. Since $(Q_1,\cdots,Q_k)$ is radical, we have that $F\in (Q_1,\cdots,Q_k)$.  Hence $\rad(P,Q_1,\cdots,Q_k)\cap S'=(Q_1,\cdots,Q_k)$.
\end{proof}

The following result generalizes \cref{lemma: principal eliminations linear} to the setting of quotients $S$ by prime sequences.

\begin{lemma}[Elimination of radical in quotient rings]\label{lemma: principal eliminations}

Let $H_1,\cdots,H_t\in S$ be a prime sequence of forms in $S$. Let $R=S/(H_1,\cdots,H_t)$. Suppose $g_1,\cdots,g_r,f_1,\cdots,f_s$ is a prime sequence of homogeneous elements in $R$. Let $P,Q_1,\cdots,Q_k\in R$  be homogeneous elements of positive degree. Let $A= \K[g_1,\cdots,g_r,f_1,\cdots,f_s]$ and $A'=\K[g_2,\cdots,g_r,f_1,\cdots,f_s]$. Suppose the following holds:
\begin{enumerate}
    \item We have $P\in \cA$ and $Q_1,\cdots,Q_k \in \K[f_1,\ldots, f_s]$.
    \item The ideal $(Q_1,\cdots,Q_k)$ is a radical ideal in $R$.
    \item As an element of the algebra $A$, we have that $P$ depends on $g_1$, i.e. $P=\sum_{j=0}^e p_j g_1^j$, where $e>0$, $p_e\neq 0$ and $p_j\in \cA'$ for all $j$.
    \item We have that $p_e$ is a non-zero divisor in $R/(Q_1,\cdots,Q_k)$.
\end{enumerate}
Then the elimination ideal $\rad_R(P,Q_1,\cdots,Q_k)\cap A'$ is generated by $Q_1,\cdots,Q_k$, i.e. we have  \[\rad_R(P, Q_1,\cdots,Q_k)\cap A' = (Q_1,\cdots,Q_k).\]
\end{lemma}

\begin{proof}
Since $(Q_1,\cdots,Q_k)$ is radical in $R$, we know that $(Q_1,\cdots,Q_k)$ is radical in the algebra $A$, by \cref{proposition: radicals and prime sequences}. Since $p_e$ is a non-zero divisor in $R/(Q_1,\cdots,Q_k)$, we know that $p_e\not\in \fp$ for any minimal prime over $(Q_1,\cdots,Q_k)$ in $R$. By \cref{proposition: radicals and prime sequences}, we know that any minimal prime $\fq\subset A$ over $(Q_1,\cdots,Q_k)$ in $A$ is given by $\fq=\fp\cap A$ for some minimal prime $\fp$ over $(Q_1,\cdots,Q_k)$ in $R$. Therefore, $p_e\not \in \fq$
 for any any minimal prime over $(Q_1,\cdots,Q_k)$ in $A$, and hence $p_e$ is a non-zero divisor in $A/(Q_1,\cdots,Q_k)$. Therefore, by applying
\cref{lemma: principal eliminations linear} to $A$, we obtain $\rad_A(P,Q_1,\cdots,Q_k)\cap A'=(Q_1,\cdots,Q_k)$. Now, by \cref{proposition: radicals and prime sequences}, we know that $\rad_R(P,Q_1,\cdots,Q_k)\cap A=\rad_A(P,Q_1,\cdots,Q_k)$. Therefore we have \[\rad_R(P,Q_1,\cdots,Q_k)\cap A'=\rad_A(P,Q_1,\cdots,Q_k)\cap A'=(Q_1,\cdots,Q_k).\]
\end{proof}

The following result proves a criterion for showing that an ideal is radical, based on discriminants of one of the generators.

\begin{lemma}\label{lemma: CM discriminant radical}
Let $S:=\K[z_1,\cdots,z_r,x_1,\cdots,x_s]$ be the graded polynomial ring with $\deg(z_i)=d_i\geq 1$ and $\deg(x_i)=e_i\geq 1$. Let $A=\K[x_1,\cdots,x_s]$ as a graded subalgebra. Let  $P,Q_1,\cdots,Q_k\in S$ be homogeneous elements of positive degree. Suppose that the following conditions hold:
\begin{enumerate}
    \item We have $Q_1,\cdots,Q_k\in A$ and the ideal $(Q_1,\cdots,Q_k)$ is a Cohen-Macaulay, radical ideal in $A$.
    \item $P$ does not have multiple factors and $P\not \in (x_1,\cdots,x_s)$ in $S$.
    \item For all $i\in[r]$ such that $P$ depends on the variable $z_i$, we have $\mathrm{Disc}_{z_i}(P)\not\in \fq\cdot S$ for any minimal prime $\fq$ of $(Q_1,\cdots,Q_k)$ in $A$.
\end{enumerate} 
Then the ideal $(P,Q_1,\cdots,Q_k)$ is radical in $S$.
\end{lemma}

\begin{proof}
Without loss of generality, suppose that $P$ only depends $z_1,\cdots,z_t$ for some $t\leq r$. By applying \cref{proposition: radicals and prime sequences}, it is enough to prove that $(P,Q_1,\cdots,Q_k)$ is radical in $\K[z_1,\cdots,z_t,x_1,\cdots,x_s]$. Therefore, without loss of generality, we may assume that $P$ depends on all the $z_i$ variables in $S$.

Let $J=(Q_1,\cdots,Q_k)$ in $A$ and $I=JS=(Q_1,\cdots,Q_k)$ in $S$. By \cref{proposition: radicals and prime sequences}, we have that $I$ is radical in $S$, since $J=(Q_1,\cdots,Q_k)$ is radical in $A$. Furthermore, we have $I\cap A= J$. Since $A/J$ is Cohen-Macaulay, we know that $S/IS\simeq(A/J)[z_1,\cdots,z_r]$ is also Cohen-Macaulay. By \cref{proposition: dominant}, we know that $P$ is a non-zero divisor in $S/I$. Since $S/I$ is Cohen-Macaulay by assumption, we conclude that $S/((P)+I)$ is also Cohen-Macaulay and hence it has no embedded primes.  Let $\pi:\mathbb{A}^{r+s}\rightarrow \mathbb{A}^s$ be the morphism corresponding to $A\subset S$. Note that $\pi(z_1,\cdots,z_r,x_1,\cdots,x_s)=(x_1,\cdots,x_s)$ for closed points. Consider the affine schemes $Y=\mathrm{Spec}(S/((P)+I))$ and $X=\mathrm{Spec}(A/J)$. Note that we have a homomorphism of finitely generated $\K$-algebras $A/J\rightarrow S/((P)+I)$, as $J\subset ((P)+I)\cap A$. Note that the corresponding morphism of affine schemes is given by $\pi|_Y:Y\rightarrow X$ and we have commutative diagram.

\[\begin{tikzcd}
 Y \arrow[r]\arrow[d,"\pi|_Y"] & \A^{r+s}\arrow[d,"\pi"]\\
 X \arrow[r] & \A^s\\
\end{tikzcd}\]

Since $J$ is a radical ideal in $A$, we have that $X$ is reduced. Since $S/((P)+I)$ does not have any embedded primes, we know that $Y$ is satisfies the $\cS_1$-property. By \cref{proposition: dominant}, we know that for every minimal prime $\fp$ over $(P)+I$ in $S$, there exists a minimal prime $\fq$ over $J$ in $A$ such that $\fp\cap A=\fq$. In particular, every irreducible component of $Y$ dominates some irreducible component of $X$. Therefore, by \cref{lemma: property of fibers}, it is enough to show that for every irreducible component $W$ of $X$, we have that $\pi|_Y^{-1}(x)$ is reduced for a general closed point $x\in W$.

 Let $x=(c_1,\cdots,c_s)\in X$ be a closed point. Let $P_x\in \K[z_1\cdots,z_r]$ denote the polynomial $P(z_1\cdots,z_r,c_1,\cdots,c_s)$. Note that $\pi|_Y^{-1}(x)=\mathrm{Spec}(\K[z_1,\cdots,z_r]/(P_x))$, as $x\in X$. Since $P$ is homogeneous and $P\not\in (x_1,\cdots,x_s)$, we conclude that $P_x\not\equiv 0$ for any $x\in X$. Now $\pi|_Y^{-1}(x)$ is not reduced iff the polynomial $P_x$ has a multiple factor, i.e. $\mathrm{Disc}_{z_k}(P_x)\equiv0$ for some variable $z_k$ such that $P_x$ depends on $z_k$. For $i\in [r]$, let $Z_i\subset \mathbb{A}^{r+s}$ be the closed subscheme $\mathrm{Spec}(S/\mathrm{Disc}_{z_i}(P))$ defined by $\mathrm{Disc}_{z_i}(P)$. Note that $\mathrm{Disc}_{z_i}(P)\not\equiv 0$, since $P$ does not have multiple factors. Now, for any variable $z_i$ such that $P_x$ depends on $z_i$, we have \[\mathrm{Disc}_{z_i}(P_x)\equiv0 \iff \mathrm{Disc}_{z_i}(P)(z_1,\cdots,z_r,c_1,\cdots,c_s)=0\iff \pi^{-1}(x)\subset Z_i.\] To summarize, for a closed point $x\in X$, we have that $\pi|_Y^{-1}(x)$ is reduced if  $\pi^{-1}(x)\not\subset Z_i$ for all $i$ such that $P_x$ depends on $z_i$. Therefore it is enough to show that, for any irreducible component $W$ of $X$ and a general closed point $x\in W$, we have that $\pi^{-1}(x)\not\subset Z_i$ for all $i\in [r]$.

Note that $\mathrm{Spec}(S/I)=\pi^{-1}(X)=X\times \mathbb{A}^r$. Fix $i\in [r]$. Let $W$ be an irreducible component of $X$ corresponding to a minimal prime $\fq$ over $J$ in $A$. Let $\dim(W)=n$ and we have $\pi^{-1}(W)=W\times \A^r$. By assumption (3), we have that $Z_i\cap \pi^{-1}(W)$ is of dimension at most $n+r-1$. 
Let $T$ be an irreducible component of $Z_i\cap \pi^{-1}(W)$ with the reduced scheme structure. Note that $\dim(T)\leq n+r-1$. If $T$ does not dominate $W$, then for a general $x\in W$ we have $\pi^{-1}(x)\cap T=\emptyset$. If $T$ dominates $W$, then the composition morphism $T\rightarrow Z_i\cap \pi^{-1}(W) \rightarrow \pi^{-1}(W) \xrightarrow[]{\pi_W}W$ is a dominant morphism of integral separated schemes of finite type. Therefore, there exists an open subset $U\subset W$ such that $\dim(\pi^{-1}(x)\cap T)= \dim(T)-\dim(W)\leq r-1$ for all $x\in U$, by \cite[Corollary 14.5]{Eis95} or \cite[Theorem 11.4.1]{Vak17}. Since there are finitely many possibilites for the irreducible components $T$ of $Z_i\cap \pi^{-1}(W)$, we conclude that $\dim(\pi^{-1}(x)\cap(Z_i\cap \pi^{-1}(W)))\leq r-1$ for a general $x\in W$. Since $\dim(\pi^{-1}(x))=r$, we conclude that $\pi^{-1}(x)\not\subset Z_i$ for a general $x\in W$. Since there are finitely many choices for $Z_i$, we conclude that for a general $x\in W$ we have $\pi^{-1}(x)\not\subset Z_i$ for all $i\in [r]$. This completes the proof.
\end{proof}

\begin{example}
Note that in \cref{lemma: CM discriminant radical}, we need that $\mathrm{Disc}_{z_i}(P)\not\in \fq S$ for all $z_i$ such that $P$ depends on $z_i$. It is not enough to assume that this property holds for just one such $z_i$. One can construct such an example as follows. Let $P=y^3+vy^2+(xu^2-z^3)$ and $Q=xu^2-z^3$ in $ \K[x,y,z,u,v]$. Then $Q\in \K[x,u,z]$ and $\mathrm{Disc}_v(P)=y^2\not\in (Q)$. However, the ideal $(P,Q)$ is not radical as $\sqrt{(P,Q)}=(y^2+yv,xu^2-z^3)$. This occurs because we have  $\mathrm{Disc}_y(P)\in (Q)$.
\end{example}

\begin{definition} Let $R$ be a finitely generated $\K$-algebra and $g_1,\cdots,g_s$ a regular sequence in $R$. Let  $\phi:\K[y_1,\cdots,y_s]\rightarrow A$ be the isomorphism defined as $\phi(y_i)=g_i$.
Let $P,Q\in \cA$. Let $\widetilde{P}=\phi^{-1}(P)$ and $\widetilde{Q}=\phi^{-1}(Q)$. We define the resultant of $P,Q$ with respect to $g_i$ in the subalgebra $\mathcal{A}$ as 
\[\mathrm{Res}_{g_i}^{\mathcal{A}}(P,Q)= \phi(\mathrm{Res}_{y_i}(\widetilde{P},\widetilde{Q})).\]
Similarly we define the discriminant of $P$ with respect to $g_i$ in the sublagebra $\mathcal{A}$ as
\[\mathrm{Disc}_{g_i}^{\mathcal{A}}(P)=\phi(\mathrm{Res}_{y_i}(\widetilde{P},\frac{\partial \widetilde{P}}{\partial y_i})).\]
\end{definition}

Using the isomorphism $\phi$ we can translate the usual properties of resultants in polynomial rings to similar properties in the subalgebra $\mathcal{A}$. 
Note that $\mathrm{Res}_{g_i}^{\mathcal{A}}(P,Q)$ is in the subalgebra generated by $\{g_1,\cdots,g_s\}\setminus \{g_i\}$ in $S$. Also, we have $\mathrm{Res}_{g_i}^{\mathcal{A}}(P,Q)\in (P,Q)$.\\

\begin{corollary}

\label{lemma: generalized discriminant lemma}
Let $H_1,\cdots,H_t\in S$ be a prime sequence of forms in $S$. Let $R=S/(H_1,\cdots,H_t)$. Suppose $g_1,\cdots,g_r,f_1,\cdots,f_s$ is a prime sequence of homogeneous elements in $R$.  Let $A'= \K[g_1,\cdots,g_r,f_1,\cdots,f_s]$ and $A=\K[f_1,\cdots,f_s]$. Let $P,Q_1,\cdots,Q_k\in R$  be homogeneous elements of positive degree. Suppose the following conditions hold:
\begin{enumerate}
    \item We have $Q_1,\cdots,Q_k\in A$ and the ideal $(Q_1,\cdots,Q_k)$ is a Cohen-Macaulay, radical ideal in $A$.
    \item We have $P\in A'$ and $P$ does not have multiple factors in $A'$. Suppose $P\not \in (f_1,\cdots,f_s)$ in $R$.
    \item For all $i\in[r]$ such that $P$ depends on $g_i$ in $A'$, we have $\mathrm{Disc}_{g_i}^{A'}(P)\not\in \fq\cdot R$ for any minimal prime $\fq$ of $(Q_1,\cdots,Q_k)$ in $A$.
\end{enumerate}
Then the ideal $(P,Q_1,\cdots,Q_k)$ is a radical ideal in $R$.
\end{corollary}
\begin{proof}
By \cref{proposition: radicals and prime sequences}, it is enough to show that $(P,Q_1,\cdots,Q_k)$ is a radical ideal in the $\K$-subalgebra $A'$. Note that we have an isomorphism of graded $\K$-algebras $A'\simeq\K[z_1,\cdots,z_r,x_1,\cdots,x_s]$ given by $g_i\mapsto z_i$ and $f_i\mapsto x_i$ where $\deg(z_i)=\deg(g_i)$ and $\deg(x_i)=\deg(f_i)$. Therefore, condition (1) of \cref{lemma: CM discriminant radical} holds. Note that $P\not \in(f_1,\cdots,f_s)R\cap A'=(f_1,\cdots,f_s)$ in $A'$. Therefore condition (2) of \cref{lemma: CM discriminant radical} holds. Now, we also note that for any minimal prime $\fq$ over $(Q_1,\cdots,Q_k)$ in $A$, we have $\mathrm{Disc}_{g_i}^{A'}(P)\not\in \fq\cdot A'$, by condition (3) above, Hence, we may apply \cref{lemma: CM discriminant radical} to conclude that $(P,Q_1,\cdots,Q_k)$ is radical in $A'$ as desired.
\end{proof}

\begin{corollary}\label{corollary: best corollary discriminant}
Let $\eta\geq 3$ and $H_1,\cdots,H_t\in S$ be an $\cR_\eta$-sequence of forms in $S$. Consider the UFD  $R=S/(H_1,\cdots,H_t)$. Suppose $g_1,\cdots,g_r,f_1,\cdots,f_s$ is a prime sequence of homogeneous elements in $R$.  Let $P\in R$  be a square-free homogeneous element of degree $d\geq 1$. Suppose we have $P\in  \K[g_1,\cdots,g_r,f_1,\cdots,f_s]$ and $P\not\in (f_1,\cdots,f_s)$ in $R$. Let $\cH=\{Q_1,\cdots,Q_t\} \subset  \K[f_1,\cdots,f_s]$ be such that: 

\begin{enumerate}
    \item $Q_i$ is square-free homogeneous of positive degree in $R$ for all $i\in[t]$.
    \item $Q_i,Q_j$ do not have common factors in $R$ for any $i\neq j$. 
    \item $(P,Q_i)$ is not radical in $R$ for all $i$.

\end{enumerate} 

Then we have $|\cH|=t\leq d^2(2d-1)$. In particular, there exist at most $d^2(2d-1)$ number of distinct elements in $\K[f_1,\cdots,f_s]$ satisfying the conditions (1),(2) and (3) above.
\end{corollary}

\begin{proof}
Let $\cB=\K[g_1,\cdots,g_r]$. Consider $P$ as a polynomial in $\cB[f_1,\cdots,f_s]$. Since $P\not\in (f_1,\cdots,f_s)$, there exists a monomial $\alpha \prod_j g_{i_j}^{e_j}$ in $P$, where $\alpha\in \K\setminus \{0\}$ and $g_{i_j}\in \{g_1,\cdots,g_r\}$. Without loss of generality, let $g_1,\cdots,g_k$ be the elements appearing in this monomial. Note that we must have $k\leq d$, since $P\in R_{d}$ and $\deg(g_i)\geq 1$ for all $i\in [r]$. Now, we have that $P\not\in (g_{k+1},\cdots,g_r,f_1,\cdots,f_s)$.

Let $A=\K[g_{k+1},\cdots,g_r,f_1,\cdots,f_s]$. Then $Q_i\in A$ for all $i\in [t]$. Now, $Q_i$ is non-zero and hence a non-zero divisor in $A$. Therefore we have that $(Q_i)$ is a Cohen-Macaulay ideal. Since $Q_i$ is square-free in $R$, it is square-free in $A$ and hence $(Q_i)$ is a radical ideal in $A$ for all $i\in [t]$. Note that $P\in A':=\K[g_1,\cdots,g_r,f_1,\cdots,f_s]$ and $P$ is square-free in $R$ by assumption. Therefore $P$ is also square-free in $A'$. We also have $P\not\in (f_1,\cdots,f_s)$ in $R$, by assumption. Therefore, by \cref{lemma: generalized discriminant lemma}, if $(P,Q_i)$ is not radical in $R$ then there exists $\ell\leq k$ and some minimal prime $\fq$ over $(Q_i)$ in $A$ such that $\mathrm{Disc}_{g_\ell}^{A'}(P)\in \fq R$. Now the minimal primes of $\fq$ over $(Q_i)$ in $A$ are given by the irreducible factors of $Q_i$ in $A$. Therefore, if $(P,Q_i)$ is not radical, then there exists some irreducible factor of $Q_i$ in $A$ that divides $\prod_{\ell=1}^k\mathrm{Disc}_{g_\ell}^{A'}(P)$. Now $\deg(\mathrm{Disc}_{g_\ell}^{A'}(P))\leq d(2d-1)$ for all $i\in [k]$, as $\deg(P)=d$. Therefore, we have $\deg(\prod_{\ell=1}^k\mathrm{Disc}_{g_\ell}^{A'}(P))\leq d^2(2d-1)$ as $k\leq d$. Therefore the number of irreducible factors of $\prod_{\ell=1}^k\mathrm{Disc}_{g_\ell}^{A'}(P)$ in $R$ is at most $d^2(2d-1)$. If $t>d^2(2d-1)$, then by the pigeonhole principle there exist $i\neq j$ such that $Q_i,Q_j$ have a common irreducible factor in $A$. This is a contradiction since $Q_i,Q_j$ do not have common factors in $R$ for any $i\neq j$.
\end{proof}


\section{Strength of forms and Strong Ananyan-Hochster spaces}\label{sec:strong-algebras}


In this section we describe the main tool which we will use to prove our Sylvester-Gallai theorem -- strong Ananyan-Hochster algebras.
As mentioned in \cref{sec:intro}, we would like to construct small algebras which behave as polynomial rings and contain many of the polynomials in our configuration.
As we saw in \cref{sec:algebraic-geometry}, for an algebra to ``behave as a polynomial ring'' it is enough to construct an algebra whose generators form a prime sequence.
Another property that we would like from our algebras, is that they are robust with regards to certain augmentations, as we will increase such algebras to contain more and more forms from our configuration. 
Moreover, we would like to preserve the structure of the original algebra inside of the augmented one -- a concept which we will make precise in \cref{proposition: constructing AH algebras}.

In a recent breakthrough work by Ananyan-Hochster \cite{AH20a}, where they positively answer Stillman's conjecture, they show that given a set of polynomials $\cH := \{H_1, \ldots, H_r\}$ in a polynomial ring $S$, one can construct a sub-algebra $\cA\subset S$ such that $\cH \subset \cA$  and the number of generators of $\cA$ is uniformly bounded 
by a function of the size of $\cH$ and the degrees of the polynomials in $\cH$.

Building on \cite{AH20a}, we define the notion of a strong Ananyan-Hochster algebra (which we denote henceforth by strong AH algebra). 
We show that strong AH algebras have excellent algebraic-geometric properties and are particularly suitable for applications to Sylvester-Gallai problems as described above.
In order to construct such algebras, we need to define a notion of rank for a form in $S$.
The notion of rank that we describe below  is called strength and was introduced in \cite{AH20a, AH20b}. 
This notion can be seen as a symmetric generalization of the notion of slice rank of a tensor ~\cite{T16}.

\subsection{Strength}
Let $R=\oplus_{d\geq 0} R_d$ be a finitely generated graded $\K$-algebra,  generated by $R_1$. 
In \cite{AH20a} the notions of collapse and strength were defined for a polynomial ring. 
We will extend those definitions to finitely generated graded $\K$-algebras and prove the necessary properties below.
Henceforth, we refer to a homogeneous element of $R$ as a \emph{form}, adopting the same notation for polynomial rings.

\begin{definition}[Collapse]\label{def:collapse} 
Given a non-zero form $F \in R_d$, we say that $F$ has a \emph{$k$-collapse} if there exist $k$ forms $G_1, \ldots, G_k$ such that $1 \leq \deg(G_i) < d$ and $F \in (G_1, \ldots, G_k)$.
\end{definition}

\begin{definition}[Strength]\label{def:strength} 
Given a non-zero form $F \in R_d$, the \emph{strength} of $F$, denoted by $s(F)$, is the least positive integer such that $F$ has a $(s(F)+1)$-collapse but it has no $s(F)$-collapse.
We say that $s(F) \geq t$ whenever $F$ does not have a $t$-collapse.
\end{definition}

\begin{remark}
By the definitions above, a form $x\in R_1$ does not have a $k$-collapse for any $k \in \N$. 
Thus, we say that for any $x \in R_1$, $s(x) = \infty$. In particular, linear forms in the polynomial ring $S$ have infinite strength.

We will make the convention that $s(0) = -1$.
\end{remark}

\begin{definition}[Minimum collapse]\label{definition: minimum collapse}
Given a non-zero form $F \in R_d$ and $s \in \N^*$ such that $s(F) = s-1$, a \emph{minimum collapse} of $F$ is any identity of the form $F = G_1 H_1 + \cdots + G_s H_s$, where $G_i, H_i$ are forms of degree in $[d-1]$.
\end{definition}

It is useful to define the min and max strength of a linear system of forms of the same degree.

\begin{definition}[Min and max strength]\label{definition: smin smax}
Given a set of forms $F_1, \ldots, F_r \in R_d$,  define 
$\smin{F_1, \ldots, F_r}$ as the \emph{minimum} strength of a \emph{non-zero} form in $\Kspan{F_1, \ldots, F_r}$ and $\smax{F_1, \ldots, F_r}$ as the \emph{maximum} strength of a form in $\Kspan{F_1, \ldots, F_r}$. 

In particular, given any non-zero finite dimensional vector space $V \subset R_d$, define $\smin{V}$ ($\smax{V}$) as the minimum (maximum) strength of any non-zero form in $V$. If $V=(0)$, then there are no non-zero forms in $V$. In this case, by convention we define $\smin{(0)}=\smax{(0)}=\infty$.
We will say that a vector space $V$ is $k$-strong if $\smin{V}\geq k$. Note that the zero vector space is infinitely strong.
\end{definition}

\begin{proposition}\label{proposition: strong with z} Let $z$ be a new variable and consider the polynomial ring $S[z]$.
\begin{enumerate}
    \item Let $F\in S$ be a form of degree $d$ with $s(F)=s$. Then, $s(F)=s$ as a form in $S[z]$, and for any $\alpha\in \K$, we have $s \leq s(F-\alpha z^d)\leq s+1$ in $S[z]$.
    \item Let $F_1,\cdots,F_m\in S$ be linearly independent forms.  Let $G_i\in S[z]$ be defined as $G_i=F_i-\alpha_iz^{\deg(F_i)}$ for some $\alpha_i\in \K$. Then $G_1,\cdots,G_m$ are linearly independent.
    \item Let $V\subset S_d$ be a $k$-strong vector space with a basis $F_1,\cdots, F_m$. Let $G_i\in S[z]$ be defined as $G_i=F_i-\alpha_iz^d$ for some $\alpha_i\in \K$. Then $\Kspan{G_1,\cdots,G_m}$ is a $k$-strong vector space in $S[z]$.
\end{enumerate}
   
\end{proposition}
\begin{proof}
   1. If $F$ has a $k$-collapse in $S$, then it is also a $k$-collapse in $S[z]$. Conversely, let $F=f_1g_1+\cdots+f_kg_k$ be a $k$-collapse in $S[z]$. Then modulo $(z)$, we have $F=\overline{f}_1\overline{g}_1+\cdots+\overline{f}_k\overline{g}_k$, which is a $\ell$-collapse for some 
    $\ell\leq k$, where $\overline{f}_i,\overline{g}_j$ are the images of $f_i,g_j$ modulo $(z)$. Thus we must have $k\geq \ell \geq s+1$ in $S$. Therefore strength of $s(F)=s$ in $S[z]$.
    
    If $F$ has a $k$-collapse in $S$, then $F-\alpha z^d$ has a $(k+1)$-collapse in $S[z]$. Hence $s(F-\alpha z^d)\leq s+1$. Now let $F-\alpha z^d=f_1g_1+\cdots+f_kg_k$ be a $k$-collapse in $S[z]$. Then modulo $(z)$, we have $F=\overline{f}_1\overline{g}_1+\cdots+\overline{f}_k\overline{g}_k$, which is a $\ell$-collapse for some 
    $\ell\leq k$, where $\overline{f}_i,\overline{g}_j$ are the images of $f_i,g_j$ modulo $(z)$. Thus we must have $k\geq s+1$. Hence $s(F-\alpha z^d)\geq s$.
    
    2. If $\sum \beta_iG_i=\sum \beta_i F_i-(\sum\beta_i\alpha_i)z^d=0$ for some choice of $\beta_i \in \K$, then we have $\sum \beta_i F_i=(\sum\beta_i\alpha_i)z^d$. 
    Thus we must have $\sum \beta_i F_i=0$ and hence $\beta_i=0$ for all $i$.
    
    3. Let $U=\Kspan{G_1,\cdots,G_m}$. 
    Then for any $G\in U\setminus \{0\}$, we have $G=\sum_i\lambda_i (F_i-\alpha_i z^d)=F-\alpha z^d$ for some $\lambda_i,\alpha \in \K$ where $F=\sum_i\lambda_i F_i$. 
    Note that $F\neq 0$ as $G\neq 0$. Thus $s(G)=s(F-\alpha z^d)\geq s(F)\geq k$, by part (1).
\end{proof}

\subsection{Strong Ananyan-Hochster Vector Spaces}

In \cite{AH20a}, the authors proved that if a form $F$ is \emph{sufficiently strong} then the singular locus of the hypersurface $V(F)$ is small. 
More generally they show that if $\smin{F_1,\cdots,F_m}$ is sufficiently high then $F_1,\cdots, F_m$ is an $\mathcal{R}_\eta$ sequence for suitable $\eta\geq 1$. We recall their main theorem, \cite[Theorem A]{AH20a} below.

\begin{theorem}\label{theorem: Ananyan-Hochster}
Let $\K$ be an algebraically closed field, $N$ be a positive integer and $S := \K[x_1,\cdots, x_N]$. 
\begin{enumerate}
    \item There exists an ascending function $A(\eta,d)\geq d-1\geq 0$ of $\eta,d\in \Z_{+}$, such that if $F\in S$ is a form of degree $d\geq 1$ with $s(F)\geq A(\eta,d)$, then the codimension of the singular locus in $S/(F)$ is at least $\eta+1$, i.e. $S/(F)$ satisfies the $\mathcal{R}_\eta$-property.
    
    \item  For every $\eta\in \N$, there exists an ascending function $ B_\eta=( B_{\eta,1},\cdots, B_{\eta,d}):\N^d\rightarrow \N^d$ with the following property:
    
    If $V=\oplus_{i=1}^d V_i \subset S$ is a graded $\K$-vector subspace with $\dim V=n$ and dimension sequence $\delta=(\delta_1,\cdots,\delta_d)$, such that for every $i\in [d]$, we have $\smin{V_i}\geq B_{\eta,i}(\delta)$, then every sequence of $\K$-linearly independent forms in $V$ is an $\mathcal{R}_\eta$-sequence.
    \item In part (2) above, it is sufficient to have $B_{\eta,i}(\delta)=A(\eta,i)+3(n-1)$, where $\delta=(\delta_1,\cdots,\delta_d)\in \N^d$ and $n=\sum_{i=1}^d\delta_i$.
\end{enumerate}
\end{theorem}

Note that the bounds $A(\eta,d)$ and $B_\eta$ are independent of the field $\K$ and the number of variables $N$ in $S=\K[x_1,\cdots,x_N]$. \cref{theorem: Ananyan-Hochster} is the main technical ingredient in the proof of Stillman's conjecture in \cite{AH20a}, where the authors proved that the projective dimension of an ideal $I\subset S$ admits a bound independent of $\K$ and $N$.

The theorem above motivates our definition of strong Ananyan-Hochster algebras. 
Informally, a strong Ananyan-Hochster algebra $\cA\subset S$ is a $\K$-subalgebra generated by a sequence $F_1,\cdots, F_m$ of forms such that $\Kspan{F_1,\cdots,F_m}$ is extremely strong. 
In our setting, we will enforce the strength to be much higher than the Ananyan-Hochster functions $B_{\eta,i}$, which will ensure that $F_1,\cdots,F_m$ remains an $\mathcal{R}_\eta$-sequence even after adjoining a constant number of forms to the subalgebra. 
This robustness property with respect to augmentation of forms is crucial in our applications to Sylvester-Gallai configurations.

\begin{definition}[Strong Ananyan-Hochster vector spaces]\label{def:strong-AH-algebra} 
Let $R=\oplus_{d\geq 0} R_d$ be a finitely generated graded $\K$-algebra,  generated by $R_1$. 
For any function $B=(B_1,\cdots,B_d):\N^d\rightarrow \N^d$, we will say that a non-zero graded vector subspace $V=\oplus_{i=1}^dV_i\subset R$, with dimension sequence $\delta$, is a $B$-strong AH vector space if $V_i$ is $B_i(\delta)$-strong for all $i$, i.e. $\smin{V_i}\geq B_i(\delta)$.  
The subalgebra $\K[V]\subset R$ generated by a $B$-strong AH vector space $V$ is called a $B$-strong AH algebra.
\end{definition}

Note that if $V=(0)$, then $V$ is $B$-strong for any function $B$, since the vector space $\smin{(0)}=\infty$.
We note that \cref{theorem: Ananyan-Hochster} implies that any homogeneous basis of a strong AH-vector space in $S$ is a $\mathcal{R}_\eta$-sequence if $B_i(\delta)$ are sufficiently large.

\begin{corollary}\label{corollary: strong sequence R_eta}
Let $V=\oplus_{i=1}^d V_i\subset S$ be a $B$-strong AH vector space for some $B:\N^d\rightarrow \N^d$. Suppose $B_i(\delta)\geq A(\eta,i)+3(\sum_i\delta_i-1)$ for some $\eta\in \N$. Then any sequence of $\K$-linearly independent forms in $V$ is an $\mathcal{R}_\eta$-sequence. If $\eta\geq 3$, then $S/(V)$ is a Cohen-Macaulay, unique factorization domain.
\end{corollary}

\begin{proof}
    Since $B_i
    (\delta)\geq B_{\eta,i}(\delta)$, we see that any sequence of $\K$-linearly independent forms in $V$ is an $\mathcal{R}_\eta$-sequence by \cref{theorem: Ananyan-Hochster}. If $\eta\geq 3$, then $S/(V)$ is a Cohen-Macaulay UFD by \cref{proposition: strong sequence CM UFD}.  
\end{proof}

For any $\mu\in \N^d$, we define the translation function $t_\mu:\N^d\rightarrow\N^d$ as $t_\mu=(t_{\mu,1},\cdots, t_{\mu,d})$ where the $i$-th component is defined by $t_{\mu,i}(\delta)=\delta_i+\mu_i$. In other words, for all $i\in [d]$ we add $\mu_i$ to the $i$-th component of $\delta$. For any $n\in \N$, we let $t_n:=t_{(n,\cdots,n)}$.

\begin{proposition}\label{proposition: properties of quotient strength} 
   Let $B:\N^d \rightarrow \N^d$ be an ascending function and $U\subset S$ be a graded vector space in $S$ with dimension sequence $\delta\in \N^d$. Let $R=S/(U)$ and $V\subset R$ be a graded vector subspace with dimension sequence $\mu\in \N^d$.
   \begin{enumerate}
           \item Let $F\in S$ be a form and $\overline{F}\in R$ be the image of $F$ in $R$. Then $ s(\overline{F})\leq s(F) $. Furthermore, there exists a form $H_1\in U$, such that $s(\overline{F})\geq s(F-H_1)-\dim(U)$.
           \item If $V$ is $B$-strong in $R$, then for any homogeneous basis $F_1,\cdots,F_m\in R$ of $V$ and any set of homogeneous lifts $\widetilde{F}_1,\cdots,\widetilde{F}_m\in S$, we have that $\Kspan{\widetilde{F}_1,\cdots,\widetilde{F}_m}$ is $B$-strong in $S$.
       \item Suppose $U$ is $B\circ t_\mu$-strong in $S$ and $V$ is $B\circ t_\delta$-strong in $R$. Then, for any basis $F_1,\cdots,F_m\in R$ of $V$ and any set of lifts $\widetilde{F}_1,\cdots,\widetilde{F}_m\in S$, we have that $U+\Kspan{\widetilde{F}_1,\cdots,\widetilde{F}_m}$ is $B$-strong in $S$.
       \item Suppose $B_i(\alpha+\beta)\geq B_i(\beta)+ \sum_j\alpha_j$ for all $\alpha,\beta\in \N^d$ and $i\in [d]$. Suppose that for any homogeneous basis $F_1,\cdots,F_m$ of $V$ and any set of homogeneous lifts $\widetilde{F}_1,\cdots,\widetilde{F}_m\in S$, we have  $U+\Kspan{\widetilde{F}_1,\cdots,\widetilde{F}_m}$ is $B$-strong in $S$. Then $V$ is $B$-strong in $R$. 
       \item Suppose $U$ is $B\circ t_n$-strong for some $n\in \N$. Let $a_1,\cdots,a_n\in R_1$ be linear forms. Then $a_1,\cdots,a_n$ is a prime sequence in $R$.
   \end{enumerate}
\end{proposition}

\begin{proof}
 1. Suppose $F\in S$ has a $k$-collapse in $S$. Then we have that the image $\overline{F}$ in $R$ has an at most $k$-collapse in $R$. Thus $s(\overline{F})\leq s(F)$. Suppose $F$ is of degree $i$ and $\overline{F}=g_1h_1+\cdots+g_kh_k$ is a $k$-collapse in $R$. Then we have $F=H+\widetilde{g}_1\widetilde{h}_1+\cdots+\widetilde{g}_k\widetilde{h}_k$ where $H\in (U)$ and $\widetilde{g}_j,\widetilde{h}_j$ are homogeneous lifts in $S$. We may write $H=H_1+H_2$, where $H_1\in U_i$ and $H_2\in (U_{<i})$. Then $F-H_1=H_2+\widetilde{g}_1\widetilde{h}_1+\cdots+\widetilde{g}_k\widetilde{h}_k$ is an at most $(\dim(U)+k)$-collapse of $F-H_1$. Hence, we have $\dim(U)+k\geq s(F-H_1)+1$. Therefore, $k\geq s(F-H_1)-\dim(U)+1$. Since $k$ was arbitrary, we conclude that $s(\overline{F})\geq s(F-H_1)-\dim(U)$. 
 
  2.  Note that $\Kspan{\widetilde{F}_1,\cdots,\widetilde{F}_m}$ has the same dimension sequence as $V$. For any non-zero $G\in \Kspan{\widetilde{F}_1,\cdots,\widetilde{F}_m}$, we have that $s(G)\geq s(\overline{G})$ by part $(1)$. Thus $\Kspan{\widetilde{F}_1,\cdots,\widetilde{F}_m}$ is also $B$-strong.
  
 3.  Let $W=U+\Kspan{\widetilde{F}_1,\cdots,\widetilde{F}_m}$. If $G\in U_i$ is a non-zero form, then $G$ is $B_i(\delta+\mu)$-strong. Let $G\in W_i\setminus U_i$. We may write $G=H+\sum \alpha_j \widetilde{F_j}$ where $H\in U_i$ and $\deg(\widetilde{F}_j)=i$. Then the image $\overline{G}=\sum \alpha_j F_j\in V$ is non-zero in $R$, as $F_1,\cdots,F_m$ are linearly independent in $R$. Then by part $(2)$, we have that $G$ is $B_i(\delta+\mu)$-strong. Hence $W$ is $B$-strong.
   
   4.  Let $F\in V_i$ be a non-zero homogeneous element and $\widetilde{F}\in S_i$ be a lift. We may construct a basis of $V$ by extending $F$ and choose a set of lifts $\widetilde{F}_1,\cdots,\widetilde{F}_m$ extending $\widetilde{F}_1=\widetilde{F}$. By part $(1)$, we have $s(F)\geq s(\widetilde{F}-H_1)-\dim(U)$, for some $H_1\in U_i$. Now we have $s(\widetilde{F}-H_1)\geq B_i(\delta+\mu)\geq B_i(\mu)+\dim(U)$. Hence $s(F)\geq B_i(\mu)$ and thus $V$ is $B$-strong in $R$. 
   
   5. Let $\wt{a}_1,\cdots,\wt{a}_n\in S_1$ be homogeneous lifts of $a_1,\cdots,a_n$to $S$. Since $U$ is $B\circ t_n$-strong in $S$ and $\Kspan{a_1,\cdots,a_n}$ is infinitely strong, we know that the vector space $U+\Kspan{\wt{a}_1,\cdots,\wt{a}_n}$ is $B$-strong in $S$ by part (4). Hence, by \cref{corollary: strong sequence R_eta}, we know that any homogeneous basis of $U+\Kspan{\wt{a}_1,\cdots,\wt{a}_n}$ is an $\cR_\eta$-sequence, and in particular a prime sequence in $S$. Therefore, by \cref{proposition: strong sequence CM UFD}, we conclude that $a_1,\cdots,a_n$ is a prime sequence in $R$.
\end{proof}

\subsection{Lifted strength}

Suppose that $F_1,\cdots, F_m\in S$ are forms such that $V=\Kspan{F_1,\cdots,F_m}$ is $B$-strong in $S$, where $B$ satisfies the Ananyan-Hochster bound given by $B_i(\delta)\geq A(3,i)+3(\sum_i \delta_i-1)$. Then by \cref{corollary: strong sequence R_eta}, we know that $S/(V)$ is a UFD. We would like to obtain a similar statement where $V$ is a strong vector space in a finitely generated $\K$-algebra $R$. However, the analogous statement is not necessarily true even if $V$ is infinitely strong and $R$ is a Cohen-Macaulay UFD, as shown in the following example.

\begin{example} \label{example: AH in quotients counterexample}
    Let $F=x_1^2+\cdots+x_5^2$ and $R=S/(F)$. Then $R$ is a Cohen-Macaulay UFD. Consider the images $x_1,x_2\in R$. Then $V=\Kspan{x_1,x_2}$ is $B$-strong in $R$ for any function $B:\N\rightarrow \N$. However, $R/(V)=S/(x_3^2+x_4^2+x_5^2)$ is not a UFD.
\end{example}

The above example shows that a direct generalization of \cref{theorem: Ananyan-Hochster} or \cref{corollary: strong sequence R_eta} is not true for UFDs in general. However, note that if $R=S/(x_1^2+\cdots+x_7^2)$, then $R/(x_1,x_2)$ is a UFD. In this case $x_1,x_2, x_1^2+\cdots+x_7^2$ is a sequence strong enough so that the quotient of $S$ is a UFD. Motivated by this observation, we define the notion of lifted strength. We show that if $R$ is a finitely generated $\K$-algebra which is Cohen-Macaulay and a UFD, and $V$ is a lifted strong vector space, then the quotient $R/(V)$ is also a Cohen-Macaulay UFD.

\begin{definition}[Lifted strength] 
Let $U\subset S$ be a graded vector space and $R=S/(U)$. 
Let $F\in R_d$ be a non-zero form. 
We define the lifted strength of $F$ with respect to $U$ as $$\lsmin{U,F}:=\min\{\smin{U_d+\Kspan{\widetilde{F}}}\}$$
where $\widetilde{F}$ varies over all forms in $S_d$ such that the image of $\widetilde{F}$ in $R$ is $F$. 
Given a set of forms $F_1,\cdots,F_m\in R_d$, we define $$\lsmin{U,F_1,\cdots,F_m}=\min\{\smin{U_d+\Kspan{\widetilde{F}_1,\cdots,\widetilde{F}_m}}\},$$ where $\widetilde{F}_i$ varies over all forms in $S_d$ such that the image of $\widetilde{F}_i$ in $R$ is $F_i$. 
Given a non-zero vector space $V\subset R_d$, we define $$\lsmin{U,V}=\min\{\lsmin{U_d,F_1,\cdots,F_m}\},$$ where $F_1,\cdots,F_m$ vary over all possible bases of $V$. 
We say that $V\subset R_d$ is $k$-lifted strong with respect to $U$ if $\lsmin{U,V}\geq k$. 
For simplicity, we omit $U$ from the notation and write $\lsmin{V}$ when $U$ is clear from the context.

Let $V=\oplus_{i=1}^dV_i\subset R$ be a graded vector space. 
For any function, $B:\N^d\rightarrow \N^d$ we will say that $V$ is $B$-lifted strong with respect to $U$, if $V_i$ is $B_i(\dim(U_i)+\dim(V_i))$-lifted strong, i.e. $\lsmin{U,V_i}\geq B_i(\dim(U_i)+\dim(V_i))$ for all $i\in [d]$.  
In other words, $V$ is $B$-lifted strong with respect to $U$, if the vector space $U+\Kspan{\widetilde{F}_1,\cdots,\widetilde{F}_m}$ is $B$-strong in $S$, for any homogeneous basis $F_1,\cdots,F_m\in R$ of $V$ and any set of homogeneous lifts $\widetilde{F}_1,\cdots,\widetilde{F}_m\in S$.
\end{definition}

\begin{remark}\label{remark: lifted strong implies UFD}
Throughout this remark, we will have $R := S/(U)$.
\begin{enumerate}
    \item If $U=0$, then $V\subset S$ is $B$-lifted strong iff $V$ is $B$-strong.
    \item The two notions $B$-lifted strong and $B$-strong in $R$ are not equivalent in general. 
    For example, let $S=\K[x_1,\cdots,x_n,y_1,\cdots,y_{2m}]$ and $U =\Kspan{x_1,\cdots,x_n}$. 
    Let $F=y_1y_2+\cdots+y_{2m-1}y_{2m}$ and $V=\Kspan{F}\subset R$. 
    Then $s(F)=m-1$ and $V$ is $(m-1)$-strong in $R$. 
    Now if $B:\N^2\rightarrow \N^2$ is a function such that $B_2(n,1)>B_2(0,1)=m-1$, then $V$ is $B$-strong in $R$ but it is not $B$-lifted strong with respect to $U$.  
    We will show in \cref{proposition: 2B implies B lifted strong}, that under suitable conditions on the function $B$, we have that $B$-lifted strong implies $B$-strong in $R$.
    \item If $0 \neq U \subset S$ has dimension sequence $\delta$ and $V\subset R$ is a $B$-lifted strong vector space with dimension sequence $\mu$, then we must have that $U_i$ is $B_i(\delta+\mu)$-strong for all $i\in [d]$. 
    Hence, a $B$-lifted strong vector space of dimension sequence $\mu$ can exist in $R$ only if $U$ itself was $B(\delta+\mu)$-strong in $S$ to begin with. 
    In particular, if $B_i(\delta)\geq A(\eta,i)+3(\sum_i\delta_i-1)$ for some $\eta\geq 3$, by \cref{corollary: strong sequence R_eta}, any sequence of $\K$-linearly independent forms in $U$ is an $\cR_\eta$-sequence in $S$. 
    Thus, $R$ is a Cohen-Macaulay UFD. 
    In other words, if there exists a $B$-lifted strong vector space $V\subset R$, then $R$ is a Cohen-Macaulay UFD.
\end{enumerate}
\end{remark} 
 \begin{proposition}\label{proposition: 2B implies B lifted strong}
    Let $B:\N^d \rightarrow \N^d$ be an ascending function such that $B_i(\delta)\geq A(\eta,i)+3(\sum_i\delta_i-1)$ for some $\eta\geq 3$. Let $U\subset S$ be a graded vector space in $S$ with dimension sequence $\delta\in \N^d$. Let $R=S/(U)$ and $V\subset R$ be a graded vector subspace with dimension sequence $\mu\in \N^d$. 
    \begin{enumerate}
        \item We have $\smin{V_i}\leq \lsmin{U,V_i}$, for all $i \in [d]$
        \item If $V$ is $B$-lifted strong with respect to $U$, then $V$ is $B$-strong in $R$.
        \item If there exists a homogeneous basis $F_1,\cdots,F_m$ of $V$ and a set of homogeneous lifts $\widetilde{F}_1,\cdots,\widetilde{F}_m\in S$, such that  $U+\Kspan{\widetilde{F}_1,\cdots,\widetilde{F}_m}$ is $2B$-strong in $S$, then $V$ is $B$-lifted strong with respect to $U$.
        \item If $V\subset R$ is $B$-lifted strong with respect to $U$, then any sequence of $\K$-linearly independent homogenous elements of $V$ is an $\mathcal{R}_\eta$-sequence (and also a prime sequence in $R$), and $R/(V)$ is a Cohen-Macaulay UFD.
    \end{enumerate}
 \end{proposition}

\begin{proof}
    1. Let $F_1,\cdots,F_m$ of $V_i$ be a homogeneous basis and a set of homogeneous lifts $\widetilde{F}_1,\cdots,\widetilde{F}_m\in S$. 
    For any non-zero form  $F\in U+\Kspan{\widetilde{F}_1,\cdots,\widetilde{F}_m}$, we have $s(\overline{F})\leq s(F)$ by \cref{proposition: properties of quotient strength}.
    Thus we have $\smin{V}\leq \lsmin{U,V}$.
    
    2. Note that $B_i(\alpha+\beta)\geq B_i(\beta)+\sum_j\alpha_j$for all $\alpha,\beta\in \N^d$ and $i\in [d]$. Thus by \cref{proposition: properties of quotient strength} part $(4)$, we have $V$ is $B$-strong in $R$.
    
    3. Let $G_1,\cdots,G_m$ be another homogeneous basis of $V$ and  $\widetilde{G}_1,\cdots,\widetilde{G}_m\in S$ a set of homogeneous lifts. For any non-zero form $G\in U+\Kspan{\widetilde{G}_1,\cdots,\widetilde{G}_m}$ of degree $i$, we have that $G=F+H$, where $F\in \Kspan{\widetilde{F}_1,\cdots,\widetilde{F}_m}$ and $H\in (U)$ are homogeneous of degree $i$. We may write $H=H_1+H_2$, where $H_1\in U_i$ and $H_2\in (U_{<i})$ are homogeneous of degree $i$. Therefore, if $G$ has a $k$-collapse, then $F+H_1$ has an at most $(\dim(U)+k)$-collapse. Now $F+H_1\in U+\Kspan{\widetilde{F}_1,\cdots,\widetilde{F}_m}$, which is a $2B$-strong vector space of dimension sequence $\delta+\mu$. Hence, $\dim(U)+k\geq 2B_i(\delta+\mu)+1 \geq B_i(\delta+\mu)+\dim(U)+1$. Therefore we have $k\geq B_i(\delta+\mu)+1$. So we have $\lsmin{U,V_i}\geq B_i(\delta+\mu)$ for all $i\in[d]$ and hence $V$ is $B$-lifted strong.
    
    4.  If $V\subset R$ is $B$-lifted strong, then  $W:=U+\Kspan{\widetilde{F}_1,\cdots,\widetilde{F}_m}$ is $B$-strong in $S$ for any homogeneous basis $F_1,\cdots,F_m$ of $V$. 
    Therefore, by \cref{corollary: strong sequence R_eta}, any homogeneous basis of $W$ is an $\mathcal{R}_\eta$-sequence in $S$. 
    In particular, any homogeneous basis of $W$, extending $\widetilde{F}_1,\cdots,\widetilde{F}_m$ is an $\mathcal{R}_\eta$-sequence in $S$. 
    By \cref{proposition: strong sequence CM UFD}, $F_1,\cdots,F_m$ is an $\mathcal{R}_\eta$-sequence in $R=S/(U)$. 
    By \cref{proposition: strong sequence CM UFD} and $\eta \geq 3$, $S/(W)$ is a Cohen-Macaulay UFD. 
    Hence $R/(V)=S/(W)$ is a Cohen-Macaulay UFD.
\end{proof}

\subsection{Strengthening and Robustness}

We will show that given a graded vector space $U$, we can always find a strong AH-vector space $V$ such that $\K[U]\subset \K[V]$. Furthermore, we can construct $V$ in a way such that any element of $U$ that was \emph{sufficiently strong} is an element of $V$. We follow the proof of \cite[Theorem B]{AH20a} to construct such strong AH-algebras below.

\begin{lemma}[Strengthening of Algebras]\label{proposition: constructing AH algebras} For any $d\in \N$ and a function $B:\N^d\rightarrow N^d$,
   there exist functions  $C_B:\N^d\rightarrow\N^d$ and $h_B:\N^d\rightarrow\N^d$, depending on $B$, such that the following holds:
    
    Given a graded vector space  $U=\oplus_{i=1}^dU_i\subset S$ with dimension sequence $\delta\in \N^d$, there exists a $B$-strong AH vector space $V=\oplus_{i=1}^d V_i$ such that 
    \begin{enumerate}
        \item $\K[U]\subset \K[V]$,
        \item for all $i\in [d]$, we have $\dim(V_i)\leq C_{B,i}(\delta)$, where $C_{B,i}$ denotes the $i$-th component of $C_B=(C_{B,1},\cdots, C_{B,d}):\N^d\rightarrow \N^d$.
    \end{enumerate}
    
Furthermore, suppose   $H=\oplus_{i=1}^dH_i\subset U$ is a graded subspace such that $\smin{H_i}\geq h_{B,i}(\delta)$ for all $i\in [d]$. 
Then there exists a $B$-strong AH vector space $V$  satisfying (1) and (2) above such that $H\subset V$.
\end{lemma}

\begin{proof}
    For any graded vector space $U=\oplus_{i=1}^d U_i$ with dimension sequence $\delta=(\delta_1,\cdots,\delta_d)\in \N^d$, we define an iterative process below.
    We will construct $V$ as an outcome of this iterative process.
    \begin{enumerate}
        \item Set $V=\oplus_{i=1}^dV_i:=U$, with $U_i=V_i$ for all $i\in [d]$.
        \item While $V$ is not $B$-strong:
        \begin{itemize}
            \item Let $\ell:=\max\{i\in [d]| \smin{V_i}<B_i(\delta)\}$.
        \item Pick $P\in V_\ell\setminus (0)$ such that $s(P)=s-1<B_i(\delta)$. Choose a basis $\cF$ of $V_\ell$ that contains $P$. Let $P=f_1g_1+\cdots+f_sg_s$ be a $s$-collapse. 
        \item Set $V_\ell\leftarrow\Kspan{\cF\setminus \{P\}}$ and $\oplus_{i<\ell}V_i\leftarrow \oplus_{i<\ell}V_i+\Kspan{f_1,\cdots,f_s,g_1,\cdots,g_s}$ as a graded vector space .
       \item Set $V\leftarrow \oplus_{i=1}^dV_i$. \end{itemize}
    \end{enumerate}
Note that after each iteration of the while loop we preserve the condition $\K[U]\subset \K[V]$. Consider the reverse lexicographic ordering on the dimension sequences $\delta\in \N^d$. Note that the reverse lexicographic ordering is a well-ordering. We will inductively define the functions $C_{B,i}(\delta)$ and prove by induction on $\delta$ that the iterative process terminates after finitely many steps with the desired bound on dimension. If $d=1$, then $U$ is $B$-strong for any function $B$. Therefore the iterative process ends at the first step with $V=U$ and we may take $C_{B,1}(\delta)=\delta=\dim(U)$ in this case. Thus the statement holds for any $d\in \N$ and any dimension sequence of the form $\delta=(\delta_1,0,\cdots,0)\in \N^d$ with $C_{B,i}(\delta)=\delta_i$ for all $i\in [d]$.  

 Let $d\geq 2$. 
Suppose that the iterative process terminates for all $\delta'<\delta$, and $\dim(V)$ is bounded by a function $C_{B,i}(\delta')$ defined for all $\delta'<\delta$. 
If $U$ is $B$-strong, the iterative process stops at the first step and we may take $V=U$. 
If $U$ is not $B$-strong, then let $\ell\in [d]$ be as in the definition of the while loop above. 
Then after one iteration of the while loop we see that $\dim(V_{>\ell})$ remains unchanged, $\dim(V_\ell)$ drops by $1$ and $\dim(V_{<\ell})$ increases by at most $2B_\ell(\delta)$. 
Thus, after the first iteration of the while loop we end up with a vector space $V$ with dimension sequence $\delta'<\delta$ and $\sum_{i=1}^d\delta'_i< \sum_{i=1}^d\delta_i+2B_\ell(\delta)$. 
Therefore, by induction the iterative process stops after finitely many steps. 
Note that $\dim(V_d)$ never increases throughout the process, so it is enough to take $C_{B,d}(\delta)=\delta_d$. 
For $i<d$, we define $C_{B,i}(\delta)=\max\{\delta_i,C_{B,i}(\delta')\}$ where $\delta'$ varies over the set of $\delta'<\delta$ such that $\sum_{i=1}^d\delta'_i< \sum_{i=1}^d\delta_i+2B_{\max}(\delta)$ with $B_{\max}(\delta)=\max\{B_i(\delta)|i\in [d]\}$. 

We will define the function $h_B$ inductively and show that if $\smin{H_i}\geq h_{B,i}(\delta)$ then we may choose $P$ and $\cF$ at each step of the iterative process such that the condition $H\subset V$ is preserved. Suppose $d=1$, then $U=V$ and we have $H\subset 
V=U$. So we may take $h_{B,1}(\delta)=B_1(\delta)$ in this case.
Suppose $d\geq 2$ and $\delta=(\delta_1,0\cdots,0)$. Then again $H\subset U=V$ and we may take $h_{B,i}(\delta_1,0,\cdots,0)=B_i(\delta)$ for all $i$. In general, we define $h_{B,i}(\delta)=\max\{B_i(\delta),h_{B,i}(\delta')\}$ where $\delta'$ varies over the set of $\delta'<\delta$ such that $\sum_{i=1}^d\delta'_i< \sum_{i=1}^d\delta_i+2B_{\max}(\delta)$.

Let $d\geq 2$. Suppose that for any $\delta'<\delta$ and a subspace $H\subset U$ with $\smin{H_i}\geq h_{B,i}(\delta')$, we can run the iterative process such that at each step we preserve the condition $H\subset V$, and the process terminates with $H\subset V$. Let $U$ be a graded vector space of dimension sequence $\delta$ and $H\subset U$ a graded subspace such that $\smin{H_i}\geq h_{B,i}(\delta)$. We fix a basis $\cF_H$ of $H$. If $U$ is $B$-strong then we have $H\subset U=V$. If $U$ is not $B$-strong, let $\ell\in [d]$ be as in the definition of the while loop above. Note that $\smin{H_\ell}\geq h_{B,\ell}(\delta)\geq  B_\ell(\delta)$. Therefore, we must have $P\not \in H_\ell$ in the iterative process above. Thus we may choose a basis $\cF$ of $V_\ell$, containing $P$ such that it also contains the basis $\cF_H$ of $H$. Then, at the end of this iteration of the while loop, we end up with a vector space $V$ with dimension sequence $\delta'<\delta$ and $\sum_{i=1}^d\delta'_i< \sum_{i=1}^d\delta_i+2B_{\max}(\delta)$. We see that $H_i\subset V_i$ for all $i\in [d]$, by the choice of the basis $\cF$. Now, $\smin{H_i}\geq h_{B,i}(\delta)\geq h_{B,i}(\delta')$ for all $i\in [d]$. Therefore, by induction, we may run the iterative process such that it  terminates with a vector space $V$ with $H\subset 
V$.
\end{proof}

\begin{corollary}[Robustness of strong algebras]\label{corollary: robust strong algebra} Let $B,G:\N^d\rightarrow\N^d$ and $\mu\in\N^d$. Suppose that $B_i(\delta)\geq h_{G,i}(\delta+\mu)$ for all $\delta\in \N^d$ and $i\in [d]$, where $h_G:\N^d\rightarrow \N^d$ is the function defined in \cref{proposition: constructing AH algebras}. Let $U\subset S$ be a $B$-strong AH vector space and $W\subset S$ is a graded vector space with dimension sequences $\delta$ and $\mu$ respectively. Then there exists a $G$-strong AH vector space $V$ such that 
\begin{enumerate}
    \item $\K[U+W]\subset\K[V]$,
    \item $U\subset V$,
    \item for all $i\in [d]$, $\dim(V_i)\leq C_{G,i}(\delta+\mu)$, where $C_G:\N^d\rightarrow \N^d$ is the function defined in \cref{proposition: constructing AH algebras}.
\end{enumerate}
\end{corollary}
\begin{proof}
    Consider the graded vector space $U'=U+W$. Then we have that $\delta'_i:=\dim(U'_i)\leq \delta_i+\mu_i$ for all $i\in [d]$. Since $U$ is $B$-strong, we have $\smin{U_i}\geq B_i(\delta)$ for all $i$. Now $U\subset U'$ is a graded subspace such that $\smin{U_i}\geq B_i(\delta)\geq h_{G,i}(\delta+\mu)\geq h_{G,i}(\delta')$. Let $V$ be the $G$-strong AH vector space constructed by applying \cref{proposition: constructing AH algebras} to the vector space $U'$ along with the subspace $H=U$. Then, by \cref{proposition: constructing AH algebras}, we know that $\K[U']\subset \K[V]$ and $\dim(V_i)\leq C_{G,i}(\delta')\leq C_{G,i}(\delta+\mu)$. Furthermore, we have $U\subset V$, since $\smin{U_i}\geq h_{G,i}(\delta')$ for all $i\in [d]$.
\end{proof}

\begin{corollary}\label{corollary: practical robustness }
 Let $B:\N^d\rightarrow \N^d$. Let $U\subset S$ be a graded vector space with dimension sequence $\delta_U\in \N^d$ and let $R=S/(U)$. Let $V\subset R$ is a graded vector space with dimension sequence $\delta_V\in \N^d$. Suppose $V$ is  $h_{2B}\circ t_k$-lifted strong with respect to $U$. Let $P_1,\cdots,P_k\in R_{\leq d}$ be homogeneous elements. Then there exists a graded vector space $V'\subset R_{\leq d}$ such that:

\begin{enumerate}
    \item $V'$ is $B$-lifted strong with respect to $U$.
    \item $P_1,\cdots,P_k\in \K[V']$.
    \item $V\subset V'$.
    \item for all $i\in [d]$, we have $\dim(V'_i)\leq C_{2B,i}(t_k(\delta_U+\delta_V))-\delta_{U,i}$.
\end{enumerate}
\end{corollary}

\begin{proof}
     Let $m=\delta_V$ and $\widetilde{F}_1,\cdots,\widetilde{F}_m$ be homogeneous lifts of a homogeneous basis of $V$. Let $\widetilde{P}_1,\cdots,\widetilde{P}_k\in S$ be homogeneous lifts of $P_1,\cdots,P_k$. Note that $U+\Kspan{\widetilde{F}_1,\cdots,\widetilde{F}_m}$ is $h_{2B}\circ t_k$-strong in $S$, as $V$ is $h_{2B}\circ t_k$-lifted strong in $R$. By applying \cref{corollary: robust strong algebra} to the vector space $U+\Kspan{\widetilde{F}_1,\cdots,\widetilde{F}_m,\widetilde{P}_1,\cdots,\widetilde{P}_k}$, we may construct a $2B$-strong vector space $V''\subset S$ such that $V''$ contains $U+\Kspan{\widetilde{F}_1,\cdots,\widetilde{F}_m}$ and $\wt{P}_1,\cdots,\wt{P}_k\in \K[V'']$. Let $V'\subset R$ be the image of $V''$ in $R$. Then $V\subset V'$ and $P_1,\cdots,P_k\in \K[V']$. By \cref{proposition: 2B implies B lifted strong}, we know that $V'$ is $B$-lifted strong in $R$. Furthermore, we know that $\dim(V''_i)\leq C_{2B,i}(t_k(\delta_U+\delta_V))$. Therefore, we have $\dim(V'_i)\leq C_{2B,i}(t_k(\delta_U+\delta_V))-\delta_{U,i}$.  
\end{proof}

\begin{proposition}\label{proposition: three forms regular sequence}
   Let $\eta\geq 3$. Let $U\subset S_{\leq d}$ be a graded vector space and let  $R=S/(U)$. Let $V\subset 
    R_{\leq d}$ be a graded vector space. Suppose that $V$ is $h_{2B}\circ t_1$-lifted strong with respect to $U$, where $B:\N^d\rightarrow \N^d$ is such that $B_i(\delta)\geq A(\eta,i)+3(\sum_i\delta_i-1)$ for all $i\in [d]$. Let $F,G,H\in R_{\leq d}$ be homogeneous elements of positive degree. Suppose $F,G\in \K[V]$ and $F,G$ do not have any common factors in $R$. If $H\not \in (V)$ then $F,G,H$ is a regular sequence in $R$.
\end{proposition}

\begin{proof}
     Let $V'\subset R$ be the $B$-lifted strong graded vector space constructed by applying \cref{corollary: practical robustness } to $V$ and $H$. In particular, $V\subset V'$ and $H\in \K[V']$. Note that, by \cref{remark: lifted strong implies UFD}, we know that any homogeneous basis of $U$ is an $\cR_\eta$-sequence in $S$ and that $R$ is a graded UFD. Furthermore, by \cref{proposition: 2B implies B lifted strong}, we know that any homogeneous basis of $V$ or $V'$ is an $\cR_\eta$-sequence (and also a prime sequence) in $R$. 

 Note that $F,G$ is a regular sequence in $R$, by \cref{lemma: radical regular sequence}. Since $R$ is Cohen-Macaulay by \cref{corollary: strong sequence R_eta}, we know that $(F,G)$ does not have any embedded primes in $R$. Since $H\not\in (V)$, we conclude that $H$ is a non-zero divisor in $R/(F,G)$, by \cref{corollary: dominant in quotient}. Therefore, $F,G,H$ is a regular sequence in $R$.
 
\end{proof}

\begin{lemma}\label{lemma: d^2 minimal primes}
Let $d,e,\eta\in \N$ with $1\leq d \leq e$ and $\eta\geq 3$. Let $U\subset S_{\leq e}$ be a $h_{2B}\circ t_2$-strong vector space in $S$, where $B:\N^e\rightarrow \N^e$ is such that $B_i(\delta)\geq A(\eta,i)+3(\sum_i\delta_i-1)$ for all $i\in [e]$. Consider the unique factorization domain $R=S/(U)$. Let $P,Q,P_1,\cdots,P_m\in R_{\leq d}$ be  homogeneous elements of positive degree where $m\geq 2^{d^2}$ and $\gcd(P,Q)=\gcd(P,P_i)=\gcd(Q,P_i)=1$ for all $i\in[m]$. If $Q\in \rad(P,P_i)$ for all $i\in [m]$, then we must have $\rad{(P,P_i)}=\rad(P,P_k)$ for two distinct $i,k\in [m]$.
\end{lemma}

\begin{proof}
Note that $P,Q$ and $P,P_i$ are regular sequences in $R$.
We have  $\rad(P,Q)\subset \rad(P,P_i)$ for all $i\in [m]$, since $Q\in \rad(P,P_i)$. Let $\mathcal{S}=\{\mathfrak{p}_1,\cdots\mathfrak{p}_\ell\}$  and $\mathcal{S}_i=\{\mathfrak{p}_{i1},\cdots, \mathfrak{p}_{i\ell_i}\}$ be the set of minimal primes over $(P,Q)$ and $(P,P_i)$ in $R$, respectively. We have $\bigcap \mathfrak{p}_j=\rad(P,Q)\subset \rad(P,P_i)=\bigcap_j \mathfrak{p}_{ij}$. Therefore, by \cite[Proposition 1.11]{AM69}, we must have that for all $i,j$, there exists $k$ such that $\mathfrak{p}_k\subset \mathfrak{p}_{ij}$. Since we know that $\mathrm{ht}(\mathfrak{p}_j)=\mathrm{ht}(\mathfrak{p}_{ij})=2$ for all $i,j$, we must have that for all $i,j$, there exists $k$ such that $\mathfrak{p}_k= \mathfrak{p}_{ij}$, i.e. $\mathcal{S}_i\subset \mathcal{S}$ for all $i \in [m]$. Therefore it is enough to show that there exist at most $d^2$ minimal primes of $(P,Q)$. 
Indeed, suppose $|\cS|\leq d^2$, then there are at most $2^{d^2}-1$ number of distinct choices for the set of minimal primes $\mathcal{S}_i=\{\mathfrak{p}_{ij}\}$ of the ideals $(P,P_i)$. 
By the pigeonhole principle, there exist distinct $i,k\in [m]$ such that $\mathcal{S}_i=\mathcal{S}_k$. 
Thus, $\rad(P,P_i)=\bigcap_{\mathfrak{q} \in \mathcal{S}_i}\mathfrak{q} = \bigcap_{\mathfrak{p}\in \mathcal{S}_k}\mathfrak{p} =\rad(P,P_k)$. 

Now we will show that $|\cS|\leq d^2$. Since $U$ is $h_{2B}\circ t_2$-strong in $S$, we note that the vector space $(0)\subset R$ is $h_{2B}\circ t_2$-lifted strong. Hence, by \cref{corollary: practical robustness }, there exists a $B$-lifted strong vector space $W\subset R$ such that $P,Q\in \K[W]$. Let $A:=\K[W]$. Note that \cref{proposition: 2B implies B lifted strong}, any homogeneous basis of $W$ is an $\cR_\eta$-sequence (and also a prime sequence) in $R$. By \cref{proposition: radicals and prime sequences}, we know that the minimal primes over $(P,Q)$ in $A$ are in one-to-one correspondence with the minimal primes over $(P,Q)$ in $R$. Now $A$ is isomorphic to a polynomial ring and $P,Q$ are of degree at most $d$ in $A$. Note that $P,Q$ do not have any common factors in $A$, since they do not have any common factors in $R$. Also, $(P,Q)\neq A$, since $P,Q \in (W)$ as they are homogeneous. Therefore, by \cref{lemma: radical regular sequence}, we have that $P,Q$ is a regular sequence in $A$. Therefore, by \cref{lemma: d^2 minimal primes basic}, we conclude that the number of minimal primes over $(P,Q)$ in $A$ is at most $d^2$.
 
\end{proof}

\section{General Quotients}\label{sec:general-quotients}

As discussed in the introduction, our inductive strategy for proving \cref{theorem: radical SG theorem over UFD} is to reduce the degree of the SG-configuration by applying a sequence of quotient homomorphisms $R\rightarrow R[z]/(W)$, where $W$ is a sufficiently strong vector space. In this section, we describe general quotients and prove the useful properties we need from them in order to reduce the degree of radical Sylvester-Gallai configurations.
The quotients defined here are generalizations of the projection maps used in \cite{S20, PS22, OS22, GOS22}.

Throughout this section, we fix positive integers $d,\eta\in \N$ with $\eta\geq 3$, and $B:\N^d\rightarrow \N^d$ denotes an ascending function such that $B_i(\delta)\geq A(\eta,i)+3(\sum_i\delta_i-1)$ for all $i\in [d]$, where $A(\eta,i)$ is the function provided by \cref{theorem: Ananyan-Hochster}.

\begin{definition}[Graded Quotients]\label{definition: quotients}
Let $U=\oplus_{i=1}^d U_i\subset S$ be a graded vector space of dimension sequence $\delta$ in $S$ and $R := S/(U)$ be the quotient ring. 
Let $V=\oplus_{i=1}^dV_i \subset R$ be a graded subspace of dimension sequence $\mu$.

Let $F_1, \ldots, F_n$ be a homogeneous basis for $V$ and $z$ be a new variable. 
For $\alpha \in \K^n$, let $V_\alpha := \Kspan{F_1 - \alpha_1 z^{\deg(F_1)}, \dots, F_n - \alpha_n z^{\deg(F_n)}}$ and $I_\alpha\subset R[z]$ be the homogeneous ideal $I_\alpha = (V_\alpha)$. 
We define the graded quotient map $\varphi_{V, \alpha}$ as the quotient homomorphism of finitely generated graded $\K$-algebras given by
$$ \varphi_{V, \alpha} : R[z] \to R[z]/I_\alpha. $$
\end{definition}

For simplicity we will often drop the subscripts $V$ or $\alpha$, and write  $\varphi_{\alpha}$ or $\varphi$ for our quotient map when there is no ambiguity about the vector space $V$ or the vector $\alpha$. 

\begin{remark}
In the above definition we abuse notation and define $V_\alpha$ with a fixed basis in mind. 
This abuse of notation is only to simplify our definition of the quotients we will use, and the choice of basis is not very important, since \cite{AH20a} shows that any basis of $V_\alpha$ forms an $\cR_\eta$ sequence.
\end{remark}

\begin{proposition}\label{proposition: graded quotient is UFD}
Suppose $V\subset R$ is $B$-lifted strong with respect to $U$. Then $R[z]$ and $R[z]/I_\alpha$ are quotients of $S[z]$ by $\cR_\eta$-sequences, for any choice of $\alpha \in \K^n$. 
In particular, they are Cohen-Macaulay UFDs.
    
\end{proposition}

\begin{proof}
    Note that by \cref{remark: lifted strong implies UFD}, we know that any homogeneous basis of $U$ is an $\cR_\eta$-sequence in $S$. Therefore, if $G_1,\cdots,G_t$ is a homogeneous basis of $U$ in $S$, then $G_1,\cdots,G_t,z$ is an $\cR_\eta$-sequence in $S[z]$.  Let $\widetilde{F}_1,\cdots,\widetilde{F}_n\in S$ be homogeneous lifts of a homogeneous basis of $V$. Then $U+\Kspan{\widetilde{F}_1,\cdots,\widetilde{F}_n}$ is $B$-strong. Let $H_i=\widetilde{F}_i-\alpha_i z^{\deg(F_i)} \in S[z]$. Then by \cref{proposition: strong with z}, we have that $W:=U+\Kspan{H_1,\cdots,H_n}$ is $B$-strong in $S[z]$. Thus, any basis of $W$ consists of an $\cR_\eta$ sequence in $S[z]$. Then by \cref{corollary: strong sequence R_eta}, we know that $R=S/(U)$, $R[z]=S[z]/(U)$ and $R[z]/I_\alpha\simeq S[z]/(W)$ are all Cohen-Macaulay UFD.
\end{proof}

\emph{General quotients}. We say that a property $\mathcal{P}$ holds for a  \emph{general} $ \alpha \in \K^n$, if there exists a non-empty open subset $\cX \subset\K^n$ such that the property $\mathcal{P}$ holds for all $\alpha\in \cX$. Here $\cX \subset \K^n$ is open with respect to the Zariski topology. 
Hence $\cX$ is the complement of the zero set of finitely many polynomial functions on $\K^n$. 

In particular, given a graded vector space $V \subset R$ as in \cref{definition: quotients}, 
we will say that a property $\mathcal{P}$ holds for a \emph{general quotient} $\varphi_{\alpha}$, if there exists a non-empty open subset $\cX \subset\K^n$ such that the property holds for all $\varphi_{\alpha}$ with $\alpha\in \cX$. The general choice of the element $\alpha$ defining a general homomorphism $\varphi_{\alpha}$ allows us to say that such homomorphisms will avoid any finite set of polynomial constraints. As shown in \cite{S20, PS22, OS22} in the case of general projections, these maps preserve several important properties of polynomials.
In this section we generalize such properties to the setting of graded quotients.
\\

Suppose $V\subset R$ is a $B$-strong vector space. Then the subalgebra $\K[V]$ is isomorphic to a polynomial ring $\K[y_1,\cdots,y_n]$, where $n=\dim(V)$. For $\alpha\in \K^n$, consider the quotient homomorphism $\varphi_\alpha:R\rightarrow R[z]/(V_\alpha)$ as defined above. The restriction of $\varphi_\alpha$ to the subalgebra $\K[V]$ is given by $y_i\rightarrow \alpha_iz^{d_i}$ under the isomorphism with $\K[y_1,\cdots,y_n]$. Therefore, we first consider such non-homogeneous quotients of a polynomial ring. 
\begin{lemma}\label{lemma: general projection basic}
Let $S=\K[x_1,\cdots,x_N]$ and $z$ be a new variable. Fix positive integers $d_1,\cdots,d_n\in \N$. For $\alpha\in \K^n$, let $I_\alpha=(x_1-\alpha_1 z^{d_1},\cdots,x_n-\alpha_nz^{d_n})$. Let $\varphi_\alpha:S[z]\rightarrow S[z]/I_\alpha$ be the quotient ring homomorphism. 

\begin{enumerate}
    \item The ideal $I_\alpha$ is prime in $S[z]$, and the composition morphism $\K[z]\hookrightarrow S[z]\rightarrow S[z]/I_\alpha$ is injective.
    \item  If $F\in \K[x_1,\cdots,x_n]$ is a non-zero polynomial, then $\varphi_\alpha(F)\neq 0$ in $S[z]/I_\alpha$ for a general $\alpha\in \K^n$.
    \item Let $F\in S\setminus \K[x_1,\cdots,x_n]$, then $\varphi_\alpha(F)\not\in \K[z]$ in $S[z]/I_\alpha$, for a general $\alpha\in \K^n$.
    \item If $F\in S$ is a non-zero polynomial, then $\varphi_\alpha(F)\neq 0$ in $S[z]/I_\alpha$ for a general $\alpha\in \K^n$.
    \item If $F,G\in S$ have no common factor, then $\gcd(\varphi_\alpha(F),\varphi_\alpha(G))\in\K[z]$ for a general $\alpha\in \K^n$. 
    \item If $F\in S$ is square-free. 
    Then, for a general $\alpha\in \K^n$, the multiple factors of $\varphi_\alpha(F)$ must be in $\K[z]$.
\end{enumerate}
\end{lemma}

\begin{proof}
1. Consider the ring homomorphism $\vartheta: S[z]\rightarrow \K[x_{m+1},\cdots,x_N,z]$ defined by $x_i\mapsto \alpha_iz^{d_i}$ for $i\in[m]$, and $x_i\mapsto x_i$ for $m+1\leq i \leq N$ and $z\mapsto z$. 
Note that $\vartheta$ is surjective and $\mathrm{ker}(\varphi)=I_\alpha$. Therefore $I_\alpha$ is a prime ideal and $S[z]/I_\alpha\simeq \K[x_{m+1},\cdots,x_N,z]$. 

Suppose $F\in \K[z]$ such that $\varphi_\alpha(F)=0$. Then $F\in \K[z]\cap I_\alpha$ in $S[z]$. 
Since $I_\alpha\subset(x_1,\cdots,x_n,z)$, we note that if $g\in \K[z]$ is a polynomial with non-zero constant term, then $g\not\in I_\alpha$. 
If $F\neq 0$, let $F=z^rg(z)$, where $r$ is the highest power of $z$ diving $F$, and $g(0)\neq 0$. 
Therefore, we have that $z^r\in I_\alpha$, as $I_\alpha$ is a prime ideal. 
Hence, $z\in I_\alpha$ and thus $(x_1,\cdots,x_n,z)\subset I_\alpha$, which implies $\mathrm{ht}(x_1,\cdots,x_n,z)=n+1$. 
However, $I_\alpha$ is generated by $n$ polynomials, and thus $\mathrm{ht}(I_\alpha)\leq n$.
This is a contradiction.

2. We have $\varphi_\alpha(F)=F(\alpha_1z^{d_1},\cdots,\alpha_n z^{d_n})$ is a polynomial in $\K[z]\subset S[z]/I_\alpha$. 
If $\varphi_\alpha(F)=0$ in $\K[z]$, then by setting $z=1$, we obtain $F(\alpha)=0$. 
Therefore if $\alpha\in \K^n\setminus V(F)$, then $\varphi_\alpha(F)\neq 0$.

3. We may write $F$ as a polynomial in $x_{m+1},\cdots,x_N$ with coefficients in $A:=\K[x_1,\cdots,x_n]$. 
Since $F\not\in A$, we know that there is a non-zero polynomial $g(x_1,\cdots,x_n)\in A$ such that $g$ is the coefficient of some monomial $\Pi_{j=m+1}^N x_j^{e_j}$ in $F$. 
By part (2), we know that $\varphi_\alpha(g)$ is a non-zero polynomial in $\K[z]$, for general $\alpha\in \K^n$. 
Therefore, for a general $\alpha\in \K^n$, the coefficient of $\Pi_{j=m+1}^N x_j^{e_j}$ in $\varphi_\alpha(F)$ is $\varphi_\alpha(g)\neq 0$. Hence $\varphi_{\alpha}(F)\not \in \K[z]$ for a general $\alpha\in \K^n$.

4. Follows immediately from parts (2) and (3) above.

5. If $F,G$ have no common factor in $S$, then $\mathrm{Res}_{x_i}(F,G)\neq 0$ for all $x_i$. By, part (4), we know that for a general $\alpha\in \K^n$, we have $\varphi_\alpha(\mathrm{Res}_{x_i}(F,G))\neq 0$ for all $i$. 
Note that $S[z]/I_\alpha=\K[x_{m+1},\cdots,x_N,z]$ for a general $\alpha\in \K^n$. 
Now, $\varphi_\alpha(\mathrm{Res}_{x_i}(F,G))=\mathrm{Res}_{x_i}(\varphi_\alpha(F),\varphi_\alpha(G))$ for $m+1\leq i\leq N$. 
Therefore, $\mathrm{Res}_{x_i}(\varphi_\alpha(F),\varphi_\alpha(G))\neq 0$ for $m+1\leq i \leq N$. 
Hence, $\varphi_\alpha(F),\varphi_\alpha(G)$ cannot have any common factor depending on $x_i$ in $\K[x_{m+1},\cdots,x_N,z]$. Hence, $\gcd(\varphi_\alpha(F),\varphi_\alpha(G))\in \K[z]$.

6. Since $F$ is square-free, we know that $F$ and $\frac{\partial F}{\partial x_i}$ do not have a common factor for any $i\in [N]$. 
Note that we have $\varphi_\alpha(\frac{\partial F}{\partial x_i})=\frac{\partial \varphi_\alpha(F)}{\partial x_i}$ for $m+1\leq i \leq N$. 
If $\varphi_\alpha(F)$ has a multiple factor that is not in $\K[z]$, then there exists $k\in [m+1,N]$ such that $\frac{\partial \varphi_\alpha(F)}{\partial x_k}$ and $\varphi_\alpha(F)$ have a common factor that depends on $x_k$. 
However, by part (5), we know that $\varphi_\alpha(\frac{\partial F}{\partial x_k})$ and $\varphi_\alpha(F)$ do not have a common factor depending on $x_k$ , for a general $\alpha\in \K^n$.
\end{proof}

\begin{proposition}\label{proposition: general quotient general}
Let $V \subset R$ be a $B$-lifted strong vector space and $\varphi_\alpha :R[z]\rightarrow R[z]/I_\alpha$ be a graded quotient as defined in \cref{definition: quotients}.
\begin{enumerate}
    \item The ideal $I_\alpha$ is a prime ideal in $R[z]$ and the composition $\K[z]\hookrightarrow R[z]\rightarrow R[z]/I_\alpha$ is injective.
    \item  If  $F \in R \setminus \{0\}$, then $\varphi_\alpha(F)\neq 0$ for a general $\alpha\in \K^n$.
    \item  If $F\not\in \K[V] \subset R$, then $\varphi_\alpha(F)\not\in \K[z]$ for a general $\alpha\in \K^n$.
    
\end{enumerate}
\end{proposition}

\begin{proof}
Let $F_1,\cdots,F_n$ be the homogeneous basis of $V$ used in the quotient map $\varphi_\alpha$. By \cref{proposition: 2B implies B lifted strong}, we know that  $F_1,\cdots,F_n$ is an $\cR_\eta$-sequence in $R$. 
Thus, $F_1,\cdots,F_n,z$ is an $\cR_\eta$-sequence in $R[z]$. We may extend $F_1,\cdots,F_n$ to a homogeneous system of parameters $F_1,\cdots,F_n$ of $R$. 
Let $A=\K[F_1,\cdots,F_n]$, $G_i=F_i-\alpha_iz^{\deg(F_i)}\in A[z]$ for 
$i\in [n]$ and $J_\alpha=(G_1,\cdots,G_n)\subset A[z]$. 
Let $\vartheta_\alpha: A[z]\rightarrow A[z]/J_\alpha$ be the quotient map. 
We have that $A\subset R$ and $A[z]\subset R[z]$ are module-finite free extensions, by \cref{theorem: CM hsop}. 
Now we prove the desired statements below.

1. Note that $R[z]$ is the quotient of $S[z]$ by an ideal generated by an $\cR_\eta$-sequence, by \cref{proposition: graded quotient is UFD}.  
Note that, $J_\alpha$ is a prime ideal in $A[z]$, by \cref{lemma: general projection basic}. 
Therefore, by \cref{proposition: radicals and prime sequences}, $I_\alpha := J_\alpha R[z]$ is prime in $R[z]$ and $I_\alpha\cap A[z]= J_\alpha$ in $A[z]$.  In particular, $A[z]/J_\alpha\rightarrow R[z]/I_\alpha$ is injective. 
By \cref{lemma: general projection basic}, $\K[z]\hookrightarrow A[z]\rightarrow A[z]/J_\alpha$ is injective. Therefore, $\K[z]\hookrightarrow R[z]\rightarrow R[z]/I_\alpha$ is injective.

2.  Suppose $\mathrm{rank}(R)=r$ as a free $A$-module. 
By tensoring with $A[z]$, we obtain that $A[z]\subset R[z]$ is a free extension of rank $r$. 
Therefore, $R[z]\xrightarrow[]{\sim}A[z]^{\oplus r}$ as $A[z]$-modules. 
Hence, 
$R[z]/I_\alpha\simeq R[z]\otimes_{A[z]} (A[z]/J_\alpha)$ is also a free $A[z]/J_\alpha$-module of rank $r$. Therefore, we have 
$R[z]/I_\alpha\xrightarrow[]{\sim}(A[z]/J_\alpha)^{\oplus r}$. Let $F\in R\subset R[z]$ have components $(f_1,\cdots,f_r)\in (A[z])^{\oplus r}$ under this isomorphism. Note that $\varphi_\alpha: R[z]\rightarrow R[z]/I_\alpha$ commutes with the quotient map $\vartheta_\alpha^{\oplus r}:(A[z])^{\oplus r}\rightarrow (A[z]/J_\alpha)^{\oplus r}$ via the isomorphisms above, i.e. the following diagram commutes.

\[\begin{tikzcd}
 R[z] \arrow[r, "\varphi_\alpha"] \arrow[d, "\simeq"] & R[z]/I_\alpha \arrow[d,"\simeq"]\\
 (A[z])^{\oplus r} \arrow[r, "(\vartheta_\alpha)^{\oplus r}"] & (A[z]/J_\alpha)^{\oplus r} \\
\end{tikzcd}\]

If $F\neq 0$, then $f_i\neq 0$ for some $i\in [r]$. Now, by \cref{lemma: general projection basic}, we know that for a general $\alpha\in \K^n$, the image $\vartheta_\alpha(f_i)\neq 0$. Since $\vartheta_\alpha(f_i)\neq 0$, we see that $\varphi_\alpha(F)\neq 0$ for a general $\alpha\in \K^n$, by the commutative diagram above.

3. Note that, by \cref{proposition: hsop injection}, we may assume that the inclusion $A[z]\subset R[z]$ is an isomorphsim onto one of the direct summands of $A[z]^{\oplus r}$ via the isomorphism $R[z]\simeq A[z]^{\oplus r}=A[z]\cdot e_1\oplus \cdots \oplus A[z]\cdot e_r$.  Without loss of generality, we may assume that it is an isomorphism with the first direct summand $A[z]\cdot e_1$. Tensoring the isomorphism with $A[z]/J_\alpha$, we see that the inclusion $A[z]/J_\alpha\subset R[z]/I_\alpha$ is an injection onto the first direct summand of $(A[z]/J_\alpha)^{\oplus r}$. Let $F\in R\subset R[z]$ have components $(f_1,\cdots,f_r)\in (A[z])^{\oplus r}$ under this isomorphism, as above. If $F\in R$ and $F\not \in A$, then we must have that $f_i\neq 0$ for some $i>1$. Therefore, we must have that $\vartheta_\alpha(f_i)\neq 0 $ for a general $\alpha\in \K^n$, by \cref{lemma: general projection basic}. Therefore, $\varphi_\alpha(F)\not \in (A[z]/J_\alpha)\cdot e_1\subset (A[z]/J_\alpha)^{\oplus r}$. Now, note that $\K[z]\subset A[z]\subset R[z]$, therefore in the quotient, we have $\K[z]\subset (A[z]/J_\alpha)\cdot e_1$. Therefore, $\varphi_\alpha(F)\not \in \K[z]$ for a general $\alpha\in \K^n$. 
\end{proof}

The next proposition is a generalization of \cite[Claim 2.23]{PS22}. Given $F,G\in R$, we show that under a general quotient $\varphi_\alpha$, the common factors of $\varphi_\alpha(F),\varphi_\alpha(G)$ can only be elements of $\K[z]$. In the proof, we would like to construct a strong subalgebra of $R$ containing $F,G$ and the vector space $V$, where the generators of $V$ are preserved. Therefore we need to assume that $V$ is sufficiently strong to begin with, so that we can apply \cref{corollary: robust strong algebra}. Recall that for any $a\in \N$ the translation function $t_a:\N^d\rightarrow \N^d$ is given by $t_a(\delta_1,\cdots,\delta_d)=(\delta_1+a,\cdots,\delta_d+a)$ and $h_B$ is the function defined in \cref{proposition: constructing AH algebras}. 

\begin{proposition}\label{proposition: gcd after projection}
Let $G:\N^d\rightarrow \N^d$ be a function such that $G_i(\delta)\geq h_{B,i}\circ t_2(\delta)$ for all $\delta\in \N^d$. 
Let $V \subset R_{\leq d}$ be a $G$-lifted strong vector space and $\varphi_\alpha :R[z]\rightarrow R[z]/I_\alpha$ be a general quotient as defined in \cref{definition: quotients}. 
Let $F, G \in R_{\leq d}$ be such that they have no common factor.
Then,
\begin{enumerate}
    \item $\gcd(\varphi_\alpha(F), \varphi_\alpha(G)) \in \K[z]$ 
    \item If $F,G$ are homogeneous, then 
    $\gcd(\varphi_\alpha(F), \varphi_\alpha(G)) = z^k $ for some $k \in \N$. 
    In particular, we have $\gcd(\varphi_\alpha(zF), \varphi_\alpha(zG)) = z^{k+1} $ for some $k \in \N$. Furthermore, if $F,G\not \in K[V]\subset R$ then $\varphi_\alpha(F)$, $\varphi_\alpha(G)$ are linearly independent.
    \item If $F\in R$ is a square-free form, then 
    $\varphi_\alpha(F)$ does not have multiple factors other than $z^k$ for some $k\in \N$.
\end{enumerate}

\end{proposition}

\begin{proof}
   1.  Recall that $R=S/(U)$. 
   Let $\widetilde{F}_1,\cdots,\widetilde{F}_n$ be homogeneous lifts of a homogeneous basis of $V$. 
   Let $\widetilde{F},\widetilde{G}\in S$ be homogeneous lifts of $F,G$. 
   Note that $U+\Kspan{\widetilde{F}_1,\cdots,\widetilde{F}_n}$ is $G$-strong in $S$. 
   By applying \cref{corollary: robust strong algebra} to the vector space $U+\Kspan{\widetilde{F}_1,\cdots,\widetilde{F}_n,\widetilde{F},\widetilde{G}}$, we may construct a $B$-strong vector space $\wt{W}\subset S$ such that $\wt{W}$ contains $U+\Kspan{\widetilde{F}_1,\cdots,\widetilde{F}_n}$ and $\wt{F},\wt{G}\in \K[\wt{W}]$. 
   Let $W\subset R$ be the image of $\wt{W}$ in $R$. 
   Then $V\subset W$ and $F,G\in \K[W]$. 
   Note that any homogeneous basis of $\wt{W}$ is an $\cR_\eta$-sequence in $S$, by \cref{corollary: strong sequence R_eta}. 
   Hence we know that $W$ is generated by an $\cR_\eta$-sequence in $R$, by \cref{proposition: strong sequence CM UFD}. 
   Let $A=\K[W]\subset R$. Let $G_i=F_i-\alpha_iz^{\deg{F_i}}$ and $J_\alpha=(G_1,\cdots,G_n)\subset A[z]$. 
   Let $\vartheta_\alpha:A[z]\rightarrow A[z]/J_\alpha$. 
   Note that $F,G$ do not have a common factor in $A$. 
   Then, by \cref{lemma: general projection basic}, we know that for a general $\alpha\in \K^n$, we have $\gcd(\vartheta_\alpha(F), \vartheta_\alpha(G)) \in \K[z]$ in $A[z]/J_\alpha$. 
   Since $A[z]$ is generated by an $\cR_\eta$-sequence in $R[z]$, defining $I_\alpha := J_\alpha R[z]$, we note that $I_\alpha\cap A=J_\alpha$, by \cref{proposition: radicals and prime sequences}. 
   Then we have $A[z]/J_\alpha\rightarrow R[z]/I_\alpha$ is injective and that $A[z]/J_\alpha$ can be identified with the image of $A[z]\rightarrow R[z]/I_\alpha$. 
   Hence we have, $\vartheta_\alpha(F)=\varphi_\alpha(F)$ and $\vartheta_\alpha(G)=\varphi_\alpha(G)$. We will show that
    $\gcd(\vartheta_\alpha(F), \vartheta_\alpha(G))$ in $A[z]/J_\alpha$ is equal to $\gcd(\varphi_\alpha(F), \varphi_\alpha(G))$ in $R[z]/I_\alpha$.
    
    Let $\wt{G}_i=\wt{F}_i-\alpha_iz^{\deg(F_i)}$. Let $\cF$ be a homogeneous basis of $\wt{W}$ in $S$ that contains $\wt{F}_1,\cdots,\wt{F}_n$. Let $\cF'$ be the sequence of forms obtained by replacing $\wt{F}_i$ with $\wt{G}_i$ for all $i$. Note that, by \cref{proposition: strong with z}, we have that $\Kspan{\cF'}$ is $B$-strong in $S[z]$ and hence an $\cR_\eta$-sequence in $S[z]$, by \cref{corollary: strong sequence R_eta}. Note that the image of $\Kspan{\cF'}$ in $R[z]/I_\alpha$, generates the image of $A[z]$.
    Therefore, we see that $A[z]/J_\alpha$ is identified with a subalgebra generated by an $\cR_\eta$-sequence in $R[z]/I_\alpha$. Therefore, \cref{lem:factor-ufd-algebra} implies that any common factor of $\varphi_\alpha(F)$ and $\varphi_\alpha(G)$ in $R[z]/I_\alpha$ must be a in $A[z]/J_\alpha$. Therefore, $\gcd(\varphi_\alpha(F), \varphi_\alpha(G))=\gcd(\vartheta_\alpha(F), \vartheta_\alpha(G))\in \K[z]$ in $R[z]/I_\alpha$, for a general $\alpha\in \K^n$.

    2. By \cref{proposition: homogeneous factors}, we see that if $h\in \K[z]\subset R[z]/I_\alpha$ and $h$ is a factor of $F$, then $h$ is homogeneous. Thus we must have $h=z^r$ for some $r$. Similarly for $G$. Therefore,  for a general $\alpha\in \K^n$, we have that  $\gcd(\varphi_\alpha(F), \varphi_\alpha(G))=z^k$ for some $k\in \N$, by part (1). If $F,G\not \in \K[V]$, then by \cref{proposition: general quotient general}, we know that $\varphi_\alpha(F),\varphi_\alpha(G)\not \in \K[z]$ for a general $\alpha\in \K^n$. If $\varphi_\alpha(F),\varphi_\alpha(G)$ were linearly dependent, then we must have $\varphi_\alpha(F),\varphi_\alpha(G)\in \K[z]$ which is a contradiction.
    
    3. Since $F\in A=\K[W]$ as above, and $F$ is square-free in $R$, we know that $F$ is square-free in $A$. 
    Let $\varphi_\alpha(F)=z^kf_1^{e_1}\cdots f_r^{e_r}$ be an irreducible factorization in $R[z]/I_\alpha$. 
    Since $\varphi_\alpha(F)\in A[z]/J_\alpha$, we know that $f_1,\cdots,f_r\in A[z]/J_\alpha$ and  $\varphi_\alpha(F)=z^kf_1^{e_1}\cdots f_r^{e_r}$ is an irreducible factorization in $A[z]/J_\alpha$, by \cref{lem:factor-ufd-algebra}. 
    By applying \cref{lemma: general projection basic} to $A$, we know that the multiple factors of $\varphi_\alpha(F)$ in $A[z]/J_\alpha$ must be in $\K[z]$, for a general $\alpha\in \K^n$. 
    Since $\varphi_\alpha(F)$ is a homogeneous element of $R[z]/I_\alpha$, we conclude that any multiple factor of $\varphi_\alpha(F)$ in $A[z]/J_\alpha$ must be a homogeneous element of $R[z]/I_\alpha$, and hence of the form $z^\ell$ for some $\ell\in \N$. 
    Since $z$ does not divide $f_1,\cdots,f_r$ in $A[z]/J_\alpha$, we conclude that $e_i=1$ for all $i\in [r]$. 
    Hence, for a general $\alpha\in \K^n$, we see that the only multiple factor $\varphi_\alpha(F)$ is $z^k$ for some $k\in \N$.
\end{proof}

\begin{remark}
    Given two elements $F,G\in R$, the results above show that under a general quotient $\varphi_\alpha$, we have that  $\varphi_\alpha(F), \varphi_\alpha(G)$ satisfy certain properties such as non-vanishing and $\gcd(\varphi_\alpha(F), \varphi_\alpha(G))\in\K[z]$ etc. In other words, given any such property $\mathcal{P}$, there exists a non-empty open subset $\mathcal{X}\subset \K^n$ (depending on $F,G$) such that $\mathcal{P}$ holds for all $\alpha\in \mathcal{X}$. Let $\cF=\{F_1,\cdots, F_m\}\subset R$ be a finite set and  $\mathcal{P}$ a property as above. We may take the finite intersection of the non-empty open subsets $\cX_{ij}$ corresponding to pairs of forms $F_i,F_j$ to assume that for a general quotient $\varphi_\alpha$, all elements of $\cF$ satisfy the property $\mathcal{P}$ such as non-vanishing and $\gcd(\varphi_\alpha(F_i), \varphi_\alpha(F_j))\in\K[z]$ etc. Therefore we may assume that the results proved in this section apply to elements of any finite set $\cF$ and
 general quotients $\varphi_\alpha$. 
\end{remark}

\subsection{Lifting from general quotients}
As part of our inductive strategy, we would like to show that if a general quotient $\varphi_\alpha(\cF)$ of a SG-configuration $\cF$ has small vector space dimension then $\cF$ must also have small dimension. In this section we prove that we can indeed lift the bounds on dimensions under a general quotients. We prove our lifting result for strong quotients of polynomial rings and generalize previous results on quadratics \cite[Claim 2.26]{PS22} and for forms in a polynomial ring \cite[Proposition 2.11]{OS22}.

Let $V\subset R$ be a $B$-strong vector space. Let $F_1,\cdots,F_n$ be a homogeneous basis of $V$. Then the graded subalgebra $\K[V]$ is isomorphic to the polynomial ring $\K[y_1,\cdots,y_n]$ via the isomorphism $y_i\mapsto F_i$. Note that if $\K[y_1,\cdots,y_n]$ is the standard polynomial ring with $\deg(y_i)=1$, then the isomorphism above is not a homomorphism of graded rings. However, if we consider $\K[y_1,\cdots,y_n]$ with the grading $\deg(y_i)=\deg(F_i)$, then we have an isomorphism of graded $\K$-algebras. Therefore, in the results below we use the grading described above to compute the dimensions of the graded factors and obtain the dimension bounds.

\begin{lemma}\label{lemma: lifting from quotient basic}
Fix positive integers $d,d_1,\cdots,d_N\in \N$. 
Let $S=\K[x_1,\cdots,x_N]$ be the polynomial ring with the grading $\deg(x_i)=d_i$ and $z$ be a new variable with $\deg(z)=1$. 
For $\alpha\in \K^n$, let $I_\alpha=(x_1-\alpha_1 z^{d_1},\cdots,x_n-\alpha_nz^{d_n})$ in $S[z]$.
Let $\varphi_\alpha:S[z]\rightarrow S[z]/I_\alpha$ be the quotient ring homomorphism. 
Let $\cF\subset S_{\leq d}$ be a finite set of homogeneous polynomials of positive degree. 
Suppose there exists $D\in \N$ such that $\dim(\Kspan{\varphi_\alpha(\cF)})\leq D$ for general $\alpha\in \K^n$. 
Then we have $$\dim(\Kspan{\cF})\leq d^2(1+d)^{2n+2}D.$$
\end{lemma}

\begin{proof}
Note that it is enough prove that $\dim(\Kspan{\cF\cap S_\ell}) \leq \ell(1+\ell)^{2n+2}D$ for all $\ell\in [d]$. 
Therefore, without loss of generality, we may assume that $\cF\subset S_d$.  Let $\cF=\{F_1,\cdots,F_t\}\subset S_d$. 
We will show that $\dim(\Kspan{\cF})\leq d(1+d)^{2n+2}D$. 
Consider the graded $\K$-subalgebras given by $R=\K[x_1,\cdots,x_n]$ and $A=\K[x_{n+1},\cdots, x_N]$. 
For any $k\geq 1$, let $r_k$ denote the number of distinct monomials of degree $k$ in $R$. 
Note that $r_k\leq (1+k)^n\leq (1+d)^n$ for all $k\in [d]$. 
Let $M^{(k)}_1,\cdots,M^{(k)}_{r_k}$ denote these monomials that span $R_k$. 
    
    Consider $F_i$ as a polynomial in $A[x_1,\cdots,x_n]$. 
    Let $F_i^{(0)}\in A$ be the degree zero part of $F_i$ as a polynomial in $A[x_1,\cdots,x_n]$. For $k\geq 1$, let $F^{(k)}_{ij}\in A$ denote the coefficient of the degree $k$ monomial $M^{(k)}_j$ in $F_i$. Note that if  $F^{(k)}_{ij}\neq 0$, then it is a homogeneous polynomial of degree $d-k$ in $A$. Let $V_d=\Kspan{F_i^{(0)}|i\in [t]}$ and for $k\in [d-1]$ we let $V_{d-k}=\Kspan{F^{(k)}_{ij}|i\in[t], j\in [r_k]}$. Let
$V=V_1+\cdots+V_{d-1}+V_d$. Then we have $\cF\subset \K[V,x_1,\cdots,x_n]$. We will show that $\dim(V_{d-k})\leq r_kD$ for all $k\in[d-1]$ and $\dim(V_d)\leq D$.

    Fix $k\in [d-1]$. For simplicity, let $F_{ij}$ denote the forms $F^{(k)}_{ij}$ and $M_j$ denote the monomials $M^{(k)}_j$ as defined above.
    For $\alpha\in \K^t$, we have $\varphi_{\alpha}(F_i)\in A[z]$ and the coefficient of $z^k$ in $\varphi_{\alpha}(F_i)$ is given by $M_1(\alpha)F_{i1}+\cdots+M_{r_k}(\alpha)F_{ir_k}$. For a general $\alpha\in \K^t$, we know that $\dim\Kspan{\varphi(\cF)}\leq D$. Therefore the dimension of the span of the coefficients of $z^k$ in all the $\varphi(F_i)$ is at most $D$, i.e. we have $\dim\Kspan{M_1(\alpha)F_{i1}+\cdots+M_{r_k}(\alpha)F_{ir_k}|i\in[t]}\leq D$. By \cref{lemma: general vandermonde invertible}, we may choose $r_k$ number of general  points $\alpha^{(1)},\cdots,\alpha^{(r_k)}\in \K^{r_k}$, such that the matrix $\cM$ defined by $\cM_{ij}= M_i(\alpha^{(j)})$ is invertible. Then we have $\dim\Kspan{F_{ij}|i\in[t],j\in[r_k]}\leq r_kD$. Furthermore, since $F_i^{(0)}\in A$, we know that $\varphi_\alpha(F_i^{(0)})=F_i^{(0)}$ and $\dim(V_d)\leq \dim(\Kspan{\varphi_\alpha(F_i^{(0)})|i\in [t]})\leq D$.

Let $F_i\in \cF$ and $k\in [d-1]$. 
By the above argument, we know that the coefficients $\{F^{(k)}_{ij}\}$ of the degree $k$ monomials $M^{(k)}_1,\cdots,M^{(k)}_{r_k}$ in $F_i$'s span a vector space of dimension at most $r_{k}D$. 
Hence, the span of $\{M_\ell F_{ij}|i\in[m],j,\ell\in[r_k]\}$ is of dimension at most $r_{k}^2D$. For $k=d$, the coefficients of the degree $d$ monomials $M^{(d)}_1,\cdots,M^{(d)}_{r_d}$ in $F_i$ are constants. For $k=0$, we know that span of the parts that are independent of $x_1,\cdots,x_n$ has dimension given by $\dim(V_d)\leq D$. Therefore, $\dim\Kspan{\cF}\leq D+\sum_{k=1}^{d-1}r_k^2D+r_d\leq d(1+d)^{2n+2}D$, since $r_k\leq (1+d)^n$ for all $k\in [d]$.
\end{proof}

\begin{proposition}\label{proposition: lifting general quotient}
Let $d,e\in \N$ such that $1\leq d\leq e$. Let $U\subset S_{\leq e}$ be a graded vector space generated by forms $H_1,\cdots,H_t$. Let $R=S/(U)$. Let $V \subset R_{\leq e}$ be a $B$-lifted strong vector space with basis $F_1,\cdots,F_n\in R$. Let $\varphi_\alpha :R[z]\rightarrow R[z]/I_\alpha$ be a graded quotient as defined in \cref{definition: quotients}. Let $\cF\subset R_{\leq d}$ be a finite set of homogeneous elements.
Suppose that there exists $D\in \N$ such that $\dim \Kspan{\varphi_\alpha(\cF)}\leq D$ for a general $\alpha\in \K^n$. 
Then  
$$\dim\Kspan{\cF}\leq d^2(1+d)^{2n+2}D\cdot\Pi_{i=1}^t\deg(H_i)\cdot\Pi_{j=1}^n\deg(F_j).$$
\end{proposition}

\begin{proof}
Note that $H_1,\cdots, H_t$ is an $\cR_\eta$-sequence in $S$ and hence $R$ is a Cohen-Macaulay UFD, by \cref{remark: lifted strong implies UFD}. 
Also, $R[z]$ and $R[z]/I_\alpha$ are quotients of $S[z]$ by $\cR_\eta$-sequences and hence by Cohen-Macaulay UFDs by \cref{proposition: graded quotient is UFD}. 
By \cref{proposition: 2B implies B lifted strong}, we know that  $F_1,\cdots,F_n$ is an $\cR_\eta$-sequence in $R$. 
Thus, $F_1,\cdots,F_n,z$ is an $\cR_\eta$-sequence in $R[z]$. 
We may extend $F_1,\cdots,F_n$ to a homogeneous system of parameters $F_1,\cdots,F_{N-t}$ of $R$ such that $\deg(F_i)=1$ for $n+1\leq i \leq N-t$.
Let $A=\K[F_1,\cdots,F_{N-t}]$, $G_i=F_i-\alpha_iz^{\deg(F_i)}\in A[z]$ for $i\in [n]$ and $J_\alpha=(G_1,\cdots,G_n)\subset A[z]$. 
Let $\vartheta_\alpha: A[z]\rightarrow A[z]/J_\alpha$ be the quotient map. 
We have that $A\subset R$ is a module-finite free extension by \cref{theorem: CM hsop}. 
Furthermore, the rank of $R$ as a free $A$-module is $r:=\mathrm{rank}(R)=\Pi_{i=1}^t\deg(H_i)\cdot\Pi_{j=1}^n\deg(F_j)$  by \cref{proposition: CM hsop and degree}. By tensoring, we obtain that $A[z]\subset R[z]$ and $A[z]/J_\alpha\rightarrow R[z]/I_\alpha$ are module-finite free extensions of rank $r$. Therefore, we have 
$R[z]/I_\alpha\xrightarrow[]{\sim}(A[z]/J_\alpha)^{\oplus r}$. Let $F_i\in \cF \subset R[z]$ have components $(f_{i1},\cdots,f_{ir})\in (A[z])^{\oplus r}$ under this isomorphism. Note that we must have $f_{ij}$ are homogeneous and  $\deg(f_{ij})\leq \deg(F_i)$ in $A[z]$, by \cref{proposition: homogeneous compoments}. We have that $\varphi_\alpha: R[z]\rightarrow R[z]/I_\alpha$ commutes with the quotient map $\vartheta_\alpha^{\oplus r}:(A[z])^{\oplus r}\rightarrow (A[z]/J_\alpha)^{\oplus r}$ via the isomorphisms above, i.e. the following diagram commutes.

\[\begin{tikzcd}
 R[z] \arrow[r, "\varphi_\alpha"] \arrow[d, "\simeq"] & R[z]/I_\alpha \arrow[d,"\simeq"]\\
 (A[z])^{\oplus r} \arrow[r, "(\vartheta_\alpha)^{\oplus r}"] & (A[z]/J_\alpha)^{\oplus r} \\
\end{tikzcd}\]

Now we know that $\dim \Kspan{\varphi_\alpha(\cF)}\leq D$ in $R[z]/I_\alpha$. Therefore, for any $j\in [r]$,  we have $\dim \Kspan{\vartheta_\alpha(f_{ij})|i\in [m]}\leq D$ in the $j$-th component of $(A[z]/J_\alpha)^{\oplus r}$. Note that $A=\K[F_1,\cdots,F_n]\subset R$ is a graded polynomial ring. 
By applying \cref{lemma: lifting from quotient basic}, we conclude that $\dim\Kspan{f_{ij}|i\in [m]}\leq d^2(1+d)^{2n+2}D$. Therefore, we have \[\dim\Kspan{\cF}\leq d^2(1+d)^{2n+2}Dr=d^2(1+d)^{2n+2}D\cdot\Pi_{i=1}^t\deg(H_i)\cdot\Pi_{j=1}^n\deg(F_j).\]
\end{proof}

\subsection{General Quotients and Sylvester-Gallai Configurations}\label{subsecttion: SG and quotients}

We show that radical SG-configurations are preserved under general quotients.
Moreover, if a $\rsg{d}$ configuration is contained in an ideal generated by a strong vector space, then we can reduce the degree of the configuration.

\begin{definition}[Reduction of a form]
Let $R$ be a finitely generated $\K$-algebra which is a UFD. 
For any $F\in R$ we define a reduction of $F$ to be a generator of the principal ideal $\rad(F)$ in $R$. 
In other words if $F=f_1^{e_1}\cdots f_r^{e_r}$ is an irreducible factorization of $F$, then $f_1\cdots f_r$ is a reduction of $F$. 
Let $\cF\subset R$ be any subset consisting of elements of positive degree. 
We define $\cF_{\mathrm{red}}$ to be a set of pairwise non-associate elements in $R$ such that for any $F\in \cF$ we have that $\cF_{\mathrm{red}}$ contains a reduction of $F$.  
\end{definition}

\begin{remark}
(a) By abuse of notation we will denote any choice of a reduction of $F$ as $F_\red$. 
Note that $F_\red$ is well-defined up to multiplication by elements of $\K$.

(b) Let $\cF\subset R$ be a finite set of homogeneous elements. 
Suppose there is a homogeneous element $z\in R_1$ such that for all $F,G\in \cF$ we have $\gcd(F,G)=z^k$ for some $k\geq 1$. Then $|\cF_\red |=|\cF|-|\cF\cap \K[z]|+1$.

\end{remark}

For the rest of  the section, we fix positive integers $d,e,\eta\in \N$ with $1\leq d\leq e$ and $\eta\geq 3$. 
Let $B:\N^e\rightarrow \N^e$ denote an ascending function such that $B_i(\delta)\geq A(\eta,i)+3(\sum_i\delta_i-1)$ for all $i\in [e]$, where $A(\eta,i)$ is the function provided by \cref{theorem: Ananyan-Hochster}. 
For any $a\in \N$, let $t_a:\N^e\rightarrow \N^e$ be the translation function defined as $t_a(\delta)=(\delta_1+a,\cdots,\delta_e+a)$. 
Let $C_B:\N^e\rightarrow \N^e$ and $h_B:\N^e\rightarrow \N^e$ be the functions provided by \cref{proposition: constructing AH algebras}. 

Throughout this section we will use the following notations. Let $S=\K[x_1,\cdots,x_N]$ and $y$ a new variable.
Let $U\subset S_{\leq e}$ be a graded vector space with dimension sequence $\delta\in \N^e$, generated by forms $H_1,\cdots,H_t$ and let $R=S/(U)$. Let $z\in R_1$. Let $V\subset R_{\leq e}$ be a graded vector space of dimension $n+1$ containing $z$. We choose a basis $z,G_1,\cdots, G_n\in \R_{\leq e}$ and for $\alpha\in \K^{n+1}$, we let $I_\alpha=(V_\alpha)=(z-\alpha_1y,G_n-\alpha_1y^{\deg(G_1)},\cdots,G_n-\alpha_ny^{\deg(G_n)})$.

\begin{proposition}\label{proposition: general quotient preserves SG}
With the notations as above, suppose that $V \subset R_{\leq e}$ is $(h_B\circ t_2)$-lifted strong vector space. 
Let $\cF \subset R$ be a $\rsgufd{d}{z}{R}$ configuration. 
Let $\varphi_\alpha :R[y]\rightarrow R[y]/I_\alpha$ be the graded quotient. Then we have

 \begin{enumerate}
     \item the set $\varphi_{ \alpha}(\cF)_{\mathrm{red}} \subset R[y]/I_{ \alpha}$ is a $\rsgufd{d}{y}{R[y]/I_{\alpha}}$ configuration for general $\varphi_\alpha$.
     \item Suppose that for all $zP\in \cF\cap R_{d+1}$ we have $P\in (V)$. Then we have that $\varphi_{ \alpha}(\cF)_{\mathrm{red}}$ is a $\rsgufd{d-1}{y}{R[y]/I_{ \alpha}}$ configuration for general $\varphi_\alpha$.
     \item  Suppose there exists $D\in \N$, such that we have a global bound $\dim \Kspan{\varphi_{ \alpha}(\cF)_{\mathrm{red}}} \leq D$, for general $\varphi_\alpha$. Then 
$$\dim\Kspan{\cF}\leq (d+1)^3(d+2)^{2n+2}(D+1)\cdot\Pi_{i=1}^t\deg(H_i)\cdot\Pi_{j=1}^n\deg(G_j).$$
In particular, we have 
$$\dim\Kspan{\cF}\leq (d+3)^{3n+6}e^tD=(d+3)^{3\dim(V)+6}e^{\dim(U)}D.$$
\end{enumerate}
\end{proposition}

\begin{proof}
   1.  Note that $\varphi_\alpha
    (\cF)_\red$ is a finite set of square-free homogeneous elements of degree at most $d+1$ in $R[y]/I_\alpha$. Since $z\in V$,  we have that $\varphi_\alpha(z)=\beta y$ for some non-zero $\beta\in \K$. We know that $V$ is also $B$-lifted strong, since $h_B\circ t_2\geq B$. Since $\cF$ is a finite set, we may assume that for a general $\alpha\in \K^n$, we have $\varphi_\alpha(F)\neq 0$ for any $F\in \cF$, by \cref{proposition: general quotient general}. We may assume that $y\in (\varphi_\alpha(\cF))_\red$, by changing the representative of $\varphi_\alpha(z)$ in the reduction. Now we prove that the defining conditions for a $\rsgufd{d}{y}{R[y]/I_\alpha}$ configuration hold.
    
    (a) 
    For any two elements $F,G\in \cF$ we know that $\gcd(F,G)=z$ and $\varphi_\alpha(z)=\beta y$. Therefore, by \cref{proposition: gcd after projection} we have $\gcd(\varphi_\alpha(F), \varphi_\alpha(G)) = y^k $ for some $k \geq 1$, for a general $\alpha\in \K^n$. In particular, for any $F',G'\in \varphi_\alpha(\cF)_\red$ we have that $\gcd(F',G') = y$.

    (b) For any two distinct $F',G'\in \varphi_\alpha(\cF)_\red$, there exists distinct $F,G\in \cF$ such that $F'=(\varphi_\alpha(F))_\red$ and $G'=(\varphi_\alpha(G))_\red$. Now, there exists $H\in \rad(F,G)\cap \cF$ such that $H\neq F,G$, as $\cF$ is a $\rsgufd{d}{z}{R}$ configuration. Then, $\varphi_\alpha(H)\in \rad(\varphi_\alpha(F),\varphi_\alpha(G))$, since $\varphi_\alpha$ is a ring homomorphism. Now we have $H'=(\varphi_\alpha(H))_\red\in \rad(F',G')\cap \varphi_\alpha(\cF)_\red$. If $F',G'\neq y$, then $F,G\not \in \K[V]$. If $H\in \K[V]$, then we have $H'=y\neq F',G'$.  By \cref{proposition: gcd after projection}, we know that the only common factor of  $\varphi_\alpha(F) ,\varphi_\alpha(H)$ is $y$. If $H\not \in \K[V]$, then we have $H'\neq y$, by \cref{proposition: general quotient general}. Therefore we conclude that  $H'=(\varphi_\alpha(H))_\red\neq (\varphi_\alpha(F))_\red=F'$. Similarly, $H'\neq G'$. Therefore, we have that $\varphi_\alpha(\cF)_\red$ is a $\rsgufd{d}{y}{R[y]/I_\alpha}$ configuration.

    2. Suppose that for all $zP\in \cF\cap R_{d+1}$, we have $P\in (V)=(z,G_2,\cdots,G_n)$. 
    Note that $\varphi_\alpha(G_i)=\alpha_iy^{\deg(G_i)}$ in $R[y]/I_\alpha$. 
    Since $\cF\subset (z)$ by definition, we see that $\varphi_\alpha(\cF)\subset (y^2)$ as $\varphi_\alpha(z)=\beta y$ and $\varphi_\alpha(G_i)\in (y)$ for all $i$. Therefore, we have $\deg(\varphi_\alpha(F)_\red)\leq d$ for any $F\in \cF$. Since $\varphi_\alpha(\cF)_\red$ is a $\rsgufd{d}{y}{R[y]/I_\alpha}$ configuration, by part (1). We conclude that since $\varphi_\alpha(\cF)_\red$ is a $\rsgufd{d-1}{y}{R[y]/I_\alpha}$ configuration.

    3. Recall that $\cF\subset R_{\leq d+1}$ and we have $\cF\subset (z)$. By definition of a $\rsgufd{d}{z}{R}$ configuration, we know that $\cF$ consists of square-free homogeneous elements. For a general $\alpha\in \K^n$, we know that for all $F\in \cF$, the only multiple factor of $\varphi_\alpha(F)$ is of the form $y^r$, by \cref{proposition: gcd after projection}. Note that $\varphi_\alpha(\cF)\cap \K[y]\subset \Kspan{y,y^2,\cdots,y^{d+1}}$. On the other hand, for any $F\in \cF$ such that $\varphi_\alpha(F)\not \in \K[y]$, we have $\varphi_\alpha(F)=y^rP$ where $P$ is square free and $y\not\mid P$. Thus \[\Kspan{\varphi_\alpha(F)\setminus \K[y]}\subset \Kspan{y^rP\mid r\in[d], y\not\mid P, yP\in (\varphi_\alpha(F))_\red}.\]
    Therefore, for a general $\alpha\in \K^n$, we have $\dim(\Kspan{\varphi_\alpha(\cF)})\leq d+1+dD< (d+1)(D+1)$, as $\dim(\Kspan{\varphi_\alpha(\cF)_\red})\leq D$. Hence, by applying \cref{proposition: lifting general quotient} to $\cF\subset R_{\leq d+1}$, we have 
    $$\dim\Kspan{\cF}\leq (d+1)^3(d+2)^{2n+2}(D+1)\cdot\Pi_{i=1}^t\deg(H_i)\cdot\Pi_{j=1}^n\deg(G_j).$$
 \end{proof}

In the proof of \cref{theorem: radical SG theorem over UFD}, we will apply general quotients successively to eventually reduce the degree of the SG-configuration. In \cref{corollary: lifting from composition of quotients}, we generalize \cref{proposition: general quotient preserves SG} to compositions of general quotients and show that we can indeed lift bounds on dimensions of SG-configurations.\\

\emph{Composition of general quotients.} Recall that $1\leq d\leq e\in \N$ and $R=S/(U)$, where $U\subset R_{\leq e}$ is a graded vector space.
Let $z_1,\cdots,z_\ell$ be new variables and $n_0,\cdots, n_{\ell-1}$ be positive integers. Suppose we have a sequence of graded quotients and Sylvester-Gallai configurations defined inductively as follows. 

\begin{enumerate}
    \item Let $R^{(0)}=R$. Let $\cF^{(0)}$ be a $\rsgufd{d}{z_0}{R}$ configuration for some $z_0\in R$.
    \item For some $0\leq k \leq \ell-1$, suppose that we have defined $R^{(k)}$ and a $\rsgufd{d}{z_k}{R^{(k)}}$ configuration $\cF^{(k)}$. We define $\cF^{(k+1)}$ iteratively as follows:
    \begin{enumerate}
        \item Suppose that $W^{(k)}\subset R^{(k)}_{\leq e}$ is a $h_B\circ t_2$-lifted strong graded vector space such that $z_k\in W^{(k)}$ and $\dim(W^{(k)})\leq n_k$.
        \item We define $R^{(k+1)}=R^{(k)}[z_{k+1}]/(W^{(k)}_{\alpha^{(k)}})$ and $\varphi_{\alpha^{(k)}}:R^{(k)}[z_{k+1}]\rightarrow R^{(k+1)}$ as the graded quotient defined by a general $\alpha^{(k)}\in \K^{\dim(W^{(k)})}$, for which the conclusions of \cref{proposition: general quotient preserves SG} hold.
        \item We define $\cF^{(k+1)}=(\varphi_{\alpha^{(k)}})_\red \subset R^{(k+1)}$, which is a $\rsgufd{d}{z_{k+1}}{R^{(k+1)}}$ configuration by \cref{proposition: general quotient preserves SG}.
    \end{enumerate}
\end{enumerate} 

Note that we assume the existence of a vector space $W^{(k)}\subset R^{(k)}$ which is $h_B\circ t_2$-lifted strong. In particular $U\subset S$ must be sufficiently strong to begin with. In the following result, we show that if we have a bound on the dimension of $\cF^{(\ell)}$ for compositions of general quotients as above, then we can lift it to a bound on the dimension of $\cF^{(0)}$.

\begin{corollary}\label{corollary: lifting from composition of quotients}

Let $1\leq \ell\in \N$. With the notations as above, for all $0\leq k \leq \ell-1$, we may write $$R^{(k+1)}=S[z_1,\cdots, z_{k+1}]/(U^{(k)})
,$$ where $U^{(k)}$ is a vector space such that $\dim(U^{(k)}) = \dim(U)+\dim(W^{(0)})+\cdots +\dim(U^{(k)})$. 
Furthermore, suppose there exists $D\in \N$ such that $\dim \Kspan{\varphi_{ \alpha}(\cF^{(\ell)})_{\mathrm{red}}} \leq D$ for all choices of general $\alpha^{(0)},\cdots, \alpha^{(\ell-1)}$ as above. 
Then,
    $$\dim\Kspan{\cF^{(0)}}\leq (d+3)^{6\ell+3(\sum_{k= 0}^{\ell-1}n_k)}\cdot e^{(\ell\dim(U)+\sum_{k=0}^{\ell-2}(\ell-k-1)n_k)}\cdot D.$$
\end{corollary}

\begin{proof}
   Given $\ba=(\alpha^{(0)},\cdots,\alpha^{(\ell-1)})$ and $W^{(0)}_{\alpha^{(0)}},\cdots, W^{(\ell-1)}_{\alpha^{(\ell -1)}}$ as above, we let $\wt{W}^{(k)}_{\alpha^{(k)}}\subset S[z_1,\cdots, z_{\ell-1}]$ be the vector space spanned by homogeneous lifts of a homogeneous basis of $W^{(k)}_{\alpha^{(k)}}\subset R^{(k)}$, for all $0\leq k \leq \ell-1$. We define  $U^{(k)}_{\ba}=U+\wt{W}^{(0)}_{\alpha^{(0)}}+\cdots + \wt{W}^{(k)}_{\alpha^{(k)}}$ in $S[z_1,\cdots, z_{k+1}]$. Note that we have $R^{(k+1)}=S[z_1,\cdots, z_{k+1}]/(U^{(k)}_{\ba})$ for all $0\leq k \leq \ell -1$. By \cref{proposition: strong with z}, we know that $\dim(W^{(k)}=\dim(W^{(k)}_{\alpha^{(k)}})$ for all $k$. Now we have $\dim(W^{(k)}_{\alpha^{(k)}})=\dim(\wt{W}^{(k)}_{\alpha^{(k)}})$ for all $k$. Hence we have \[\dim(U^{(k)}_{\ba})= \dim(U)+\dim(W^{(0)})+\cdots +\dim(U^{(k)})\leq \dim(U)+n_0+\cdots+n_k\] for all $k$. Therefore we may take $U^{(k)}=U^{(k)}_{\ba}$. 
   
   Now we prove the dimension bound by induction on $\ell$. If $\ell=1$, then by applying \cref{proposition: general quotient preserves SG} to the general quotient $\varphi_{\alpha^{(1)}}$, we obtain that 
   $$\dim\Kspan{\cF^{(0)}}\leq (d+3)^{6+3n_0}\cdot e^{\dim(U)}\cdot D.$$
   Therefore we may assume that $\ell \geq 2$ and that the statement holds for all $k\leq \ell-1$. By applying \cref{proposition: general quotient preserves SG} to $R^{(\ell-1)}=S[z_1,\cdots,z_{\ell-1}]/(U^{(\ell-2)}_{\ba})$ and the general quotient $\varphi_{\alpha^{(\ell-1)}}$ we obtain that $$\dim\Kspan{\cF^{(\ell-1)}}\leq D':= (d+3)^{6+3n_{\ell-1}}\cdot e^{\dim(U)+\sum_{k=0}^{\ell-2}n_k}\cdot D.$$
   Since the statement holds for $k=\ell-1$, we obtain by induction that 
    $$\dim\Kspan{\cF^{(0)}}\leq (d+3)^{6(\ell-1)+3(\sum_{k= 0}^{\ell-2}n_k)}\cdot e^{((\ell-1)\dim(U)+\sum_{k=0}^{\ell-3}(\ell-k-2)n_k)}\cdot D'.$$
    Therefore we have
     $$\dim\Kspan{\cF^{(0)}}\leq (d+3)^{6\ell+3(\sum_{k= 0}^{\ell-1}n_k)}\cdot e^{(\ell\dim(U)+\sum_{k=0}^{\ell-2}(\ell-k-1)n_k)}\cdot D.$$
\end{proof}

\section{Main Result}\label{sec:main}


In this section, we prove the radical Sylvester-Gallai theorem for forms of degree at most $d$. 
For the rest of this section, we will denote our $\rsgufd{d}{z}{R}$ configuration $\cF=\{z,zF_1,\cdots,zF_m\}$, where $m := |\cF|-1$ and use the convention from \cref{remark: grading of SG configs} to denote for any $0 \leq j \leq d$:
$$\cF_j := \{F \in R_j \mid zF \in \cF\} \text{ and } m_j := |\cF_j|.$$
Also, for any $j > 1$ we denote by $\cF_{< j} = \bigcup_{i=1}^{j-1} \cF_i$. Moreover, for an element $F \in \cF_d$, we define the following sets:
\begin{align*}
    \Fspan(F) &:= \{ G \in \cF_d \ \mid \ |\Kspan{F, G} \cap \cF_d | \geq 3 \} \\
    \Grad(F) &:= \{ P \in \cF_{< d} \ \mid \ z, P, F \text{ regular sequence in $R$, and } (F, P) = \rad(F, P)  \} \\
\end{align*}
Note that for $F,G\in \cF_d$, we have $|\Kspan{F, G} \cap \cF_d | \geq 3 $ iff $| \Kspan{zF, zG} \cap \cF | \geq 3 $.

\begin{proposition}\label{proposition: auxiliary sets} With the definitions as above, the following holds.

\begin{enumerate}
    \item Any two elements $P,Q\in \bigcup_{i=1}^d\cF_i$ do not have any common factors in $R$.
    \item Suppose $F\in \cF_d$ and $P\in \Grad(F)$. Then there exists $G\in \cF_d$ such that $G\in (F,P)$.
\end{enumerate}

\end{proposition}
\begin{proof}
    1. Recall that $\cF$ is a set of non-associate square-free homogeneous elements and $\gcd(F_i,F_j)=z$ for any $i\neq j$. Therefore, if $P,Q\in \bigcup_{i=1}^d\cF_i$, i.e. $zP,zQ\in \cF$, then $P,Q$ have no common factors.
    
    2. We have $zF,zP\in \cF$. By definition of $\rsgufd{d}{z}{R}$ configurations, there exists $zG\in \cF$ such that $zG\in \rad(zF,zP)$ and $zG\neq zF, zP$. As $P\in \Grad(F)$, we have that $z,P,F$ is a regular sequence. By \cref{lemma: radical regular sequence} we have that $G\in \rad(F,P)$. As $P\in \Grad(F)$, we have $G\in (F,P)=\rad(F,P)$. Since $\deg(P)<d=\deg(F)$ by assumption, we see that $\deg(G)=d$, as $G,P$ do not have any common factors by part (1).
\end{proof}

We define the set $\Frad(F)$ as follows.

\begin{align*}
\Frad(F) &:= \{ G \in \cF_d \ \mid \ \exists \ P \in \Grad(F) \st G \in (F, P) \}
\end{align*}

By \cref{proposition: auxiliary sets}, we know that if $\Grad(F)$ is non-empty, then $\Frad(F)$ is also non-empty.

\subsection{Preparatory lemmata}\label{subsection: combinatorial preparations}

We fix positive integers $d,e, \eta\in \N$ with $1\leq d\leq e$ and $\eta\geq 3$. 
We let $B:\N^e\rightarrow \N^e$ denote an ascending function such that $B_i(\delta)\geq A(\eta,i)+3(\sum_i\delta_i-1)$ for all $i\in [e]$. Here $A(\eta,i)$ is the function given by \cref{theorem: Ananyan-Hochster}.
For any $n\in \N$, let $t_n:\N^e\rightarrow \N^e$ be the translation by $n$ function defined as $t_n(\delta)=(\delta_1+n,\cdots,\delta_d+n)$. 
Let $C_B:\N^e\rightarrow \N^e$ and $h_B:\N^e\rightarrow \N^e$ be the functions defined in \cref{proposition: constructing AH algebras}.
We first prove the linear case of our main theorem below.

\begin{proposition}\label{proposition: linear case main theorem} 
Let $U \subset S_{\leq e}$ be a graded vector space. 
Suppose $U$ is $B\circ t_2$-strong.
Consider the graded unique factorization domain $R := S/(U)$. 
Let $z \in R_1$ and $\mathcal{F}=\{z, za_1,\cdots,za_m\} \subset R_{\leq 2}$ be a  $\rsgufd{1}{z}{R}$ configuration. 
Then $\dim(\Kspan{\cF})\leq 26$.
\end{proposition}

\begin{proof} We may assume that $m\geq 4$. Consider the $\K$-vector space $R_1$ and the set $\cG=\{z,a_1,\cdots,a_m\}$. We will show that $\cG $ is a $(1,1/3)$-linear Sylvester-Gallai configuration in $R_1$. Indeed, let $a_i,a_j\in \cG$ where $i\neq j$. Since $U$ is $B\circ t_2$-strong in $S$, we know that $(a_i,a_j)$ is a prime ideal in $R$ by \cref{proposition: properties of quotient strength}. If $z,a_i,a_j$ is not a regular sequence in $R$, then $z$ is a zero-divisor in $R/(a_i,a_j)$ and hence we have $z\in (a_i,a_j)$. As $R$ is a graded $\K$-algebra, we conclude that $z\in \Kspan{a_i,a_j}$. Now suppose that $z,a_i,a_j$ is a regular sequence in $R$. Then, there exists $za_k\in \rad(za_i,za_j)$ for some $k\neq i,j$. Hence, we have $a_k\in \rad(a_i,a_j)$ by \cref{lemma: radical regular sequence}. Since $(a_i,a_j)$ is prime in $R$ and $a_i,a_j,a_k\in R_1$, we must have that $a_k\in \Kspan{a_i,a_j}$. Note that $|\cG|=m+1$ and $m-1\geq \frac{2}{3}m$ as $m\geq 4$. Hence, for all $a_i$ there exist at least  $\frac{2}{3}m$ number of $a_j\in\cG$ such that $|\Kspan{a_i,a_j}\cap \cG|\geq 3$. Therefore, $\cG$ is a $(1,\frac{1}{3})$-linear Sylvester-Gallai configuration in $R_1$ by \cref{linear sg simple}. Therefore $\dim(\Kspan{\cG})\leq 25$ by \cref{prop:lincdeltasg}. Therefore, $\dim(\Kspan{za_1,\cdots,za_m})\leq 25$ and $\dim(\Kspan{\cF})\leq 26$.
\end{proof}

\begin{lemma}\label{lemma: Grad set small}
    Let $U\subset S_{\leq e}$ be a $B$-strong graded vector space in $S$. 
    Consider the graded UFD $R=S/(U)$. 
    Let $\cF$ be a $\rsgufd{d}{z}{R}$ configuration.
    For any $F \in \cF_d$, we have \[|\Grad(F)| \leq 2^{d^2} \cdot |\Frad(F)| \leq 2^{d^2} \cdot m_d.\]
\end{lemma}

\begin{proof}
    Since $\Frad(F)\subset\cF_d$, it is enough to prove the first inequality. We may assume that $\Grad(F)$ is non-empty, otherwise the inequality is trivially satisfied. Let $\Grad(F) = \{P_1, \ldots, P_k\}$. 
    Note that $\deg(P_i) < d = \deg(F)$ by definition. 
    If $(F,P_i)=(F,P_j)$ for some $i,j\in[k]$, then $P_i\in (F,P_j)$. 
    As $\deg(P_i)<d$, we must have $P_i\in (P_j)$. 
    Hence we must have $i=j$, since $P_i,P_j$ do not have any common factors for $i\neq j$, by \cref{proposition: auxiliary sets}. 
    Therefore, we have that $(F,P_i)\neq (F,P_j)$ for all $i\neq j$.
    
    By \cref{proposition: auxiliary sets}, we know that for every $P_i\in \Grad(F)$, there exists at least one $G\in \Frad(F)$ such that $G\in (F,P_i)$. Let us fix such a choice of $G_i\in \Frad(F)$ for each $P_i\in \Grad(F)$. If $|\Grad(F)| > 2^{d^2} \cdot |\Frad(F)| $, then by the pigeonhole principle, there exist at least $2^{d^2}$ number of $P_i$'s in $\Grad(F)$ such that the corresponding $G_i$ is the same element, say $G\in \Frad(F)$. Therefore, there exist at least $2^{d^2}$ number of $P_i$'s in $\Grad(F)$ such that $G\in (F,P_i)= \rad(F,P_i)$. By \cref{lemma: d^2 minimal primes}, there exist at least two distinct $P_i,P_j\in \Grad(F)$ such that $\rad(F,P_i)=\rad(F,P_j)$. Hence there exist $P_i,P_j\in \Grad(F)$ such that $(F,P_i)=(F,P_j)$ for some $i\neq j$. This is a contradiction as $(F,P_i)\neq (F,P_j)$ for all $i\neq j$.
\end{proof}

\begin{lemma}\label{lemma: algebra contains many then ideal contains Fd}
    Let $U\subset S_{\leq e}$ be a graded vector space and let  $R=S/(U)$. 
    Let $z\in R_1$ and $W\subset 
    R_{\leq e}$ be a graded vector space. 
    Suppose that $W$ is $h_{2B}\circ t_1$-lifted strong with respect to $U$ and we have $z\in W$. Let $\cF$ be a $\rsgufd{d}{z}{R}$ configuration.
    
    \begin{enumerate}
        \item Suppose $|\cF_d|=m_d>6d^3$. If $|\cF_d\cap \K[W]|\geq \frac{2m_d}{3}$, then $\cF_d\subset(W)$.
        \item If $|\cF\cap \K[W]|> (d+1)m_d+d^2(2d-1)$, then $\cF_d\subset (W)$. 
    \end{enumerate}
    
\end{lemma}

\begin{proof} Let $\cF \cap \K[W] =: \{z, zP_1, \ldots, zP_k \}$. By \cref{remark: lifted strong implies UFD}, we know that any homogeneous basis of $U$ is an $\cR_\eta$-sequence in $S$ and that $R$ is a graded UFD. Furthermore, by \cref{proposition: 2B implies B lifted strong}, we know that any homogeneous basis of $W$ is an $\cR_\eta$-sequence (and also a prime sequence) in $R$. Note that, by \cref{lem:factor-ufd-algebra}, we must have $P_i\in \K[W]$ for all $i\in [k]$. Suppose there is $F \in \cF_d \setminus (W)$. 
Let $V\subset R$ be the $B$-lifted strong graded vector space constructed by applying \cref{corollary: practical robustness } to $W$ and $F$. In particular, $W\subset V$ and $F\in \K[V]$. By \cref{proposition: 2B implies B lifted strong}, we know that any homogeneous basis of $V$  is an $\cR_\eta$-sequence (and also a prime sequence) in $R$. In particular, $W$ and $V$ have homogeneous bases containing $z$, that form a prime sequences in $R$. Note that $z,P_i,F$ is a regular sequence in $R$, by \cref{proposition: three forms regular sequence}. By applying \cref{corollary: best corollary discriminant} to $F$ and the algebras $\K[W]\subset \K[V]$, we have that $(F, P_i)$ is not radical for at most $d^2(2d-1)$ elements in $\{P_1,\cdots, P_k\}$. Now we prove (1) and (2) below.

    1. Suppose $\cF_d\cap \K[W]\geq \frac{2m_d}{3}$. Without loss of generality, we may assume that $P_1,\cdots,P_r\in \cF_d\cap \K[W]$ for some $r\geq \frac{2m_d}{3}$. Furthermore, we can say that $(F, P_i)$ is radical for all $ i \leq r-d^2(2d-1)$.
    Now, the SG condition implies that there is $zG_i \in \cF \setminus \{zP_i, zF\}$ such that $zG_i \in \rad(zP_i, zF)$. Let us fix such a choice of $G_i$ for each $P_i$ where $i\in [r]$. Since $z,P_i,F$ is a regular sequence in $R$, we have that $G_i\in \rad(F,P_i)$ by \cref{lemma: radical regular sequence}.
    As $(F, P_i)$ is radical for $1\leq i \leq r-d^2(2d-1)$, we have $G_i\in (F,P_i)$ for all such $i$.  Since $\deg(F) = d$, we must have $\deg(G_i)=d$, as $G_i,P_i$ do not have common factors by \cref{proposition: auxiliary sets}. In particular, $G_i\in \Kspan{F,P_i}\cap \cF_d$. Note that since $F\not \in (W)$ and $P_i\in \K[W]$, we must have that $G_i\in\cF_d\setminus (W)$. If $G_i=G_j$ for some $i\neq j\in [r]$, then we have $G_i=G_j\in (F,P_i)\cap (F,P_j)$. Since all of $G_i,G_j,F,P_i,P_j\in R_d$, we must have that $\Kspan{F,P_i}=\Kspan{F,P_j}$. This is contradiction since $P_i,P_j$ do not have common factors and $F\not \in \K[W]$, whereas $P_i,P_j\in \K[W]$. 
    Therefore, $G_i\neq G_j$ for $1\leq i\neq j\leq r-d^2(2d-1)$. Note that $G_i\in \cF_d\setminus (W)$ for all $i\in [r]$ and $|\cF_d\setminus (W)|\leq |\cF_d\setminus \K[W]|\leq \frac{m_d}{3}$. Now we have \[r-d^2(2d-1)>\frac{2m_d}{3}-2d^3>\frac{m_d}{3},\] as $m_d>6d^3$. Hence, by the pigeonhole principle, there exists $1\leq i\neq j\leq r-d^2(2d-1)$ such that $G_i=G_j$. This is a contradiction.  Therefore, there does not exist $F\in \cF_d\setminus (W)$ and we have $\cF_d\subset (W)$.

 2. Recall that $(\bigcup_{i=1}^d\cF_i)\cap \K[W]=\{P_1,\cdots,P_k\}$ and $k> (d+1)m_d+d^2(2d-1)$. By the argument in the first paragraph, without loss of generality, we can say that $(F, P_i)$ is radical for all $ i \leq k-d^2(2d-1)$.
   Now, the SG condition implies that there is $zG_i \in \cF \setminus \{zP_i, zF\}$ such that $zG_i \in \rad(zP_i, zF)$. Since $z,P_i,F$ is a regular sequence in $R$, we have that $G_i\in \rad(F,P_i)$ by \cref{lemma: radical regular sequence}.
    As $(F, P_i)$ is radical for $1\leq i \leq k-d^2(2d-1)$, we have $G_i\in (F,P_i)$ for all such $i$.

    Since $\deg(F) = d$, we must have $\deg(G_i)=d$, as $G_i,P_i$ do not have common factors by \cref{proposition: auxiliary sets}. In other words, $G_i \in \cF_d \setminus \{F,P_i\}$ for $1\leq i \leq k-d^2(2d-1)$. Note that $k-d^2(2d-1)>(d+1)m_d$ and $|\cF_d|=m_d$ by assumption. Therefore, by the pigeonhole principle, there exists $G\in \cF_d$ such that $G\in \rad(F,P_i)$ for at least $(d+1)$ number of $P_i$'s. Without loss of generality, suppose that $G\in \rad(F,P_i)$ for $1\leq i \leq (d+1)$. Since $G, F\in \cF_d$ and $\rad(F,P_i)=(F,P_i)$, we may write $G=\alpha_i F+ P_i A_i$ for some $\alpha_i\in \K$ and $A_i\in R_{d-\deg(P_i)}$. If $\alpha_i\neq \alpha_j$ for some $i\neq j\leq (d+1)$, then we have that $(\alpha_i-\alpha_j)F=P_jA_j-P_iA_i\in (P_i,P_j)\subset (W)$. This is a contradiction, since $F\not \in (W)$. Therefore, we  must have that $\alpha_1=\alpha_i$ for all $1\leq i \leq (d+1)$. Therefore, we have $G-\alpha_1F=P_iA_i$ for all $1\leq i \leq (d+1)$. Now $\deg(G-\alpha_1F)=d$ and $P_i,P_j$ do not have any common factors for $i\neq j$, by \cref{proposition: auxiliary sets}. Therefore, we have that $\Pi_{i=1}^{d+1}P_i$ divides $G-\alpha_1F$, which is a contradiction. Hence we must have that $F\in (W)$ and hence $\cF_d\subset (W)$.
\end{proof}

\begin{lemma}\label{lemma: fraction of Fd}
   Let $\mu \in (0, 1)$ and $k\geq \lceil \frac{37d^3}{\mu}\rceil$. 
   Let $U\subset S_{\leq e}$ be a graded vector space with dimension sequence $\delta_U\in \N^e$. 
   Let $\wt{B}:\N^e\rightarrow \N^e$ be such that $\wt{B}\geq h_{2B}\circ t_1 $. Suppose $U$ is $h_{2\wt{B}}\circ t_{k}$-strong in $S$. Consider the graded unique factorization domain  $R=S/(U)$. Let $\cF$ be a $\rsgufd{d}{z}{R}$ configuration. Suppose that $2d^3<m_d$.  Then there exists a $\wt{B}$-lifted strong vector space $W\subset R_{\leq d}$ such that the following holds:
   
   \begin{enumerate}  
        \item $\dim(W_i)\leq C_{2\wt{B},i}(t_{k}(\delta_U))$
        \item $z\in W$
        \item $|\cF_d \cap \K[W]| > 2d^3$.
        \item $|\cF_d \cap (W)| \geq (1-\mu) \cdot m_d$.
        \item Furthermore, if $2d^3<m_d<\varepsilon m$ for some $\varepsilon\in (0,1)$ such that $\varepsilon <\frac{1}{d^32^{d^2+3}}$, then 
        $$|\cF_{<d} \cap (W)| \geq (1-\varepsilon d^32^{d^2+2}) \cdot m\geq \frac{m}{2}.$$ 
   \end{enumerate}
\end{lemma}

\begin{proof} 
    First we will show that we can construct a $\wt{B}$-lifted strong vector space $W\subset R$ satisfying properties (1)-(4) above. 
    Note that the vector space $(0)\subset R$ is $h_{2\wt{B}}\circ t_{k}$-lifted strong. 
    Let $r := 2d^3$ and $\nu := \dfrac{\mu}{2r}$. 
    Suppose $\cF_d$ is an $\dlsg{r}{\nu}$ configuration. 
    By \cref{prop:lincdeltasg},  
    $$\dim(\Kspan{\cF_d}) \leq 2d^3+1+ \frac{8}{\nu}=2d^3+1+\frac{32d^3}{\mu}< \frac{36d^3}{\mu}< k$$ 
    Therefore we may apply \cref{corollary: practical robustness } to the vector space $(0)\subset R$ and a basis of $\Kspan{\cF_d,z}$ to obtain a $\wt{B}$-lifted strong vector space $W\subset R_{\leq d}$ such that $z\in W$, $\cF_d\subset \K[W]$ and $ \dim(W_i)\leq C_{2\wt{B},i}(t_{k}(\delta_U))$. We also have $|\cF_d\cap \K[W]|=|\cF_d|>2d^3$.
    This proves (1)-(4) in this case.
    
    Hence we may assume that $\cF_d$ is not a $\dlsg{r}{\nu}$ configuration. Consider the set $$\cH=\{F\in \cF_d\mid |\Fspan(F)|< 2\nu m_d\}.$$
    If $|\cH|\leq r$, then $\cF$ is an  $\dlsg{r}{\nu}$ configuration, by \cref{linear sg simple}. 
    Therefore we must have that $|\cH|> r$. In particular, there exist $r+1$ elements in $\cF_d$, say $F_1, \ldots, F_{r+1}$, such that $|\Fspan(F_i)| < 2\nu m_d$ for all $i\in [r+1]$.
    Let $\cB := \cF_d \setminus \bigcup_{i=1}^{r+1} \Fspan(F_i)$. Suppose for some $G\in \cB$ and $i\in [r+1]$, we have that $z,F_i,G$ is a regular sequence and $(F_i,G)$ is a radical ideal. Now there exists $zH\in \cF$ such that $zH\in \rad(zF_i,zG)$. Since $z,F_i,G$ is a regular sequence by assumption, we have that $H\in \rad(F_i,G)=(F_i,G)$. Hence we have $H\in \cF_d$ as $F_i,G\in R_d$ and thus $H\in \Kspan{F_i,G}$. 
    This is a contradiction since $G\not \in \Fspan(F_i)$. 
    Therefore, for any $G \in \cB$ and $i \in [r+1]$, we must have that $z, F_i, G$ is not regular or $(F_i, G)$ not radical.

    Since $r+2=2d^3+2< k$, we may apply \cref{corollary: practical robustness } to the vector space $(0)\subset R$ and the set $z,F_1,\cdots, F_{r+1}$ to obtain a $\wt{B}$-lifted strong vector space $W\subset R_{\leq d}$ such that $z,F_1,\cdots,F_{r+1}\in \K[W]$ and $ \dim(W_i)\leq C_{2\wt{B},i}(t_{k}(\delta_U))$. 
    Note that $|\cF_d\cap \K[W]| \geq r+1 > 2d^3$. We will show that $\cB\subset (W)$. 
    
    Fix $G\in \cB$. Since $\wt{B}\geq h_{2B}\circ t_1$, we may apply \cref{corollary: practical robustness } to $W$ and $G$ to obtain a $B$-lifted strong vector space $V\subset R$ such that $G\in \K[V]$ and $W\subset V$.  Therefore we may apply \cref{corollary: best corollary discriminant} to $G$ and $W\subset V$ in $R$. Thus, if $G\not \in (W)$, then  $(G,F_i)$ is not radical for at most $d^2(2d-1)$ number of elements in $\K[W]$ by \cref{corollary: best corollary discriminant}. Since $r>d^2(2d-1)$, we must have that $z,G,F_i$ is not a regular sequence for some $i\in [r+1]$. This is a contradiction since $z,F_i,G$ is a regular sequence in $R$, by \cref{proposition: three forms regular sequence}. Therefore we must have that $G\in (W)$ and hence $\cB\subset (W)$. Now we have \[|\cB| \geq m_d - \sum_{i=1}^r |\Fspan(F_i)| \geq m_d - r \cdot 2\nu m_d = (1-\mu) \cdot m_d.\]
    
    Now we prove property (5). 
    Let $W\subset R$ be a $\wt{B}$-lifted strong vector space satisfying (1)-(4) above. 
    In particular, $|\cF_d\cap \K[W]| > 2d^3 = r$. 
    We may assume that $F_1,\cdots,F_r\in \cF_d\cap \K[W]$.  
    By \cref{lemma: Grad set small}, we know that $|\Grad(F)|\leq 2^{d^2}m_d$ for all $F\in \cF_d$. 
    Let $\cG=\cF_{<d}\setminus \bigcup_{j=1}^r\Grad(F_i)$. 
    Note that 
    \[|\cG|\geq m-m_d-r2^{d^2}m_d\geq m-(1+2d^32^{d^2})\varepsilon m\geq (1-\varepsilon d^32^{d^2+2}) \cdot m\geq \frac{m}{2},\]
    since $|\Grad(F)|\leq 2^{d^2}m_d$ and $m_d<\varepsilon m$ and $\varepsilon d^32^{d^2+2}< \frac{1}{2}$. 
    We will show that $\cG\subset (W)$. 
    Let $G\in \cB$. 
    Since $\wt{B}\geq h_{2B}\circ t_1$, we may apply \cref{corollary: practical robustness } to $W$ and $G$ to obtain a $B$-lifted strong vector space $V\subset R$ such that $G\in \K[V]$ and $W\subset V$.  
    Therefore we may apply \cref{corollary: best corollary discriminant} to $G$ and $W\subset V$ in $R$. 
    Thus, if $G\not \in (W)$, then  $(G,F_i)$ is not radical for at most $d^2(2d-1)$ number of elements in $\K[W]$ by \cref{corollary: best corollary discriminant}. 
    Since $r>d^2(2d-1)$, we must have that $z,G,F_i$ is not a regular sequence for some $i\in [r]$. 
    This is a contradiction since $z,F_i,G$ is a regular sequence in $R$, by \cref{proposition: three forms regular sequence}. 
    Therefore we must have that $G\in (W)$ and hence $\cG\subset (W)$. 
\end{proof}

\subsection{Proof of the main theorem}

We are now ready to prove our main technical theorem on radical Sylvester-Gallai configurations.

\subsubsection{Definitions of \texorpdfstring{$\Lambda_{d,e}$}{Lambdade},  \texorpdfstring{$\lambda_{d,e}$}{lambdade} and auxiliary functions}\label{subsection: definition of functions} 

We now introduce some auxiliary functions and notation that we will use in \cref{theorem: radical SG theorem over UFD}.

We fix positive integers $e\geq 2$ and $\eta\geq 3$. We let $d\in \N$ such that $1\leq d \leq e$. 
For all positive integers $d\leq e$, we will inductively define ascending functions $ \Lambda_d : \N^e \to \N^e$ and  $\lambda_d:\N^e\rightarrow \N$, which shall appear in the statement of \cref{theorem: radical SG theorem over UFD}. The functions $\Lambda_d,\lambda_e$ will depend on both $d$ and $e$. We will omit $e$ from the notaion for simplicity.
For any $n\in \N$, let $t_a:\N^e\rightarrow \N^e$ be the translation by $n$ function defined as $t_a(\delta)=(\delta_1+a,\cdots,\delta_e+a)$. 
We fix an ascending function $B: \N^e\rightarrow \N^e$ such that $B_i(\delta)\geq A(\eta,i)+3(\sum_i\delta_i-1)$ for all $\delta\in \N^e$ and $i\in [e]$, where $A(\eta,i)$ are the functions in \cref{theorem: Ananyan-Hochster}. 
Let $C_B:\N^e\rightarrow \N^e$ and $h_B:\N^e\rightarrow \N^e$ be the functions defined in \cref{proposition: constructing AH algebras}. \\

\emph{Potential function $\phi$.}
Let $U\subset S_{\leq e}$ be a graded vector space in $S$ such that $R=S/(U)$ is an UFD. 
Let $z\in R_1$ and $\cF$ be a $\rsgufd{d}{z}{R}$ configuration. We define the following potential function:

\[\phi(\cF)=\sum_{j=1}^dj\cdot|\cF_j|.\]

Note that for any $\rsgufd{d}{z}{R}$ configuration $\cF$ we have $0<\phi(\cF)\leq d|\cF|$. 
The potential function is a weighted measure of the size of the configuration $\cF$. 
In the proof of \cref{theorem: radical SG theorem over UFD}, we will show that the potential of a Sylvester-Gallai configuration drops under suitably chosen general quotients (defined in \cref{sec:general-quotients}). 
We will use this to reduce a $\rsgufd{d}{z}{R}$ configuration to a $\rsgufd{d-1}{y}{R'}$ configuration where $R'$ is a general quotient of $R$.\\

\emph{The parameters $\varepsilon_d$ and $k_d$.} 
Let $d\geq 2$. 
We choose and fix a rational number $\varepsilon_d\in (0,1)$ such that $\varepsilon_d< \frac{1}{d^32^{d^2+3}}$. 
Let $k_d=\left\lceil \frac{(5d)^3}{\varepsilon_d} \right\rceil $. 
Note that $\varepsilon_d$ and $k_d$ satisfy the numerical hypotheses of \cref{lemma: fraction of Fd} with $\mu= \frac{\varepsilon_d}{3}$ as $k_d\geq \lceil \frac{37d^3}{\mu}\rceil$. 
We have $k_d>6d^3$ by definition, and we note that $(1-8d(d+2)\varepsilon_d)\geq \frac{1}{4}$. 
We will use these facts in the proof of \cref{theorem: radical SG theorem over UFD}.\\

\emph{The functions $ \Lambda_d$ and $\Lambda^{(\ell)}_d$.} 
For all positive integers $ 2 \leq d \leq e$, we will inductively define ascending functions $ \Lambda^{(\ell)}_d: \N^e \to \N^e$ for $0\leq \ell\leq 8d+1$, and use the functions $\Lambda^{(\ell)}_d$ to define the ascending function $\Lambda_d : \N^e \to \N^e$.  
Recall that $B:\N^e\rightarrow \N^e$ is an ascending function such that $B_i(\delta)\geq A(\eta,i)+3(\sum_i\delta_i-1)$ for all $i\in [e]$, as defined above.
 
For $d=1$, we let $\Lambda_1=2B\circ t_2$. For $d\geq 2$, suppose that $\Lambda_{d-1}$ has been defined. 
We define $\Lambda_d$ and $\Lambda^{(\ell)}_d:\N^e\rightarrow \N^e$ for $0\leq \ell\leq 8d+1$ as follows. 
 
Let $\Lambda^{(8d+1)}_d=h_{2\Lambda_{d-1}}\circ t_{k_d}$. 
Suppose $\Lambda^{(\ell+1)}_d$ has been defined for some  $\ell \in [0, 8d]$. 
We inductively define $\Lambda^{(\ell)}_d=h_{2\Lambda^{(\ell+1)}_d}\circ t_{k_d}$. 
Given the definition of $\Lambda^{(\ell)}_d$ for all $\ell \in [0, 8d]$, we define $\Lambda_d=h_{2\Lambda^{(0)}_d}\circ t_{k_d}$. 
 
Recall that $h_G\geq G$ and $h_G\circ t_k\geq h_G$  for any function $G:\N^e\rightarrow\N^e$. 
Since $\Lambda_1=h_{2B}\circ t_2$, we have $\Lambda^{(\ell)}_d\geq h_{2B}\circ t_2$ for all $d\leq e$ and $\ell\in [0,8d+1]$ by induction. 
In the proof of \cref{theorem: radical SG theorem over UFD}, we will use the $\Lambda^{(\ell)}_d$ functions to lower bound the strengths of iteratively constructed vector spaces.\\

\emph{The functions $\Gamma^{(\ell)}_d$ and $\Gamma_{d-1}$.} 
For all $0\leq \ell \leq 8d+1$, we will define the functions $\Gamma^{(\ell)}_d:\N\rightarrow \N$ and $\Gamma_{d-1}:\N\rightarrow \N$ inductively as follows:

\[\Gamma^{(\ell)}_d(n) = 
\max\left\{\sum_{i=1}^e C_{2\Lambda^{(\ell)}_{d},i}(t_{k_d}(\delta)) \ \ | \ \  \delta\in \N^e \text{ with } \sum_{i=1}^e\delta_i\leq n\right\}\]
\[\Gamma_{d-1}(n)=
\max\left\{\sum_{i=1}^e C_{2\Lambda_{d-1},i}(t_{k_d}(\delta)) \ \ | \ \ \delta\in \N^e \text{ with } \sum_{i=1}^e\delta_i\leq n \right\}.\]

Note that $\Gamma^{(\ell)}_d$ and $\Gamma_{d-1}$ are non-decreasing functions. 
In the proof of \cref{theorem: radical SG theorem over UFD}, we will use the functions $\Gamma^{(\ell)}_d$ and $\Gamma_{d-1}$ to bound the dimensions of iteratively constructed vector spaces. \\

\emph{The functions $\lambda_d$.} 
Given $n\in \N$, let $n_0=n+\Gamma^{(0)}_d(n)$ and inductively define $$n_k=\Gamma^{(k)}_d \left(n+\sum_{j=0}^{k-1}n_k \right)$$ 
for all $k\leq 8d+1$ and finally let $n'=n+\Gamma_{d-1}(n+\sum_{k=0}^{8d+1}n_k)$. 
For all $d\geq 1$ , we define a function $\lambda_d:\N\rightarrow \N$ as follows.

For $d=1$, let $\lambda_1=26$.
For $d\geq 2$, let $D(n)=\lambda_{d-1}(n+n_0+\cdots+n_{\ell-1}+n')$ and
\[\lambda_d(n)= (d+3)^{6(8d+2)+3(\sum_{k= 0}^{8d+1}n_k+n')}\cdot e^{((8d+2)n+\sum_{k=0}^{8d}(8d+1-k)n_k+n')}\cdot D(n).\]  

In \cref{theorem: radical SG theorem over UFD}, we show that the functions $\lambda_d$ provide upper bounds on the dimension of the span of any radical Sylvester-Gallai configuration of degree $d$.

\begin{remark}
    Note that the functions $\Lambda_d, \Lambda_d^{(\ell)}, \Gamma_d^{(\ell)},\Gamma_{d-1}, \lambda_d,\lambda_d^{(\ell)}$ defined above depend on the integer $e$ that we fixed in the beginning of this section. We omit $e$ from the notation for simplicity. In the proof of \cref{theorem: radical SG theorem over UFD} below, we will fix the integer $e$ and consider these functions for $d\in[1,e]$. We also note that all these functions defined above are ascending functions.
\end{remark}
\subsubsection{Main theorems}\label{subsection: final proof}

We begin with our main technical theorem, which proves bounds on radical Sylvester-Gallai configurations over quotients of the polynomial ring by sufficiently strong ideals.
We then show how the radical Sylvester-Gallai theorem over polynomial rings follows from our technical theorem.

\generalRadicalSG*

\begin{proof}
Fix an integer $e\geq 2$. 
For all positive integers $d\leq e$, let $\Lambda_d$, $\Lambda^{(\ell)}_d$, $\lambda_d$, $\varepsilon_d$ and $k_d$ be defined as in \cref{subsection: definition of functions}. We let $\Lambda_{d,e}:=\Lambda_d$ and $\lambda_{d,e}:=\lambda_d$. We will show that these functions satisfy the conclusion.

Suppose $d=1$. Since $\Lambda_1=h_{2B}\circ t_2$, we have that $\dim \Kspan{\cF} \leq \lambda_1(\dim(U))=26$ by \cref{proposition: linear case main theorem}. If $d\geq 2$, we will construct a sequence of general quotients to eventually reduce $\cF$ to a $(d-1)$-radical Sylvester-Gallai configuration and then apply \cref{corollary: lifting from composition of quotients} together with the inductive bound given by $\lambda_{d-1}$. 
    
From now on we assume that $d\geq 2$.
For $0\leq \ell \leq 8d+1$, we define a graded $\K$-algebra $R^{(\ell)}$, a $\Lambda^{(\ell)}_d$-lifted strong vector space $W^{(\ell)}\subset R^{(\ell)}$ and a $\rsgufd{d}{z_\ell}{R^{(\ell)}}$ configuration $\cF^{(\ell)}$ inductively as follows.

\begin{enumerate}
    \item Set $R^{(0)}=R$, $U^{(0)}=U\subset S$, $z_0=z$ and $\cF^{(0)}=\cF$. 
    \item Suppose that for some $\ell\in [0,8d]$, we have defined a quotient ring $R^{(\ell)}$, a $\rsgufd{d}{z_\ell}{R^{(\ell)}}$ configuration $\cF^{(\ell)}$ with $z_\ell \in R^{(\ell)}_1$ and $U^{(\ell-1)}\subset S[z_1,\cdots,z_\ell]$. 
    
    If $m^{(\ell)}_d:=|\cF^{(\ell)}_d|>6d^3$ :
    \begin{itemize}
        \item [2.1.] Apply \cref{lemma: fraction of Fd}  with $\mu=\frac{\varepsilon_d}{3}$ and the constant $k=k_d$, to obtain a $\Lambda^{(\ell)}_d$-lifted strong vector space $W^{(\ell)}\subset R^{(\ell)}$ such that $z_\ell\in W^{(\ell)}$and  $|\cF^{(\ell)}_d \cap (W^{(\ell)})| \geq (1-\frac{\varepsilon_{\ell+1}}{3}) \cdot m^{(\ell)}_d$.
        \item [2.2.] Let $z_{\ell+1}$ be a new variable. We let  $R^{(\ell+1)}:=\frac{R^{(\ell)}[z_{\ell+1}]}{(W^{(\ell)}_\alpha)}$ and $\varphi_\alpha:R^{(\ell)}[z_{\ell+1}]\rightarrow R^{(\ell+1)}$ be the quotient defined by a general $\alpha^{(\ell)}\in \K^{\dim(W^{(\ell)})}$for which the conclusions of \cref{proposition: general quotient preserves SG} hold.
        \item [2.3.] Let $\cF^{(\ell+1)}=(\varphi_\alpha(\cF^{(\ell)}))_\red$ as defined in \cref{subsecttion: SG and quotients}.
        \item [2.4.] 
        Let $U^{(\ell+1)}\subset S[z_1,\cdots,z_{\ell+1}]$ be the vector space given by \cref{corollary: lifting from composition of quotients}, such that we have $R^{(\ell+1)}=S[z_1,\cdots,z_k]/(U^{(\ell)})$.
    \end{itemize}
\end{enumerate}

Note that $\Lambda^{(\ell-1)}_d\geq h_{2\Lambda^{(\ell)}_{d}}\circ t_{k_d}$ and $\Lambda^{(\ell)}_{d}\geq h_{2B}\circ t_1$. 
Therefore, \cref{lemma: fraction of Fd} applies in Step 2.1. and we get that $W^{(\ell)}$ is $\Lambda^{(\ell)}_d$-lifted strong with respect to $U^{(\ell)}$. 
By \cref{proposition: strong with z}, we know that the vector space $W^{(\ell)}_\alpha$ is also $\Lambda^{(\ell)}_{d}$-lifted strong in $R[z_{\ell+1}]$ for all $\alpha\in\K^{\dim(W^{(\ell)})}$. 
Note that $\Lambda^{(\ell)}_{d}>h_B\circ t_2$. 
Therefore, by \cref{proposition: general quotient preserves SG}, we conclude that $\cF^{(\ell+1)}=(\varphi_\alpha(\cF^{(\ell)}))_\red$ is a $\rsgufd{d}{z_{\ell+1}}{R^{(\ell+1)}}$ configuration for a general $\alpha\in \K^{\dim(W^{(\ell)})}$. 
Note that the definitions of $R^{(\ell)}$ and $\cF^{(\ell)}$ depend on the choice of the general $\alpha^{(\ell)}\in \K^{\dim(W^{(\ell)})}$. 
We omit this from the notation for simplicity. 

Now we will show that the size of $\cF^{(\ell)}_d$ significantly reduces every time we increase $\ell$. 
In particular we have the following.

\textbf{Claim:} $m^{(\ell)}_d\leq 6d^3$ for some $\ell\leq 8d+1$.

\begin{proof}[Proof of Claim]
Note that $m^{(\ell)}_d\geq m^{(\ell+1)}_d$ for all $\ell$. 
Suppose that $m^{(\ell)}_d> 6d^3$ for all $0\leq \ell \leq 8d+1$. 
If $|\cF^{(\ell)}_d \cap \K[W^{(\ell)}]|\geq \frac{2m^{(\ell)}_d}{3}$ for some $\ell\leq 8d$, then by \cref{lemma: algebra contains many then ideal contains Fd}, we have that $\cF^{(\ell)}_d\subset (W^{(\ell)})$. 
Therefore, by \cref{proposition: general quotient preserves SG}, we see that $\cF^{(\ell+1)}$ is a $\rsgufd{d-1}{z_{\ell+1}}{R^{(\ell+1)}}$ configuration. 
In particular, $\cF^{(\ell+1)}_d=\emptyset$ and $m^{(\ell+1)}_d=0$. 
Therefore we may assume that $|\cF^{(\ell)}_d \cap \K[W^{(\ell)}]|< \frac{2m^{(\ell)}_d}{3}$ for all $0 \leq \ell \leq 8d$.
    
Recall that $|\cF^{(\ell)}_d \cap (W^{(\ell)})| \geq (1-\frac{\varepsilon_\ell}{3}) \cdot m^{(\ell)}_d$ by the construction in step 2.1. 
Since all elements of $\cF^{(\ell)}_d \cap (W^{(\ell)})$ will be of lower degree after a general quotient and reduction, we have that $m^{(\ell+1)}_d\leq  \frac{\varepsilon_{\ell+1}}{3} m^{(\ell)}_d$. 
Thus:
\[m^{(\ell+1)}:=|\cF^{(\ell+1)}|=|\varphi_\alpha(\cF^{(\ell)})_\red|\geq |\cF^{(\ell)}_d|-|\cF^{(\ell)}_d\cap \K[W^{(\ell)}]|\geq \frac{m^{(\ell)}_d}{3}\] 
as $|\cF^{(\ell)}_d \cap \K[W^{(\ell)}]|< \frac{2m^{(\ell)}_d}{3}$. 
Therefore, using the inequalities above, we obtain
\[m^{(\ell+1)}_d\leq \frac{\varepsilon_{\ell+1}}{3} m^{(\ell)}_d \leq \varepsilon_{\ell+1} m^{(\ell+1)}\]
for all $0\leq \ell\leq 8d$. Therefore, by \cref{lemma: fraction of Fd}, we know that 
\[|\cF^{(\ell+1)}_{<d} \cap (W^{(\ell+1)})| \geq  \frac{m^{(\ell+1)}}{2}\]
for all $0\leq \ell \leq 8d$, and hence
\[|\cF^{(\ell)}_{<d} \cap (W^{(\ell)})| \geq  \frac{m^{(\ell)}}{2}\] for all $1\leq \ell \leq 8d$. 
Therefore we have 
\[\phi(\cF^{(\ell)})-\phi(\cF^{(\ell+1)})\geq \frac{m^{(\ell)}}{2}\]
for all $1\leq \ell\leq 8d$. 
Note that $\phi(\cF^{(1)})\leq \phi(\cF^{(0)})\leq dm^{(0)}=dm$. 
If $m^{(\ell)}\geq \frac{m}{4}$ for all $1\leq \ell \leq 8d$, then we have that $\phi(\cF^{(8d+1)})=0$, which is a contradiction. 
Hence, it is enough to prove that $m^{(\ell)}\geq \frac{m}{4}$ for all $1\leq \ell \leq 8d$.

For all $1\leq \ell \leq 8d$, we have
\[(d+2)\varepsilon_d m^{(\ell)}\geq (d+2)m^{(\ell)}_d> (d+1)m^{(\ell)}_d+d^2(2d-1)\]
since $ m^{(\ell)}_d\leq \varepsilon_d m^{(\ell)} $ and $m^{(\ell)}_d> 2d^3$. 
If for some $1\leq \ell \leq 8d$, we have $|\cF^{(\ell)}\cap \K[W]|>(d+2)\varepsilon_d m^{(\ell)}$, then $\cF^{(\ell)}_d\subset (W^{(\ell)})$ by \cref{lemma: algebra contains many then ideal contains Fd}. 
Thus we have $m^{(\ell+1)}_d=0$ and $\ell+1\leq 8d+1$, which is a contradiction. 
Therefore we must have that $|\cF^{(\ell)}\cap \K[W^{(\ell)}]|\leq (d+2)\varepsilon_d  m^{(\ell)}$ for all $1\leq \ell \leq 8d$.  
Letting $\gamma=(d+2)\varepsilon_d $, we have 
\[m^{(\ell+1)}\geq (1-\gamma)m^{(\ell)}\geq (1-\gamma)^\ell m^{(0)}\geq (1-\gamma)^{8d} m \geq (1-8d\gamma) m\geq \frac{m}{4}\]
for all $1\leq \ell \leq 8d$, where the last inequality holds since $(1-8d(d+2)\varepsilon_d)\geq \frac{1}{4}$. 
This completes the proof of the claim.
\end{proof}

Thus we may henceforth assume $m^{(\ell)}_d\leq 6d^3$ for some $ \ell\leq 8d+1$. 
Let $\ell$ be the smallest such number. 
Note that $W^{(\ell-1)}\subset R^{(\ell-1)}$ is $\Lambda^{(\ell-1)}_d$-lifted strong by construction and $\Lambda^{(\ell-1)}_d\geq h_{2\Lambda_{d-1}}\circ t_{k_d}$. 
Therefore the vector space $(0)$ is $h_{2\Lambda_{d-1}}\circ t_{k_d}$-lifted strong in $R^{(\ell)}$. 
As $6d^3<k_\ell$, we may apply \cref{corollary: practical robustness } to obtain a $\Lambda_{d-1}$-lifted strong vector space $W\subset R^{(\ell)}_{\leq e}$ such that $z_\ell,\cF^{(\ell)}_d\subset \K[W]$ and $\dim(W_i)\leq C_{2\Lambda_{d-1},i}(t_{k_\ell}(\delta_U+\delta_{W^{(\ell)}}))$ for all $i\in [e]$. 

Let $y$ be a new variable. By \cref{proposition: strong with z}, we know that the vector space $W_\alpha$ is $\Lambda_{d-1}$-lifted strong in $R^{(\ell)}[y]$ for all $\alpha\in\K^{\dim(W^{(\ell)})}$. 
Note that $\Lambda_{d-1}>h_B\circ t_2$. 
Let $R'=R^{(\ell)}[y]/(W_\alpha)$. 
By \cref{proposition: general quotient preserves SG}, we conclude that $\varphi_\alpha(\cF^{(\ell)})_\red$ is a $\rsgufd{d-1}{y}{R'}$ configuration for a general $\alpha \in \K^{\dim(W)}$.

By construction, we have $R^{(k)}=S[z_1,\cdots,z_k]/(U^{(k-1)})$ for all $1 \leq k \leq \ell$, where $U^{(k-1)}$ is given by \cref{corollary: lifting from composition of quotients}. 
By \cref{lemma: fraction of Fd} we have, $\dim(W^{(k)}_i)\leq C_{2\Lambda^{(\ell)}_{d},i}(t_{k_d}(\delta_{U^{(k-1)}}))$.  
Therefore 
$$\dim(W^{(k)})=\sum_{i=1}^e\dim(W^{(k)}_i)\leq \sum_{i=1}^e C_{2\Lambda^{(k)}_{d},i}(t_{k_d}(\delta_{U^{(k-1)}}))\leq \Gamma^{(k)}_d(\dim(U^{(k-1)})).$$
By \cref{corollary: lifting from composition of quotients}, we also have $\dim(U^{(k-1)})=\dim(U)+\sum_{j=0}^{k-1}\dim(W^{(j)})$. 
Let $\dim(U)=n$, $n_0=n+\Gamma^{(0)}_d(n)$ and define $n_k=n+\Gamma^{(k)}_d(n+\sum_{j=0}^{k-1}n_j)$. 
By induction and using the inequality above, we have $\dim(W^{(k)})\leq n_k$ for all $0\leq k \leq \ell$. 
Similarly, we may write $R'=S[z_1,\cdots,z_\ell,y]/(U')$ for some $U'$. 
Then we have 
$$\dim(W)\leq n':= n+\Gamma_{d-1}(n+\sum_{j=0}^\ell n_j) \ \text{ and } \ \dim(U')=\dim(U)+\sum_{j=0}^{\ell-1}\dim(W^{(j)})+\dim(W).$$

We have shown above that for all $\alpha^{(0)},\cdots, \alpha^{(0)}, \alpha$ such that the conclusions of \cref{proposition: general quotient preserves SG} hold, we get that $\varphi_\alpha(\cF^{(\ell)})_\red$ is a $\rsgufd{d-1}{y}{R'}$ configuration. 
Since $W_\alpha\subset R^{(\ell)}$ is $\Lambda_{d-1}$-lifted strong, we know that $U'\subset S[z_1,\cdots,z_\ell,y]$ is $\Lambda_{d-1}$-strong in $S[z_1,\cdots,z_\ell,y]$. 
Therefore by our induction hypothesis, we may apply the theorem to $R'=S[z_1,\cdots,z_\ell,y]/U'$ and the configuration $\varphi_\alpha(\cF^{(\ell)})_\red$ and obtain 
\[\dim(\Kspan{\varphi_\alpha(\cF^{(\ell)})_\red})\leq \lambda_{d-1}(\dim(U'))\leq \lambda_{d-1}(n+n_0+\cdots+n_{\ell-1}+n').\]
Setting $D := \lambda_{d-1}(n+n_0+\cdots+n_{\ell-1}+n')$, \cref{corollary: lifting from composition of quotients}, implies that 
\[\dim(\Kspan{\cF}) \leq (d+3)^{6(\ell+1)+3(\sum_{k= 0}^{\ell-1}n_k+n')}\cdot e^{((\ell+1)n+\sum_{k=0}^{\ell-2}(\ell-k)n_k+n')}\cdot D.\]
Hence, by definition of $\lambda_d$, we conclude that
$\dim(\Kspan{\cF})\leq \lambda_d(n)=\lambda_d(\dim(U)).$
\end{proof}

As a corollary of our technical theorem above, we obtain the radical Sylvester-Gallai theorem below, proving \cite[Conjecture 2]{G14}.

\radSGmain*

\begin{proof}
    We define $\lambda(d)=\lambda_d(0)$, where $\lambda_d:\N\rightarrow \N$ is the function given by \cref{theorem: radical SG theorem over UFD}. If $\cF := \{F_1, \ldots, F_m\}$, where $F_i \in S$, and $z$ is a new variable (free over $S$), \cref{proposition: vanilla radical SG reduces to general radical SG} implies that the set $\cG := \{zF_1, \ldots, zF_m\}$ is a $\rsgufd{d}{z}{S[z]}$ configuration, with $U = 0$.
    Since the zero vector space is arbitrarily strong, we have that $\cG$ satisfies the conditions of \cref{theorem: radical SG theorem over UFD}, with the vector space $U=(0)\subset S_{\leq d}$. Hence we have that $\dim \Kspan{\cG} \leq \lambda(d)$.
    As $\dim \Kspan{\cF} = \dim \Kspan{\cG}$, we are done.
\end{proof}

Using the Stillman uniformity principle, we obtain the following direct consequence of \cref{theorem: sylvester-gallai}.

\stillmanSG*

\begin{proof}
Note that by \cref{theorem: sylvester-gallai}, we may assume that $I$ is generated by at most $\lambda(d)$ forms of degree at most $d$. Therefore, the Stillman uniformity type results \cite[Theorem 1.1, Proposition 4.6]{ESS21} imply that each of the ideal invariants of $I$ are uniformly bounded by functions of $d$ only, independent of $\K$ and $N$.
\end{proof}

\bibliographystyle{alpha}
\bibliography{main}

\appendix

\section{Auxiliary algebraic results}\label{appendix: vector SG}

We note a few elementary results below which will be used in the previous section of the paper.


We note the relationship between transcendence degree and vector space dimension.

\begin{proposition}
    Let $\cF\subset S=\K[x_1,\cdots,x_N]$ be a finite set of forms. Then $\mathrm{trdeg}_{\K}(\cF)\leq \dim(\Kspan{\cF})$.
\end{proposition}

\begin{proposition}\label{proposition: homogeneous factors}
Let $R$ be a graded integral domain. Let $F\in R$ be a homogeneous element. If $g\in R$ and $g$ divides $F$, then $g$ is homogeneous. In particular, if $R$ is a UFD, then all irreducible factors of $F$ are homogeneous.

\end{proposition}

\begin{proof}
Suppose $F=gh$ for some $h\in R$. Let $g=g_d+g_{d-1}+\cdots+g_r$ and $h=h_e+h_{e-1}+\cdots+h_s$ be the decomposition into homogeneous parts of $g$ and $h$, where $g_d,g_r,h_e,h_s\neq 0$, and $g_r,h_s$ are the least degree non-zero homogeneous parts of $g$ and $h$ respectively. Suppose that $g$ is not homogeneous, i.e. $d>r$. Since $F$ is homogeneous, we must have $g_dh_e=0$ or $g_rh_s=0$. This is a contradiction as $R$ is a domain.
\end{proof}

\begin{proposition}\label{proposition: homogeneous compoments}
Let $R$ be a finitely generated $\N$-graded $\K$-algebra with $R_0=\K$. Let $A\subset R$ be a finitely generated graded $\K$-subalgebra. Suppose $A\subset R$ is a module-finite free extension of rank $r$ and $R$ is generated by homogeneous elements $u_1,\cdots,u_r\in R$ as a free $A$-module. Consider the $A$-module isomorphism $\phi:R\xrightarrow[]{\sim}A\cdot e_1\oplus \cdots \oplus A\cdot e_r$ given by $\phi(u_i)=e_i$. Let $F\in R_d$ be a homogeneous element of degree $d$. Let $\phi(F)=(f_1,\cdots,f_r)\in A^{\oplus r}$. Then $f_i$ are homogeneous and $\deg(f_i)\leq d$ in $A$ and $R$.
\end{proposition}

\begin{proof}
Suppose $\deg(u_j)=s_j$ and $d_j=d-s_j$. Let $f_{id_j}$ be the degree $d_j$ homogeneous part of $f_i$ in $A$ and $f_{i,\neq d_j}=f_i-f_{id_j}$. Now we have $F=f_1u_1+\cdots+f_ru_r$ in $R$. Since $F$ is homogeneous of degree $d$, we must have $F=f_{1d_1}u_1+\cdots+f_{rd_r}u_r$ and $f_{1,\neq d_1}u_1+\cdots+f_{r,\neq d_r}u_r=0$. Since $u_1,\cdots,u_r$ generate $R$ as a free $A$-module, we must have that $f_{i,\neq d_j}=0$ for all $j\in [r]$. In particular, $f_i=f_{i,d_j}$ is homogeneous of degree $d_j\leq d$.
\end{proof}

\begin{lemma}\label{lemma: general vandermonde invertible}
Let $\cM := \{M_i|i\in [r]\}$ be a set of monomials in $S$ of degree at most $d$. 
Let $x^{(j)}=(x^{(j)}_1,\cdots,x^{(j)}_n)$ be new variables for $j\in [r]$. 
If $M$ is the $r\times r $ matrix with $(i,j)$-th entry given by $M_{ij}=M_i(x^{(j)})$, then $\det(M)\neq 0$.
\end{lemma}

\begin{proof}
    Let $M_i = \bx^{\bu_i}$, where $\bu_i \in \N^n$ is the exponent vector of $M_i$.
    Since $\deg(M_i) \leq d$, we have $\sum_{k=1}^n u_{ik} \leq d$.
    Let $D = 2d + 1$ and consider the vector $\be := (1, D, D^2, \ldots, D^{n-1})$.
    Then the inner products $\mu_i := \langle \be, \bu_i \rangle$ are such that 
    $$\mu_i - \mu_j = \sum_{k=1}^n (u_{ik} - u_{jk}) D^{k-1} \neq 0 $$
    whenever $i \neq j$.
    Hence, after relabeling, we can assume that $\mu_1 > \mu_2 > \cdots > \mu_r$.
    In particular, we can write $\mu_i = \lambda_i + (r-i)$ for some decreasing sequence of integers $\lambda_1 \geq \lambda_2 \geq \cdots \geq \lambda_r$.

    Thus, if we consider the natural ring homomorphism given by $\varphi(x^{(j)}_k) = y_j^{D^{k-1}}$, where $y_1, \ldots, y_r$ are new variables, we obtain that $\varphi(M)$ is a generalized Vandermonde matrix corresponding to the partition $\lambda:= (\lambda_1, \ldots, \lambda_r)$, and therefore we have that 
    $$\varphi(\det(M)) = \det(\varphi(M)) = s_{\lambda}(y_1, \ldots, y_r) \cdot \prod_{i < j} (y_i - y_j)$$
    which implies $\det(M) \neq 0$.
\end{proof}

\end{document}